\documentclass[11pt]{amsart}
\usepackage{tikz}
\usetikzlibrary{arrows.meta}
\usepackage[utf8]{inputenc}
\usepackage[T1]{fontenc}
\usepackage{geometry}
\usepackage{amsmath}
\usepackage{amsthm}
\usepackage{amssymb}
\usepackage{mathtools}
\usepackage{enumitem}
\usepackage{xcolor}
\usepackage[unicode=true,pdfusetitle,
 bookmarks=true,bookmarksopen=false,
 breaklinks=false,pdfborder={0 0 1},backref=false,colorlinks=false]
 {hyperref}

\numberwithin{equation}{section}
\numberwithin{figure}{section}

\theoremstyle{plain}
\newtheorem{thm}{Theorem}[section]
\newtheorem{prop}[thm]{Proposition}
\newtheorem{cor}[thm]{Corollary}
\newtheorem{lem}[thm]{Lemma}
\theoremstyle{definition}
\newtheorem{defn}[thm]{Definition}
\theoremstyle{remark}
\newtheorem{rem}[thm]{Remark}
\theoremstyle{plain}

\newcommand{\set}[1]{\left\{\, #1 \,\right\}}
\newcommand{\abs}[1]{\lvert #1 \rvert}
\newcommand{\norm}[1]{\lVert #1 \rVert}
\newcommand{\vep}{\varepsilon}
\newcommand{\supp}{\operatorname{spt}}
\def\R{\mathbb{R}}
\def\e{\varepsilon}
\title{Space-time geometry of Hele-Shaw flow}
\author{Carson Collins}
\address{Department of Mathematics, University of California, Los Angeles, CA 90095, USA}
\email{ctcollins@math.ucla.edu}
\author{Inwon Kim}
\address{Department of Mathematics, University of California, Los Angeles, CA 90095, USA}
\email{ikim@math.ucla.edu}
\author{Sebastian Munoz}
\address{Department of Mathematics, University of California, Los Angeles, CA 90095, USA}
\email{sebastian@math.ucla.edu}
\date{}

\begin{document}

\begin{abstract}
We study the dynamic behavior of the free boundary of the Hele-Shaw flow with a nonnegative source, in
which the pressure satisfies $-\Delta p=c$ in $\{p>0\}$ and the free boundary moves with
normal velocity $V=\abs{\nabla p}$. For initial domains satisfying an interior ball
condition, we prove that the patch $\{p>0\}$ expands at a locally positive rate, so that its
hitting time is locally Lipschitz; this is optimal, since distinct portions of the patch may
collide. We also show that the free boundary is locally either regular, or of collision type, or undergoing parabolic-scale hole closings of various dimensions: such a space-time characterization of singularities is new in Hele-Shaw flow.
We further show that, near any
time at which no collision occurs, the space-time free boundary is a $C^1$ hypersurface. This implies that at such times every free boundary point has a space-time normal: the space-time unit normal is parallel to the time direction precisely where the spatial free boundary has singular geometry.
The results apply in particular to the tumor growth model with nutrients introduced in \cite{PQV14},
but the analysis of singularities is new even for the classical injection problem.
\end{abstract}

\maketitle
\tableofcontents

\section{Introduction and main results}\label{sec:introduction}

In this paper we study the Hele-Shaw flow with a source term, which can be written in terms of the nonnegative pressure variable $p(x,t)$ as
\begin{equation*}\tag{HS}\label{eq:intro-HS-law}
\left\{\begin{array}{lll}
    -\Delta p=c &\hbox{ in }&\{p>0\},\\
    V=|\nabla p| &\hbox{on }&\partial\{p>0\},
\end{array}\right.
\end{equation*}
where $c=c(x,t) \geq 0$ is a given source term and $V=V_{x,t}$ denotes the outer normal velocity of the {\it free boundary } $\partial\{p>0\}$; the second equation is Darcy's law. We are interested  in understanding the effect of the source term on the evolution of the free boundary.

 The classical Hele-Shaw free boundary problem arose in fluid dynamics, especially in injection or Laplacian-growth settings where the forcing is prescribed on a fixed boundary or at injection sites \cite{SaffmanTaylor58,Richardson72,BCMP73}. More recently, variants of \eqref{eq:intro-HS-law}  have been actively studied in the context of crowd motion and tumor growth. To motivate the presence of the source term, it is helpful to view \eqref{eq:intro-HS-law}
as a constrained transport equation:
\begin{equation}\label{eqn:density}
\partial_t\rho - \nabla\cdot (\rho \nabla p) = c\rho,  \quad 0\leq \rho \leq 1.
\end{equation}

Here $p\geq 0$ is the pressure variable generated by the constraint on $\rho$: when $c\geq 0$, $p$ satisfies $p(1-\rho)=0$.  
When the initial density $\rho_0:=\rho(\cdot,0)$ is a {\it patch}, namely when $\rho_0=\chi_{\Omega_0^+}$ a.e., where $\Omega_0^+:=\{\rho_0=1\}$, it is known that $\rho(\cdot,t)$ remains a patch \cite[Theorem~1.2(c')]{JKT21}. For patch solutions, \eqref{eq:intro-HS-law} is the pressure formulation of \eqref{eqn:density}, and $\{p>0\}$ coincides with the saturated region $\{\rho=1\}$. 
For solutions of \eqref{eqn:density} with non-patch initial data, by contrast, the source can raise the density to saturation in a region where $p=0$, thereby allowing a new component of $\{p>0\}$ to nucleate away from the existing pressure phase (see e.g. \cite{MPQ17}). We do not address this mechanism in this paper and thus focus on \eqref{eq:intro-HS-law}.

While we consider a general class of $c$ in this paper, our primary motivating example is  the tumor growth model \cite{PQV14,DavidPerthame,JKT23,CJK25}. Here the source $c$ denotes the nutrient in the environment that the tumor consumes. In this case $c$ is coupled with $\rho$ and satisfies
\begin{equation}\label{eqn:nutrient}
\partial_t c - \Delta c = -c\rho.
\end{equation}
The model \eqref{eqn:density}--\eqref{eqn:nutrient} is well known to generate diverse patterns of irregular tumor growth, including the growth of thin fingers, patterns very different from those of the injection problem \cite{CLN03,MRCSdendritic,FTXZ23}.
The dynamics behind these numerically observed irregular shapes remain far from clear, even at a formal level. Thus, studying the dynamic profile of the set $\{p>0\}$ in terms of the source $c$  carries both mathematical and practical interest.

  While well-posedness of weak solutions holds for a broad class of source terms $c$ \cite{PQV14,MPQ17,JKT21,GKM22,DavidPerthame}, classical solutions of \eqref{eq:intro-HS-law}, in general, only exist in short-time regimes \cite{EscherSimonett97,PruSim16}. Extensive literature exists on short-time free-boundary regularity and on regularization from flat or Lipschitz profiles \cite{Kim06,CJK07,CJK09,ChangLaraGuillen16,kimZhang24,Liang26}, as well as on explicit examples of finite-time singularities in zero-surface-tension Hele-Shaw flow and Laplacian growth \cite{Howison86,LeeTeodorescuWiegmann09}. There are also comprehensive studies of the stationary profiles of singularities at a fixed time, via the elliptic obstacle problem \cite{Caffarelli98,Monneau03,FS18,FRS20,EFW25}.  By contrast, the dynamic characterization of singularities, our focus in this paper, is much less understood.

Below we highlight our key findings on the space-time geometry of the flow. For simplicity, we state them for initial domains $\{\rho_0=1\}$ satisfying an interior ball condition. Here {\it singular} points refer to the free boundary points near which the set $\{p>0\}$ does not approximate a half-space.
  \begin{enumerate}[label=(\alph*)]
      \item\label{i:result-nondegeneracy} {\it Nondegeneracy:} the patch expands at a locally positive rate, depending on the geometry of $\Omega_0^+$ and on $c$. In particular, the pressure stays uniformly nondegenerate at the interface; see Theorem~\ref{thm:HS-hitting-time-lipschitz}.
      \item\label{i:result-space-time-profile} {\it Space-time profile of singularities:} at every singular  point that is not of {\it collision type} (at such a point the zero set collapses, at small scales, onto a hyperplane, as when two advancing fronts of the patch collide), we establish the sharp extinction rate of the zero set, together with its asymptotic shape, a cylinder over an ellipsoid, both of which depend only on the blowup profile at the singular point. Thus, all singularities are essentially either collisions of free boundaries or hole closings of various dimensions; see Theorem~\ref{thm:closing-rates}.
      \item\label{i:result-space-time-regularity} {\it Regularity of the space-time boundary:} near any time at which no collision-type singularity occurs, the space-time free boundary is a $C^1$ hypersurface. It follows that  at these times its singular points are precisely the points where its space-time normal is parallel to the time direction, namely where the patch expands with infinite speed; see Theorem~\ref{thm:no_top_stratum_C1}.
 \end{enumerate}

For the tumor growth model \eqref{eqn:density}--\eqref{eqn:nutrient}, these results resolve questions on the singular behavior of the patch raised in
\cite{CJK25}, which we recall below alongside the corresponding theorems. On the other hand, our results are new contributions in the broader context of Hele-Shaw type flows. In particular, results~\ref{i:result-space-time-profile} and
\ref{i:result-space-time-regularity} are new even for the classical injection problem. Alongside them, and for general initial domains, our analysis yields a Hopf--Lax type inequality for the pressure variable (Theorem~\ref{prop:HS-pressure-hopf-lax}), of independent interest, which serves as a key tool in the proof of result~\ref{i:result-space-time-profile}.

We note that the patch expands over time due to its velocity law in \eqref{eq:intro-HS-law}, and thus the space-time free boundary is the graph of the {\it hitting time} $T(x):\R^d\to\R\cup\{+\infty\}$, which can be formally defined as the first time at which the patch contains $x$ (see \eqref{eq:intro-w-obstacle}). Our regularity analysis can be stated in terms of $T$, since formally $\{T(x)=t\}=\partial\Omega_t$ and thus the normal velocity of $\Omega_t$ is $V=|\nabla T|^{-1}$. Result~\ref{i:result-nondegeneracy} corresponds to Lipschitz regularity of $T$, which is optimal due to potential topological changes of the patch. Result~\ref{i:result-space-time-regularity} states that $T$ is $C^1$ near times without collision-type singularities, with singular set given exactly by the critical set $\{\nabla T=0\}$.

We stress that our focus is on finite-time singularities of the free boundary. The main challenge in studying the large-time behavior of the free boundary lies in the potential degeneration of the source term $c$, which occurs in the tumor growth model \eqref{eqn:density}--\eqref{eqn:nutrient} as the tumor consumes the nutrient. Global problems, such as the characterization of the limit shape of the tumor patch as $t\to\infty$, remain intriguing open questions; we make some initial observations on this matter in Section~\ref{sec:tumor-applications}.

\subsection{Setting and standing assumptions}\label{subsec:intro-setting}

Throughout the paper, $p$ denotes a weak solution of \eqref{eq:intro-HS-law} in the
distributional sense: writing
$\rho:=\chi_{\{p>0\}}$, we require that
\[
    \partial_t\rho-\Delta p=c\rho \qquad\hbox{in }\mathcal D'(\R^d\times(0,\infty)),
\]
together with natural integrability and time-continuity conditions (see Definition~\ref{def:HS-weak-solution}). Section~\ref{sec:HS-weak-basic} develops the
basic theory that we require, including a comparison principle
(Theorem~\ref{thm:HS-distributional-comparison}), proved in
Appendix~\ref{app:distributional-comparison} in a relaxed form that accommodates the
degenerate barriers used throughout the paper.

 We describe the evolution of the patch through the Baiocchi
transform \cite{BCMP73}
\begin{equation}\label{eq:intro-baiocchi}
    w(x,t):=\int_0^tp(x,s)\,ds,
    \qquad
    \Omega_t:=\{w(\cdot,t)>0\},
    \qquad
    T(x):=\inf\{t>0:x\in\Omega_t\}.
\end{equation}
The set $\Omega_t$ is open, agrees with $\{p(\cdot,t)>0\}$ up to a null set, and satisfies
$\Omega_t=\{T<t\}$; see \eqref{eq:HS-open-phase-identities}. It follows that for any interval $I\subset (0, \infty)$, $\partial\{w>0\}\cap\{t\in I\}$ is the space-time graph of $T$ over the set
\begin{equation}\label{eq:O}
    O_I:=\{x\in\R^d:T(x)\in I\}.
\end{equation}
We abbreviate $O := O_{(0,\infty)} = \{ x\in \R^d : 0 < T(x) < \infty\}$.

Integrating the equation for $p$ in time shows that, away from the initial patch, each slice
$w(\cdot,t)$ solves an elliptic obstacle problem  in $O$:
\begin{equation}\label{eq:intro-w-obstacle}
    \Delta w(x,t)=f(x,t)\chi_{\{w(x,t)>0\}},
    \qquad
    f(x,t):=1-\int_{T(x)\wedge t}^{t}c(x,s)\,ds.  
\end{equation}
whose source $f$ is positive at times close to $T(x)$ and involves the hitting time itself.
Consequently, at each $x\in O\setminus\overline{\Omega_0^{+}}$ the function $w(\cdot,T(x))$
admits a quadratic blowup at $x$ (Lemma~\ref{lem:stationary-blowup}), which is either a
half-space solution or a nonnegative $2$-homogeneous polynomial $p_2$ with
$\Delta p_2=1$, since $f(x,T(x))=1$ by \eqref{eq:intro-w-obstacle}. We call $x$ {\it regular} in the first case and {\it singular} in the
second, and denote by $\mathcal R$ and $\Sigma$ the corresponding subsets of $O$. A singular
blowup satisfies $D^2p_2\neq0$, so the {\it spine}
$V_x:=\ker D^2p_2=\{p_2=0\}$ has dimension at most $d-1$, and $\Sigma$ stratifies as
\begin{equation}\label{eq:intro-strata}
    \Sigma=\bigcup_{k=0}^{d-1}\Sigma_k,
    \qquad
    \Sigma_k:=\{x\in\Sigma:\dim V_x=k\}.
\end{equation}
We write $\Sigma(t_0):=\Sigma\cap\{T=t_0\}$ and $\Sigma_k(t_0):=\Sigma_k\cap\{T=t_0\}$ for the
time slices. The top stratum $\Sigma_{d-1}$ consists of the points at which the zero set of
the blowup is a hyperplane; these are the {\it collision-type} singularities referred to in
result~\ref{i:result-space-time-profile} above, which occur, for instance, when two advancing portions of the patch meet. All of our local results
concern the complementary strata $\Sigma_k$ with $k\leq d-2$.

Our standing hypotheses are the following. The first collects the qualitative requirements on
the source, and the second is a mild restriction on the initial patch:
\begin{equation*}\tag{A1}\label{eq:A-source}
\begin{gathered}
    c:\R^d\times[0,\infty)\to[0,\infty)\hbox{ is continuous and, for every }\tau>0,\\
    0<\underline c_\tau\leq c\leq\overline c_\tau<\infty
    \quad\hbox{on }\R^d\times[0,\tau],
    \qquad
    \norm{\nabla_xc(\cdot,t)}_{L^\infty(\R^d)}\leq L_\tau
    \quad\hbox{for a.e. }t\in[0,\tau];
\end{gathered}
\end{equation*}
\begin{equation*}\tag{A2}\label{eq:A-initial}
    \Omega_0^{+}\hbox{ is bounded.}
\end{equation*}
Both hypotheses are assumed in each of our main results. At each fixed time, Hypothesis \eqref{eq:A-source} controls the spatial variation of \(c\): the Lipschitz bound at small scales, and the uniform upper and lower bounds at large scales. It also makes absolute and relative control of temporal differences equivalent, since \((\partial_t\log c)^-=(\partial_t c)^-/c\).

Let us mention that the injection problem, in which the source is concentrated on a fixed
boundary, does not satisfy \eqref{eq:A-source}. There, however, the uniform positivity of
the fixed boundary data replaces the lower bound on $c$ in \eqref{eq:A-source}. Overall, the technical work needed here is largely eliminated due to the simplifications that arise from \(c\equiv0\) in the bulk. We therefore expect our results to extend
without notable challenges; see Section~\ref{sec:injection}. On the other hand, when $c$ changes sign, the
correspondence between \eqref{eqn:density} and \eqref{eq:intro-HS-law} breaks down
altogether, and the regularity of the free boundary remains largely open \cite{GKM22}.

The remaining hypotheses enter only some of our results. The first is a
nondegeneracy condition on the initial patch:
\begin{equation*}\tag{IB}\label{eq:A-interior-ball}
    \Omega_0^{+}\hbox{ satisfies a uniform interior ball condition of radius }r_0>0 .
\end{equation*}
This excludes inward cusps and corners of the initial patch, and is the hypothesis
that yields Lipschitz, rather than merely H\"older, continuity of the hitting time
(see Theorem~\ref{thm:HS-hitting-time-lipschitz} below).
The second is a time-regularity requirement on the source, indexed by an integrability
exponent $m$:
\begin{equation*}\tag{S$_m$}\label{eq:A-source-time}
    (\partial_tc)^-\in L^1_{\operatorname{loc}}\bigl([0,\infty);L^m(\R^d)\bigr),
    \qquad m\in(d/2,\infty] .
\end{equation*}

The hypothesis \eqref{eq:A-source-time} comes in varying strengths: stronger time
regularity of the source yields stronger conclusions. A hypothesis of this kind is needed
because the quasi-static evolution \eqref{eq:intro-HS-law} provides no time regularity for
the pressure. Indeed, $p$ may increase discontinuously when patch parts merge together. On the other hand, the decrease of the pressure can still be controlled: a lower bound on $\partial_t p$ can be derived under assumptions such as \eqref{eq:A-source-time}, either through Aronson--B\'enilan estimates for the porous medium approximation (see e.g. \cite{PQV14}) or by a direct argument, as we do here (Lemma~\ref{lem:HS-finite-difference-pt-lower}). For $m>d/2$, the threshold of elliptic regularity, this control is uniform
in space, and the pressure acquires a canonical pointwise representative
(Proposition~\ref{prop:HS-USC-representative}), to which all pointwise statements below
refer. In this range the Hopf--Lax inequality of Theorem~\ref{prop:HS-pressure-hopf-lax}
holds, and the hitting time is H\"older continuous, with no hypothesis on the initial patch
beyond \eqref{eq:A-initial} (Remark~\ref{rem:HS-hitting-time-holder}). The endpoint case
$m=\infty$, combined with \eqref{eq:A-interior-ball}, yields the Lipschitz estimate of
Theorem~\ref{thm:HS-hitting-time-lipschitz}. For the
nutrient system \eqref{eqn:density}--\eqref{eqn:nutrient} in the physical dimensions
$d\leq3$ we verify
in the proof of Theorem~\ref{thm:tumor-application} that $\partial_t c\in L^4(0,\tau;L^\infty(\R^d))$, so  \eqref{eq:A-source-time} holds for $m=\infty$, so all the conclusions discussed above apply.

\subsection{Main results}\label{subsec:intro-results}

Our first result establishes the nondegeneracy asserted in~\ref{i:result-nondegeneracy}, in the form of a Lipschitz
bound on the hitting time. We recall that the set $O$ was defined below \eqref{eq:O}.

\begin{thm}[Nondegeneracy and Lipschitz regularity of the hitting time]\label{thm:HS-hitting-time-lipschitz}
Assume \eqref{eq:A-source}, \eqref{eq:A-initial}, \eqref{eq:A-interior-ball}, and \eqref{eq:A-source-time} with $m=\infty$. Then, with $O=T^{-1}((0,\infty)),$
\[
    \partial\{p>0\} =\{(x,T(x)):x\in O\},
\]
and $T:\R^d\to[0,\infty]$ is continuous in the extended sense. Moreover, for each $\tau>0$
there exist $L<\infty$ and $\vep>0$, depending only on
\[
    d,\ r_0,\ \tau,\ \inf_{\R^d\times[0,\tau+1]}c,\
    \norm{\nabla_xc}_{L^\infty(\R^d\times(0,\tau+1))},\
    \norm{(\partial_tc)^-}_{L^1(0,\tau+1;L^\infty(\R^d))},
\]
such that
\begin{equation}\label{eq:HS-T-local-Lipschitz-estimate}
    \abs{T(x)-T(y)}\leq L\abs{x-y}
    \qquad\hbox{whenever }\min\{T(x),T(y)\}<\tau\hbox{ and }\abs{x-y}<\vep .
\end{equation}
In particular, $O$ is open, and $T$ is locally Lipschitz in $O$.
There also exist $\kappa,r_*>0$, depending only on the same quantities and on
$\overline c_\tau:=\sup_{\R^d\times[0,\tau]}c$, such that
\begin{equation}\label{eq:HS-pressure-linear-nondegeneracy}
    \sup_{B_r(x)}p(\cdot,t)\geq\kappa r
    \qquad\hbox{if }x\in\partial\Omega_t,\quad 0<t<\tau,\quad 0<r<r_*.
\end{equation}

\end{thm}

Because the normal velocity of the patch is $\abs{\nabla T}^{-1}$, the estimate
\eqref{eq:HS-T-local-Lipschitz-estimate} is a quantitative lower bound on the speed of the free boundary,
valid at singular as well as at regular points. This is the regularity conjectured in
\cite{CJK25}, where the hitting time was shown to be H\"older continuous, and Lipschitz
only on the regular set. Lipschitz
continuity is optimal, since at a collision $T$ behaves like the minimum
of two smooth functions.

Separately, and without any nondegeneracy assumption on the initial patch, we obtain the
following Hopf--Lax inequality for the pressure.

\begin{thm}[Hopf--Lax inequality for the pressure]\label{prop:HS-pressure-hopf-lax}
Assume \eqref{eq:A-source}, \eqref{eq:A-initial}, and \eqref{eq:A-source-time}. Let $0<t_0<t_1$, and set
\begin{equation}\label{eq:HS-Hopf-Lax-cinf-L-M}
    \underline c:=\inf_{\R^d\times[t_0,t_1]}c,\qquad
    L:=\norm{\nabla_xc}_{L^\infty(\R^d\times(t_0,t_1))},\qquad
    M:=\frac{L}{\underline c} .
\end{equation}
Fix $R_*\geq1$ with $\overline{\Omega_{t_1}}\subset B_{R_*}$, and for $t\in[t_0,t_1]$ put
$\Phi(t):=\int_{t_0}^t\norm{(\partial_sc)^-(\cdot,s)}_{L^m(B_{4R_*})}\,ds$ and
\begin{equation}\label{eq:HS-Hopf-Lax-Lambda-q}
    \Lambda(t):=
    \begin{cases}
    \displaystyle
    C_{d,m}\,\underline c^{-1}
    \left(L^{\frac{d}{d+m}}\Phi(t)^{\frac{m}{d+m}}+\Phi(t)\right),
        &d/2<m<\infty,\\[2ex]
    \displaystyle
    \underline c^{-1}\Phi(t),
        &m=\infty .
    \end{cases}
\end{equation}
Then, writing $A(r,s):=Mr\frac{s-t_0}{t_1-t_0}+\Lambda(s)$, we have
\begin{equation}\label{eq:HS-pressure-hopf-lax}
    p(y,t_0)
    \leq
    e^{A(\abs{x-y},t_1)}\,p(x,t_1)
    +\frac{\abs{x-y}^2}{4(t_1-t_0)^2}
    \int_{t_0}^{t_1}e^{A(\abs{x-y},s)}\,ds
    \qquad\hbox{for all }x,y\in B_{R_*}.
\end{equation}
\end{thm}

Formally, \eqref{eq:HS-pressure-hopf-lax} is the integrated form of the Hamilton--Jacobi inequality
$\partial_t p-(\abs{\nabla p}-M(t)p)_+^2+k(t)p\geq0$, where the coefficients $M(t)$ and $k(t)$ measure the
relative spatial oscillation and time decrease of the source; see Remark~\ref{rem:HS-Hopf-Lax-HJ}. An inequality of
this type was discovered in \cite{CJK25}, through a porous medium approximation together    
with elaborate Aronson--B\'enilan estimates, and a direct derivation was raised there as an
open question; the proof of Theorem~\ref{prop:HS-pressure-hopf-lax} answers this through a novel comparison argument, and is significantly simpler than the original derivation. Our inequality also
improves on the earlier one in that it
carries no additive error term, and it requires less time regularity of the source.

We turn to the local analysis at a singular point, which forms the core of the paper
(Sections~\ref{sec:static-barriers}--\ref{sec:improvement-flatness}). After a
translation we place the point at the origin and its hitting time at $t_0$.
If $V$ is the spine of the quadratic blowup at the singular point, we
write $x=(x',x'')\in V^\perp\times V$.
The theorem below measures the portion of the zero set that lies outside a thin
cone around the spine $V$ by its maximal distance from $V$ (see
Figure~\ref{fig:r-perp}). The aperture of the cone is a small fixed constant
$\alpha=\alpha(d,k)\in(0,1)$.
The theorem determines both the rate at which this transverse radius
vanishes and the asymptotic shape of the zero set.

\begin{figure}[ht]
\centering
\begin{tikzpicture}[scale=0.84,>=Stealth]
    \fill[green!10] plot [smooth cycle, tension=0.9] coordinates
        {(0:3.6) (40:3.95) (85:3.75) (130:3.55) (175:3.7) (220:3.5) (265:3.75) (315:3.55)};
    \draw[thin] plot [smooth cycle, tension=0.9] coordinates
        {(0:3.6) (40:3.95) (85:3.75) (130:3.55) (175:3.7) (220:3.5) (265:3.75) (315:3.55)};
    \filldraw[fill=white, draw=black, thin] plot [smooth cycle, tension=0.7] coordinates
        {(0.48,0.35) (0.22,0.55) (0.18,0.7) (0.3,1.1) (0.25,1.5) (0.62,2.0)
         (0.5,2.35) (0.85,2.75) (0,3.4)
         (-0.75,2.7) (-0.35,2.1) (-0.5,1.6) (-0.28,1.0) (-0.18,0.7)
         (-0.33,0.25) (-0.36,0) (-0.33,-0.25) (-0.18,-0.7)
         (-0.35,-1.2) (-0.3,-1.8) (-0.7,-2.3) (-0.6,-2.9) (0,-3.4)
         (0.8,-2.8) (0.4,-2.2) (0.55,-1.7) (0.25,-1.05) (0.18,-0.7)
         (0.32,-0.28) (0.42,-0.05)};
    \draw[->] (-3.9,0) -- (3.9,0) node[below right] {$V^\perp$};
    \draw[->] (0,-3.9) -- (0,3.9) node[above right] {$V$};
    \draw[dashed] (110:3.3) -- (0,0) -- (70:3.3);
    \draw[dashed] (250:3.3) -- (0,0) -- (290:3.3);
    \node[right] at (70:3.35) {\tiny $|x'|=\alpha|x''|$};
    \draw[dotted, thick] (0,0) circle (3);
    \node at (160:2.65) {\small $B_{R_0}$};
    \fill (0.48,0.35) circle (1.7pt);
    \draw[dashed, thin] (0.48,-0.75) -- (0.48,0.95);
    \draw[<->, >={Stealth[length=3.5pt,width=2.2pt]}] (0,-0.5) -- (0.48,-0.5);
    \node[below right] at (0.2,-0.52) {\scriptsize $r_\perp(t)$};
    \node at (40:3.45) {$\Omega_t$};
\end{tikzpicture}
\caption{The setting of Theorem~\ref{thm:closing-rates} near a singular point with spine
$V$, at a time $t<t_0$: within $B_{R_0}$ (dotted), the zero set of $p(\cdot,t)$ (white) lies
in the union of a thin cone around $V$ (dashed) with a small ball, and $r_\perp(t)$ is its
maximal distance from $V$ outside the cone.}
\label{fig:r-perp}
\end{figure}
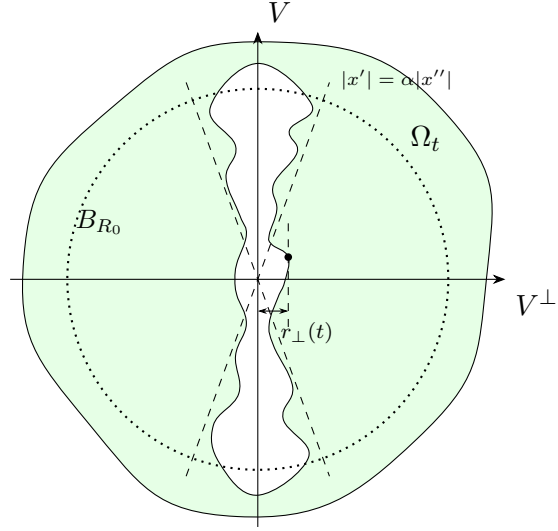

\begin{thm}[Closing rate and asymptotic shape at a singular point]\label{thm:closing-rates}
Assume \eqref{eq:A-source}, \eqref{eq:A-initial}, and \eqref{eq:A-source-time}. Let $t_0>0$, and let
$0\in\Sigma_k(t_0)$ with $0\leq k\leq d-2$. Let $p_2$ be the blowup of $w(\cdot,t_0)$ at the
origin, and set $V:=\ker D^2p_2$. Let $\alpha\in(0,1)$ be a
sufficiently small constant depending only on $d$ and $k$. For $R_0>0$, set
\begin{equation}\label{eq:intro-r-perp}
    r_\perp(t):=
    \sup\bigl\{\abs{x'}:x\in B_{R_0}\setminus\Omega_t,\
    \abs{x'}>\alpha\abs{x''}\bigr\}.
\end{equation}
We have $r_\perp(t)>0$ for all $t<t_0$, and if $R_0$ is sufficiently small, then $r_\perp(t)$ tends to $0$ as
$t\uparrow t_0$. Moreover, there exist $C>1$ and, for each $\epsilon>0$,
$C_\epsilon>1$ such that the following estimates hold for $t<t_0$
sufficiently close to $t_0$:
\begin{enumerate}[label=(\roman*)]
\item\label{item:intro-closing-high-d}
if $k=0$ and $d\geq3$, then
$C^{-1}(t_0-t)^{1/2}\leq r_\perp(t)\leq C(t_0-t)^{1/2}$;
\item\label{item:intro-closing-d2}
if $k=0$ and $d=2$, then
$C^{-1}(t_0-t)\leq r_\perp(t)^2\log\frac{R_0}{r_\perp(t)}$
$\leq C(t_0-t)$;
\item\label{item:intro-closing-intermediate}
if $1\leq k\leq d-2$, then
$C_\epsilon^{-1}(t_0-t)^{\frac12+\epsilon}\leq r_\perp(t)\leq C(t_0-t)^{1/2}$.
\end{enumerate}
Finally,
\begin{equation}\label{eq:intro-cylinder-limit}
    \{x:w(r_\perp(t)x,t)=0\}\longrightarrow E \times V
    \qquad\hbox{as }t\uparrow t_0,
\end{equation}
locally in Hausdorff distance, where $E\subset V^\perp$ is a $(d-k)$-dimensional ellipsoid depending
only on $p_2$.
\end{thm}

Thus, at a singular point below the top stratum, the zero set closes at essentially the
parabolic rate $(t_0-t)^{1/2}$ of the explicit cylindrical solution.
The rates are fully sharp for $k=0$, while for $1\leq k\leq d-2$ they determine
the exact exponent: $r_\perp(t)=(t_0-t)^{\frac12+o(1)}$ as $t\uparrow t_0$.

As a by-product of the same barriers we obtain a quantitative form of the divergence of the
velocity.

\begin{cor}[Blowup of the free boundary velocity]\label{prop:gradient-lq-blowup}
Under the assumptions of Theorem \ref{thm:closing-rates},
\[
    \int_{t_0-\delta}^{t_0}\int_{B_\delta}\abs{\nabla p}^q\,dxdt=\infty
    \qquad\hbox{for every }0<\delta<t_0\hbox{ and every }q>d+2 .
\]
If moreover $\Sigma_k(t_0)$ has positive lower $\mathcal H^k$-density at the origin, the same
conclusion holds for every $q>d-k+2$.
\end{cor}

In particular, taking $k=d-2$ shows that the bound $\nabla p\in L^4(\R^d\times (0,\tau))$ of \cite{DavidPerthame} is sharp. The blowup at $q>d-k+2$ can be observed in self-similar cylindrical profiles; see Remark~\ref{rem:lq-sharpness}.

The closing rate of Theorem~\ref{thm:closing-rates} describes the free boundary at times
strictly before $t_0$. Our next result concerns the hitting time itself, and quantifies
the nondegeneracy of the flow at a singular point.

\begin{thm}[Vanishing gradient of the hitting time at singular points]\label{prop:cone_T_differentiability}
Assume the hypotheses of Theorem~\ref{thm:HS-hitting-time-lipschitz}, and let $0\in\Sigma_k(t_0)$ with
$0\leq k\leq d-2$. Then for every $\vep\in(0,1)$ there exists $r_\vep>0$ such that
\begin{equation}\label{eq:cone_T_near_sharp_modulus}
    \abs{T(x)-T(y)}\leq r^{1-\vep}\abs{x-y},
    \qquad x,y\in B_{r/8},\quad 0<r<r_\vep .
\end{equation}
In particular $T$ is differentiable at the origin with $\nabla T(0)=0$; and if $x_j\to0$ and
$T$ is differentiable at each $x_j$, then $\nabla T(x_j)\to0$.
\end{thm}

Since $\abs{\nabla T}^{-1}$ is the normal speed of the patch, \eqref{eq:cone_T_near_sharp_modulus}
says that the free boundary sweeps a ball of radius $r$ around a singular point in time
$o(r^{2-\vep})$. In particular, at every singular point outside the top stratum the velocity
diverges, at a rate at least $r^{-(1-\vep)}$ at distance $r$, for every $\vep>0$. 
Combining Theorem~\ref{prop:cone_T_differentiability} with the corresponding statement at regular points
yields the regularity of the space-time free boundary asserted in~\ref{i:result-space-time-regularity}.

\begin{thm}[$C^1$ regularity at times without collisions]\label{thm:no_top_stratum_C1}
Assume the hypotheses of Theorem~\ref{thm:HS-hitting-time-lipschitz}, let $t_0>0$, and suppose that
$\Sigma_{d-1}(t_0)=\emptyset$. Then there is an open interval $t_0\in I\Subset(0,\infty)$ such that 
\[
    T\in C^1(O_I)\quad \hbox{ and }\quad
    \Sigma\cap O_I=\{x\in O_I:\nabla T(x)=0\} ,{\hbox{ where } O_I \hbox{ is given by } \eqref{eq:O}.}
\]
In particular the space-time free boundary is a $C^1$ hypersurface over $O_I$, and its
singular points are exactly the points at which the normal velocity of the patch diverges.
\end{thm}

Collisions act nonlocally: when two portions of the patch merge, the
pressure is determined on a newly connected domain, and its normal derivative---and hence
the velocity---can change instantaneously at far-away regular parts of the boundary
(see \cite[App.~A.2]{PQV14}).
It is illuminating to contrast the behavior of the hitting time in Hele-Shaw flow with that in the classical and
the supercooled Stefan problems. In the supercooled problem, although solutions are unstable and non-unique, the hitting time of any given solution is the best behaved of the three: every
singularity---collisions included---forms with infinite speed, and the hitting time is
$C^1$ \cite[Thm.~1.2]{EKM25}. In the classical problem, collisions occur with finite
speed, and the hitting time is Lipschitz \cite{caffarelli1978} but fails to be $C^1$
precisely at the top stratum \cite[Thm.~1.4]{FRT26}; the failure is local in space. In
Hele-Shaw flow, a single collision can affect the regularity of $T$ globally in
space. Theorem~\ref{thm:no_top_stratum_C1} shows that
collisions are the only mechanism for such discontinuities: the remaining singularities
are too small, in the sense of capacity, to disturb the pressure field away from
the singularity.

Finally, we show that the tumor growth model falls within the scope of the
above results.

\begin{thm}[The tumor growth model with nutrients]\label{thm:tumor-application}
Let $d\in \{1,2,3\}$, and let $(\rho,p,c)$ be a patch solution of
\eqref{eqn:density}--\eqref{eqn:nutrient} with $\rho(\cdot,0)=\chi_{\Omega_0^{+}}$ and
$\Omega_0^{+}$ bounded and satisfying \eqref{eq:A-interior-ball}. Assume that
\begin{equation}\label{eq:tumor-initial-nutrient-assumptions}
    c_0\in W^{2,\infty}(\R^d),
    \qquad
    \underline c_0:=\inf_{\R^d}c_0>0 .
\end{equation}
Then $c$ satisfies \eqref{eq:A-source} and \eqref{eq:A-source-time} with
$m=\infty$. Consequently,
Theorems~\ref{thm:HS-hitting-time-lipschitz}, \ref{thm:closing-rates},
\ref{prop:cone_T_differentiability} and \ref{thm:no_top_stratum_C1} all apply to the tumor
patch; in particular, the hitting time is locally Lipschitz in $\{T<\infty\}$.
\end{thm}

The dimensional restriction is necessary for our argument, which bounds $\partial_t c$
by inserting the sharp $L^4$ estimate for $\nabla p$ of \cite{DavidPerthame}, used in its
full strength, into the smoothing of the heat semigroup; the resulting singularity in time
is integrable exactly in the physically relevant dimensions $d\leq 3$. In Section~\ref{sec:tumor-applications} we also show that $\abs{\Omega_t}\geq\abs{\Omega_0^{+}}+Ct^{d/2}$ for the growth of the tumor region
(Proposition~\ref{prop:tumor-volume-growth}), which implies
in particular that the tumor region cannot remain bounded.

\subsection{Main ideas}\label{subsec:intro-ideas}

The difficulty of the analysis comes from two features of \eqref{eq:intro-HS-law}.
The first is the quasi-static nature of the
flow, discussed in \S\ref{subsec:intro-setting} and shared with the injection
problem: the dynamics of the patch are strongly
nonlocal, and only the decrease of the pressure in time can be controlled. The second is the presence of the
hitting time in the elliptic problem \eqref{eq:intro-w-obstacle} solved by the Baiocchi transform. For the injection and
Stefan problems the Baiocchi transform linearizes the flow into a {\it homogeneous} obstacle
problem; here the obstacle problem source $f$ depends on $T$, whose Lipschitz continuity is itself
one of our main results and cannot be improved, and every argument that differentiates the obstacle problem must carry a corrector
for the variation of $f$.

We begin with the proof of Theorem~\ref{thm:HS-hitting-time-lipschitz}, which compares the
solution with a time-rescaled copy of itself. The strategy has its roots in the
barrier methods of Caffarelli \cite{Caffarelli77,Caffarelli98}, and \cite{kim03}, here in substantially refined form, as the rescaling must
be adapted to the time dependence of the source. Given $\vep>0$, we seek an
increasing reparametrization $\Theta$ for which $\Theta'(t)\,p(x,\Theta(t))$ is a
supersolution dominating every $\vep$-translate of $p$. The comparison principle reduces
this to an ordering of the sources and an ordering of the initial data. The first is
arranged in Lemma~\ref{lem:HS-source-envelope} by solving an ordinary differential
equation for $\Theta$, and is where the endpoint $m=\infty$ of \eqref{eq:A-source-time} is
used. The second is the only place in the paper where the interior ball condition
\eqref{eq:A-interior-ball} enters: radial solutions planted in interior balls show that
the patch covers an $\vep$-neighborhood of $\Omega_0^{+}$ within time proportional to
$\vep$ (Lemma~\ref{lem:HS-initial-nondegeneracy}). The resulting lag $\Theta(t)-t$ is
linear in $\vep$, and this yields the Lipschitz estimate.

We turn to the Hopf--Lax inequality of Theorem~\ref{prop:HS-pressure-hopf-lax}, for which we compare $p$
with the barrier
\[
    \bar p(z,t):=
    e^{-A(t)}
    \Bigl(p\bigl(z+\tfrac{t-t_0}{t_1-t_0}h,t_0\bigr)-C(t)\Bigr)_+ ,
\]
built from the {\it single} time slice $p(\cdot,t_0)$ by translating it at constant velocity
$h/(t_1-t_0)$ while subtracting a growing constant. Completing the square shows that the
free boundary condition $\partial_t\bar p\leq\abs{\nabla\bar p}^2$ holds identically, and \eqref{eq:HS-pressure-hopf-lax}
follows by evaluating the comparison at $(x,t_1)$. The main technical work lies in the low
regularity of $\bar p$, which inherits only what a single time slice of a weak solution
provides.

Sections~\ref{sec:static-barriers} and
\ref{sec:closing-rates} analyze the patch shortly before a singularity of $\Sigma_k$ forms.
In this regime the zero set is asymptotically concentrated near
the spine $V$, and thus we introduce an extra parameter $\alpha>0$ so that for a small $R_0>0$ the conical annulus
$U(V,\alpha,R_0,r):=\{r<\abs{x}<R_0,\ \abs{x'}>\alpha\abs{x''}\}$ lies in the patch
(see Figure~\ref{fig:U-region}). On such a region we construct an explicit subsolution for the pressure by
separating radial and angular variables.
The crucial estimate, proved in Lemma~\ref{lem:first_eigenvalue}, 
yields the almost-optimal growth $p\gtrsim\underline cR_0^{2-\nu_\alpha}\abs{x}^{\nu_\alpha}$, with $\nu_{\alpha}\to 0$ as $\alpha\to 0$.
This degeneration of $\nu_{\alpha}$ occurs only when $k\leq d-2$: this is consistent with the fact that top-stratum
collisions may occur with finite speed and therefore need not satisfy the conclusions of
Theorems~\ref{thm:closing-rates} and \ref{prop:cone_T_differentiability}.

Using the growth rate we obtain here, we show the closing rates of Theorem~\ref{thm:closing-rates} by a barrier argument.
An interesting point here is the use of the Hopf--Lax inequality (Theorem~\ref{prop:HS-pressure-hopf-lax}) to show that the pressure growth estimate prevents the zero set from closing too quickly
(see
Section~\ref{sec:closing-rates}).
For the asymptotic shape we
apply
the classification of global solutions of the obstacle problem
\cite{FriedmanSakai1986,EFW25} to the rescaled solution, to show that zero sets converge to a cylinder
over an ellipsoid determined by $p_2$.

Theorem~\ref{prop:cone_T_differentiability} is proved by comparing the pressure with directional
derivatives of $w$, in the spirit of Caffarelli's barrier proof of Lipschitz regularity for
the Stefan problem \cite{caffarelli1978}, and of its recent refinements near singular points
\cite{EKM25,FRT26}. One seeks a bound $\abs{\nabla w}\leq R^{1-\sigma}p$ on a
space-time window around the singularity, which integrates to the modulus
\eqref{eq:cone_T_near_sharp_modulus}. In contrast to the Stefan problem, the lack of time regularity of $p$ forces us to treat
the times before and after $t_0$ by different arguments, relying on the closing rate and the
barrier of Lemma~\ref{lem:cone_pressure_away_spine} on the backward side, and on the slice
comparison for the pressure on the forward side. A second difficulty new to this paper is
the dependence of $f$ in \eqref{eq:intro-w-obstacle} on the hitting time.
With a constant source, one makes a competitor superharmonic by subtracting an
explicit quadratic; no fixed quadratic can absorb a source that varies with $T$. We
instead introduce Green's function correctors, whose precise choice is a delicate point of
the argument. The resulting
estimate holds on a window whose time length is intrinsic to the spatial scale, and a
bootstrap argument simultaneously widens the window and improves the modulus.

Finally, Theorem~\ref{thm:no_top_stratum_C1} requires one further global ingredient. Knowing that $\nabla T$ vanishes at singular points and
is determined by $\nabla p$ at regular ones does not by itself give continuity of
$\nabla T$, because $\abs{\nabla p}$ at a regular point can jump when the topology of
$\Omega_t$ changes elsewhere. We show that only the top stratum can cause such jumps. In the
absence of $\Sigma_{d-1}$ the singular set is contained in countably many $C^1$ manifolds of
dimension at most $d-2$, and so has zero Sobolev $2$-capacity. Combined with the $C^1$
convergence of the regular portions of the free boundary, this gives Mosco convergence of the
spaces $H^1_0(\Omega_t)$ as $t\to t_*$ (Lemma~\ref{lem:no_top_mosco}). Since $p(\cdot,t)$ is
the minimizer of the Dirichlet energy on $H^1_0(\Omega_t)$, Mosco convergence yields strong
$H^1$ continuity in time, and elliptic estimates up to the regular boundary upgrade this to
local uniform convergence of $\nabla p$.

\subsection{Outline of the paper}\label{subsec:intro-outline}

Section~\ref{sec:preliminaries} collects the required facts about the elliptic obstacle
problem, and Section~\ref{sec:HS-weak-basic} develops the weak-solution theory, with the
comparison principle proved in Appendix~\ref{app:distributional-comparison}.
Section~\ref{sec:HS-nondegeneracy-hopf-lax} proves
Theorems~\ref{thm:HS-hitting-time-lipschitz} and \ref{prop:HS-pressure-hopf-lax}.
Section~\ref{sec:static-barriers} constructs the static pressure barriers, and
Section~\ref{sec:closing-rates} combines them with the Hopf--Lax inequality to prove
Theorem~\ref{thm:closing-rates} and Corollary~\ref{prop:gradient-lq-blowup}.
Section~\ref{sec:improvement-flatness} proves
Theorems~\ref{prop:cone_T_differentiability} and \ref{thm:no_top_stratum_C1}.
Section~\ref{sec:tumor-applications} proves Theorem~\ref{thm:tumor-application},
Section~\ref{sec:injection} discusses the injection problem, and
Appendix~\ref{app:global-obstacle} collects the facts about global solutions of the obstacle
problem used in the blowup analysis.

\section{Preliminaries}\label{sec:preliminaries}
In this section we recall some basic properties of the obstacle problem used throughout the paper. We say that $v$ solves the obstacle problem with source $f$ in $\Omega\subset \R^d$ if
\begin{equation}
    v\geq 0 \hbox{ and } \Delta v = f\chi_{\{v > 0\}} \hbox{ in }\Omega.
\end{equation}
Solutions to this problem exist and can be found by minimizing $\int\frac12|\nabla v|^2 + fv$ over nonnegative $v$ with prescribed Dirichlet data.

We fix a parameter $\lambda > 0$, and assume that all sources in the statements below satisfy
\begin{equation}
    \inf_{\Omega} f \geq \lambda.
\end{equation}

The first result we present is the quadratic growth of $v$ away from its free boundary $\partial\{ v > 0\}$.

\begin{lem}[{\cite[Theorems 2.1 and 2.4]{blank}}]\label{obst-quadratic-growth}
    Let $v$ solve the obstacle problem with source $f$ in $B_1$. Then for $r\in (0,1/2)$, we have
    \begin{enumerate}[label=(\roman*)]
        \item\label{item:obst-nondegeneracy} If $0\in \overline{\{ v > 0\}}$, then
        \[ \frac{\lambda}{2d}r^2 \leq \sup_{B_r} v. \]
        \item\label{item:obst-growth-bd} If $0\in \{ v= 0\}$, then
        \[ \sup_{B_r} v \leq C(d)\|f\|_\infty r^2. \]
    \end{enumerate}
\end{lem}

As a consequence of this quadratic growth estimate, we can perform a quadratic blowup at any point in the zero set. We summarize some basic properties of this blowup in the next statement.

\begin{lem}[{\cite[Theorem 2, Corollary 7, Lemma 12a]{Caffarelli98}}]\label{lem:obst-blowup}
Let $f_n$ be uniformly bounded in $C^{0,\alpha}(B_1)$ for some $\alpha\in (0,1)$. Let $u_n$ solve the obstacle problem with source $f_n$ in $B_1$, and suppose that $u_n(0) = 0$ for each $n$. Let $r_n > 0$ with $\lim_{n\to\infty}r_n = 0$, and set $v_n(x) = r_n^{-2}u_n(r_n x)$. Then,
\begin{enumerate}[label=(\roman*)]
    \item\label{item:obst-precompact} For any compact $K\subset \R^d$, there exists $N(K)$ such that for $n \geq N(K)$, $v_n$ is defined on $K$ and bounded in $C^{1,1}(K)$, uniformly in $n$.
    \item\label{item:obst-converges} There exists a subsequence $n_k$ along which $f_{n_k}(r_{n_k}\cdot)$ converges in $C^{0,\alpha-}_{\mathrm{loc}}$ to some constant $f_0$ and $v_{n_k}$ converges in $C^{1,1-}_{\mathrm{loc}}$ to some $v_0$. This $v_0$ solves
    \[ \Delta v_0 = f_0\chi_{\{v_0 > 0\}}, \qquad x\in \R^d. \]
    \item\label{item:obst-global-convex}The function $v_0$ is convex, and hence $\{ v_0 = 0\}$ is convex.
    \item\label{item:obst-hausdorff-dist} If $\{ v_0 = 0\}$ has nonempty interior, then the sets $\{ v_{n_k} = 0\}$ converge locally in Hausdorff distance to $\{ v_0 = 0\}$. That is, for every compact $K\subset \R^d$, we have
    \[ \lim_{k\to\infty}\max\left(\sup_{x\in K\cap \{ v_0 = 0\}} \mathrm{dist}(x, \{ v_{n_k} = 0\}), \sup_{x\in K\cap \{ v_{n_k} = 0\}} \mathrm{dist}(x, \{ v_0 = 0\}) \right) = 0. \]
\end{enumerate}
\end{lem}

In the next lemma, we restrict to the special case that $u_n\equiv u, f_n\equiv f$, and $0\in \partial \{ u > 0\}$. Then the blowup limit from the previous lemma is independent of the choice of $r_n$, and is either a polynomial or a half-space solution.

\begin{lem}[{\cite[Theorem 7, Theorem 8]{Caffarelli98}}]\label{lem:stationary-blowup}
Let $u$ solve the obstacle problem in $B_1$ with source $f\in C^{0,\alpha}(B_1)$, and assume $0\in \partial \{u > 0\}$. Set $u_r(x) = r^{-2}u(rx)$. Then $u_r$ converges in $C^{1,1-}_{\mathrm{loc}}$ as $r\to 0^+$ to a $u_0$ which solves $\Delta u_0 = f(0)\chi_{\{ u_0 > 0\}}$ on $\R^d$. Moreover, $u_0$ is either a nonnegative 2-homogeneous polynomial $p_2$ with $\Delta p_2 \equiv f(0)$, or $u_0(x) = \frac{f(0)}{2}(x\cdot e)_+^2$ for some $e\in \mathbb{S}^{d-1}$.
\end{lem}
\begin{thm}[Regularity of the hitting interface]
\label{thm:hitting-interface-regularity}
Assume the hypotheses of Theorem~\ref{thm:HS-hitting-time-lipschitz}, and let $I\Subset(0,\infty)$ be an interval.  Then the following hold in \(O_I=T^{-1}(I)\):
\begin{enumerate}[label=(\roman*)]
\item \cite[Props 5.4]{CJK25}
\(\mathcal R\) is relatively open and \(\Sigma\) is relatively closed.
\item \cite[Prop 5.6 and App.~6]{CJK25}. Near each \(q\in\mathcal R\), the boundaries $\partial\Omega_t$ admit a local graph representation $x_d=g(x',t)$, where $g(\cdot,t)$ is uniformly $C^{1,\beta}$ in $x'$ and $g(x',\cdot)$ is uniformly Lipschitz in $t$, for any $0<\beta<1$. The corresponding spatial outward unit normal extends continuously along \(\mathcal R\).
\item \cite[Props 5.13]{CJK25} The quadratic blowup of \(w(\cdot,T(q))\) at $q$ depends continuously on $q\in\Sigma$.
In particular, \(\Sigma_{d-1}\) is relatively closed in \(O_I\).
\item\label{i:singular-set-c1-manifold} \cite[Prop 5.15]{CJK25} For $k=0,\ldots,d-1$, near each \(q\in\Sigma_k\), the set \(\Sigma\) is contained in a
\(C^1\) manifold $M$ of dimension \(k\). The zero set of the quadratic blowup of $w(\cdot, T(q))$ at $q\in \Sigma$ is contained in the $k$-hyperplane tangent to $M$ at $q$.
\end{enumerate}

\end{thm}

\section{Weak Hele-Shaw solutions and basic estimates}\label{sec:HS-weak-basic}

Let us denote the space-time domains
\[
    Q:=\R^d\times[0,\infty),\qquad
    Q_\tau:=\R^d\times[0,\tau),\qquad Q_\infty:=Q.
\]
The definitions and some of the basic tools in this section allow an arbitrary
nonnegative source \(c\in L^\infty(Q_\tau)\), as they are sometimes applied to
barriers whose sources vanish.

Given a pressure \(p\), we define
\begin{equation}\label{eq:HS-basic-notation}
    \rho=\rho_p:=\chi_{\{p>0\}},
 \qquad
    \Omega_t^{+}:=\set{x\in\R^d:\rho(x,t)=1}\quad(t\geq0).
\end{equation}

Our analysis combines barrier arguments, which are most natural in the viscosity-solutions framework (see e.g. \cite{kimZhang24}), with integral estimates that require the weak formulation. We therefore work with weak solutions throughout, establishing a comparison principle (Theorem~\ref{thm:HS-distributional-comparison}) flexible enough to accommodate the degenerate barriers used in the paper.

\begin{defn}[Weak subsolutions and supersolutions]
\label{def:HS-distributional-subsuper}
Let \(\tau\in(0,\infty]\).  Let \(c\in L^\infty_{\operatorname{loc}}(Q_\tau)\)
be nonnegative and let \(p:Q_\tau\to[0,\infty)\).  Write
\(\rho:=\rho_p:=\chi_{\{p>0\}}\).  We say that \(p\) is a weak subsolution to
the Hele-Shaw law \eqref{eq:intro-HS-law}, with source \(c\), on \(Q_\tau\) if
\begin{equation}\label{eq:HS-weak-regularity}
    p\in L^2_{\operatorname{loc}}([0,\tau);H^1(\R^d)),
    \qquad
    \rho\in L^\infty_{\operatorname{loc}}([0,\tau);L^1(\R^d)),
\end{equation}
and, for the chosen \(L^1\)-valued representative of \(\rho\),
\begin{equation}\label{eq:HS-weak-time-traces}
    \rho\in C([0,\tau);L^1(\R^d)),
\end{equation}
and
\begin{equation}\label{eq:HS-subsolution-distributional}
    \partial_t\rho-\Delta p\leq c(x,t)\rho
    \qquad\hbox{in }\mathcal D'(\R^d\times(0,\tau)).
\end{equation}
A weak supersolution to \eqref{eq:intro-HS-law}, with source \(c\), is
defined by reversing the inequality in \eqref{eq:HS-subsolution-distributional}.
\end{defn}

\begin{defn}[Weak solution]\label{def:HS-weak-solution}
Let \(\tau\in(0,\infty]\).  Let \(c\in L^\infty_{\operatorname{loc}}(Q_\tau)\)
be nonnegative, and let \(p:Q_\tau\to[0,\infty)\).  Write
\(\rho:=\rho_p:=\chi_{\{p>0\}}\).  We say that \(p\) is a weak
solution to the Hele-Shaw law \eqref{eq:intro-HS-law}, with source \(c\), on
\(Q_\tau\), if it is both a weak subsolution and a weak supersolution on
\(Q_\tau\). Equivalently, \(p\) satisfies
\eqref{eq:HS-weak-regularity} and \eqref{eq:HS-weak-time-traces}, and
\begin{equation}\label{eq:HS-weak-equation}
    \partial_t\rho-\Delta p=c(x,t)\rho
    \qquad\hbox{in }\mathcal D'(\R^d\times(0,\tau)).
\end{equation}
When \(\tau=\infty\), we simply say that \(p\) is a weak solution.
\end{defn}

Existence and uniqueness of weak solutions to \eqref{eq:intro-HS-law} were established in \cite{PQV14} for the tumor growth model, and the argument extends readily to a general source $c$.

The following comparison
principle is a special case of
Theorem~\ref{thm:app-HS-relaxed-comparison}.
\begin{thm}[Comparison principle]\label{thm:HS-distributional-comparison}
Let \(0<\tau<\infty\), and let
\(c^-,c^+\in L^\infty(\R^d\times(0,\tau))\) satisfy
\(0\leq c^-\leq c^+\).  Let \(p^-\) and \(p^+\) be, respectively, a
weak subsolution to \eqref{eq:intro-HS-law} with source \(c^-\)
and a weak supersolution to \eqref{eq:intro-HS-law} with
\(c^+\) on \(Q_\tau\).  Write
\(\rho^\pm:=\chi_{\{p^\pm>0\}}\). If
\begin{equation}\label{eq:HS-initial-order}
    \rho^-(\cdot,0)\leq\rho^+(\cdot,0)
    \qquad\hbox{a.e. in }\R^d,
\end{equation}
then
\begin{equation}\label{eq:HS-pressure-comparison}
    p^-\leq p^+
    \qquad\hbox{a.e. in }\R^d\times(0,\tau).
\end{equation}

\end{thm}

\begin{proof}
Apply Theorem~\ref{thm:app-HS-relaxed-comparison} to the relaxed pairs
\((\rho_{p^-},p^-)\) and \((\rho_{p^+},p^+)\).  These pairs
satisfy Definition~\ref{def:app-HS-relaxed-subsuper} because \(p^\pm\) are
sub- and supersolutions in the saturated sense of
Definition~\ref{def:HS-distributional-subsuper}.  The relaxed initial ordering
\eqref{eq:app-HS-relaxed-initial-order} is exactly
\eqref{eq:HS-initial-order}, and
\eqref{eq:app-HS-relaxed-pressure-comparison} gives
\eqref{eq:HS-pressure-comparison}.
\end{proof}

\begin{lem}[Basic properties of weak solutions]\label{lem:HS-basic-properties}
Let \(p\) be a weak solution to \eqref{eq:intro-HS-law}, with source \(c\)
satisfying \eqref{eq:A-source}. Then the following
hold:
\begin{enumerate}[label=(\roman*)]
\item For any $\tau>0$, \(p\in L^\infty(Q_{\tau})\).  If \eqref{eq:A-initial}
holds, then \(p\) is compactly supported in \(Q_{\tau}\).
\item The sets \(\Omega_t^{+}\) are nondecreasing up to null sets, that is,
\begin{equation}\label{eq:HS-monotone-sets}
    \abs{\Omega_s^{+}\setminus\Omega_t^{+}}=0
    \qquad\hbox{for }0\leq s\leq t.
\end{equation}
Moreover,
\begin{equation}\label{eq:HS-lower-laplacian-bound}
    \Delta p\geq -c\rho\geq -c
    \qquad\hbox{in }\mathcal D'(Q).
\end{equation}
\item For a.e. \(t\in[0,\infty)\), the slice \(p(\cdot,t)\) satisfies
\begin{equation}\label{eq:HS-slice-variational}
    \int_{\R^d}\nabla p(\cdot,t)\cdot\nabla\phi\,dx
    =
    \int_{\R^d}c(\cdot,t)\phi\,dx,
\end{equation}
for every  \(\phi\in H^1(\R^d)\) such that $\phi\geq 0$ and
\(\phi(1-\rho(\cdot,t))=0\) a.e.
\item If \(0<s<t<\infty\) are two times for which
\eqref{eq:HS-slice-variational} holds, \(A\geq0\), and
\begin{equation}\label{eq:HS-source-slice-comparison}
    c(\cdot,s)\leq Ac(\cdot,t)
    \qquad\hbox{ a.e. in }\Omega_s^{+} ,
\end{equation}
then
\begin{equation}\label{eq:HS-slice-pressure-comparison}
    p(\cdot,s)\leq Ap(\cdot,t)
    \qquad\hbox{ a.e. in  } \R^d .
\end{equation}
In particular, \eqref{eq:HS-slice-pressure-comparison} holds
for a.e. \(0<s<t\leq\tau\), with \(A=\overline c_\tau/\underline c_\tau\).
\end{enumerate}
\end{lem}

\begin{proof}
Set
\[
    C_\tau:=\norm{c}_{L^\infty(Q_{\tau})}.
\]
We first prove \eqref{eq:HS-monotone-sets}.  Fix \(s\geq0\), and let
\(\sigma>0\).  On the time interval \([0,\sigma)\), set
\[
    \rho^-(x,t):=\rho(x,s),\qquad p^-(x,t):=0,\qquad c^-(x,t):=0,
\]
and
\[
    \rho^+(x,t):=\rho(x,s+t),\qquad
    p^+(x,t):=p(x,s+t),\qquad
    c^+(x,t):=c(x,s+t).
\]
The frozen pair \((\rho^-,p^-)\) is a relaxed weak subsolution in the sense of
Definition~\ref{def:app-HS-relaxed-subsuper}: the algebraic constraints in
\eqref{eq:app-HS-relaxed-graph} and the regularity in
\eqref{eq:app-HS-relaxed-regularity} are immediate, and
\(\partial_t\rho^-=\Delta p^-=c^-\rho^-=0\).
The shifted pair \((\rho^+,p^+)\) is a relaxed weak supersolution with source
\(c^+\).  Since the relaxed initial densities are equal,
Theorem~\ref{thm:app-HS-relaxed-comparison} gives
\[
    \rho(x,s)\leq\rho(x,s+t)
    \qquad\hbox{for a.e. }(x,t)\in Q_{\sigma}.
\]
  The continuity in \eqref{eq:HS-weak-time-traces} then gives this
inequality for each \(t\in(0,\sigma)\), after changing the exceptional set in
\(x\).  As
\(\sigma\) is arbitrary and \(\rho(\cdot,t)=\chi_{\Omega_t^{+}}\), this proves
\eqref{eq:HS-monotone-sets}.  Since this monotonicity is equivalent to
\(\partial_t\rho\geq0\) in distributions, \eqref{eq:HS-weak-equation} gives
\eqref{eq:HS-lower-laplacian-bound}.

We next prove the compact-support propagation in (i), under the additional
assumption that \(\Omega_0^{+}\) is bounded.  Choose \(r_0>0\) such that
\(\Omega_0^{+}\subset B_{r_0}\) up to a null set, and set
\begin{equation}\label{eq:HS-compact-support-barrier}
    R(t):=r_0e^{C_\tau t/d},\qquad
    \bar\rho(x,t):=\chi_{B_{R(t)}}(x),\qquad
    \bar p(x,t):=\frac{C_\tau}{2d}\bigl(R(t)^2-\abs{x}^2\bigr)_+ .
\end{equation}
\(\bar p\) is the radial classical Hele-Shaw solution with constant source
\(C_\tau\): indeed,
\(-\Delta\bar p=C_\tau\) in \(B_{R(t)}\), \(\bar p=0\) on
\(\partial B_{R(t)}\), and
\(R'(t)=C_\tau R(t)/d=\abs{\nabla\bar p}\) on \(\partial B_{R(t)}\).
Thus \(\bar p\), with density \(\bar\rho\), is a weak solution with source
\(C_\tau\).  Since
\(\rho(\cdot,0)\leq\bar\rho(\cdot,0)\) and \(c\leq C_\tau\), the comparison
principle (Theorem~\ref{thm:HS-distributional-comparison}) gives
\begin{equation}\label{eq:HS-compact-support-comparison}
    p\leq\bar p
    \qquad\hbox{a.e. in }\R^d\times(0,\tau).
\end{equation}
Since \(p\geq0\), \(\bar p=0\) outside \(B_{R(t)}\), and
\(\rho=\chi_{\{p>0\}}\), both \(\rho\) and \(p\) are supported in
\(B_{r_0e^{C_\tau\tau/d}}\) on \([0,\tau]\), up to null sets.

We now extract the slice identity \eqref{eq:HS-slice-variational} from the
spacetime identity in Lemma~\ref{lem:app-HS-elliptic-positive-phase}.  Fix a
time \(t\in(0,\tau)\) which is a right Lebesgue point of \(p\) in
\(H^1(\R^d)\) and of \(c\) in \(L^2(B_R)\) for every integer \(R\geq1\).
These times have full measure.  Let
\(\phi\in H^1(\R^d)\) be compactly supported and satisfy
\(\phi(1-\rho(\cdot,t))=0\).  By \eqref{eq:HS-monotone-sets},
\(\phi(1-\rho(\cdot,s))=0\) for a.e. \(s>t\).  Thus, for \(h>0\) small,
\[
    \eta_h(x,s):=\frac1h\mathbf 1_{(t,t+h)}(s)\phi(x)
\]
is admissible in
\eqref{eq:app-HS-elliptic-positive-phase-identity}, and hence
\[
    \frac1h\int_t^{t+h}
    \left(
    \int_{\R^d}\nabla p(\cdot,s)\cdot\nabla\phi\,dx
    -
    \int_{\R^d}c(\cdot,s)\phi\,dx
    \right)ds=0 .
\]
Letting \(h\downarrow0\) gives \eqref{eq:HS-slice-variational} for compactly
supported \(\phi\).  For a general
\(\phi\in H^1(\R^d)\) with \(\phi(1-\rho(\cdot,t))=0\), apply the compactly
supported case to \(\chi_R\phi\), where \(\chi_R\) is a standard spatial
cutoff, and let \(R\to\infty\).  The passage to the limit uses
\(\rho(\cdot,t)\in L^1(\R^d)\), \(c(\cdot,t)\in L^\infty(\R^d)\), and
\(p(\cdot,t),\phi\in H^1(\R^d)\).

We next prove the \(L^\infty\) bound in (i), which does not require
\(\Omega_0^{+}\) to be bounded.  For a.e. \(t\), testing
\eqref{eq:HS-slice-variational} with
\((p(\cdot,t)-k)_+\) gives
\begin{equation}\label{eq:HS-level-set-energy-bound}
    \int_{\{p(\cdot,t)>k\}}\abs{\nabla p(\cdot,t)}^2\,dx
    \leq
    C_\tau\int_{\{p(\cdot,t)>k\}}(p(\cdot,t)-k)\,dx .
\end{equation}
The standard De Giorgi level-set iteration applied to
\eqref{eq:HS-level-set-energy-bound} yields
\begin{equation}\label{eq:HS-slice-Linfty-bound}
    \norm{p(\cdot,t)}_{L^\infty(\R^d)}
    \leq C_d C_\tau\abs{\Omega_t^{+}}^{2/d}
    \leq C_d C_\tau
    \left(\sup_{0\leq s\leq\tau}\abs{\Omega_s^{+}}\right)^{2/d}
\end{equation}
for a.e. \(t\in(0,\tau)\).  The last quantity is finite because
\eqref{eq:HS-weak-regularity} holds locally in time, proving
\(p\in L^\infty(Q_{\tau})\).

Finally, to prove \eqref{eq:HS-slice-pressure-comparison}, fix
\(0<s<t \leq\tau\) and \(A\geq0\) satisfying
\eqref{eq:HS-source-slice-comparison}, and set
\[
    \zeta:=(p(\cdot,s)-Ap(\cdot,t))_+ .
\]
By \eqref{eq:HS-monotone-sets}, \(\zeta\) is an admissible test function in
\eqref{eq:HS-slice-variational} at both times.
Using \(\zeta\) at time \(s\) and \(A\zeta\) at time \(t\), then
subtracting, gives
\[
    \int_{\R^d}\abs{\nabla\zeta}^2\,dx
    =
    \int_{\R^d}\zeta\bigl(c(\cdot,s)-Ac(\cdot,t)\bigr)\,dx
    \leq0 .
\]
Thus \(\zeta\) is spatially constant and, since
\(\zeta\in L^2(\R^d)\), $\zeta \equiv 0$, which proves
\eqref{eq:HS-slice-pressure-comparison}.
\end{proof}

In the next lemma we consider a broader family of source conditions than
\eqref{eq:A-source-time}, with the time integrability \(L^1\) replaced by \(L^\beta\),
\(\beta\in[1,\infty]\). The resulting estimate on \((\partial_t p)^-\) is of independent
interest; in particular, the case \(\beta=\infty\) gives a pointwise lower bound on
\(\partial_t p\).

\begin{lem}[Lower estimate on \(\partial_t p\)]
\label{lem:HS-finite-difference-pt-lower}
Let \(p\) be a weak solution to \eqref{eq:intro-HS-law}. Suppose that for each $\tau>0$ there exists \(R_{\tau}>0\) such that \(\Omega_t^{+}\subset B_{R_\tau}\) for \(0\leq t<\tau\). If $c$ satisfies
\begin{equation}\label{eq:HS-source-negative-time-derivative}
    (\partial_t c)^-
    \in L^\beta_{\operatorname{loc}}
    ([0,\infty);L^m_{\operatorname{loc}}(\R^d)) \hbox{ for some } m\in(1,\infty],\ \beta\in[1,\infty],
\end{equation}
then, for every $\tau>0$,  \(\partial_t p\) is a signed Radon
measure on \(B_{R_{\tau}} \times(0,\tau)\), and its negative part satisfies
\begin{equation}\label{eq:HS-pt-negative-part-regularity}
    \norm{(\partial_t p)^-}_{L^\beta(0,\tau;L^{\bar m}(B_{R_{\tau}}))}
    \leq
    C\norm{(\partial_t c)^-}_{L^\beta(0,\tau;L^m(B_{R_{\tau}}))}.
\end{equation}
Here \(C=C(d, m,\bar m,\beta,R_\tau)\), and \(\bar m\) may be chosen as follows: \(\bar m=d m/(d-2 m)\) if \(1< m<d/2\), \(\bar m=\infty\) if \(m>d/2\), and \(\bar m\) is an arbitrary finite exponent if \(m=d/2\).
\end{lem}

\begin{proof}
Fix \(\tau>0\), set \(R:=R_\tau\), and let
\(I=(a,b)\Subset(0,\tau)\).  We first prove the estimate on
\(B_R\times I\); the conclusion on \(B_R\times(0,\tau)\) then follows by
exhausting \((0,\tau)\) by compact subintervals.
Set
\[
    w(x,t)=\int_{0}^tp(x,s)ds, \qquad \eta_c(x,t):=\int_0^t c(x,s)\rho(x,s)\,ds .
\]
Integrating \eqref{eq:HS-weak-equation} in time gives
\begin{equation}\label{eq:HS-Baiocchi-equation}
    \Delta w(\cdot,t)
    =
    \rho(\cdot,t)-\rho(\cdot,0)-\eta_c(\cdot,t)
    \qquad\hbox{in }\mathcal D'(\R^d).
\end{equation}
For \(0<h<(\tau-b)/2\), write
\[
    \delta_h f(t):=\frac{f(t+h)-f(t)}{h},
    \qquad
    u_h:=\delta_h^2w,
\]
and define
\begin{equation}\label{eq:HS-source-second-average}
    A_h(x,t):=\frac1{h^2}
    \left(\int_{t+h}^{t+2h}c(x,s)\,ds-\int_t^{t+h}c(x,s)\,ds\right).
\end{equation}
On \(\R^d\setminus\Omega_{t+h}^{+}\), monotonicity gives
\(w(\cdot,t)=w(\cdot,t+h)=0\) a.e., and therefore
\(u_h=h^{-2}w(\cdot,t+2h)\geq0\).  Thus
\(\{u_h<0\}\subset\Omega_{t+h}^{+}\subset B_R\) up to null sets.
By \eqref{eq:HS-monotone-sets}, for a.e. \(x\in\Omega_{t+h}^{+}\) one has
\(\delta_h^2\rho\leq0\) and \(\delta_h^2\eta_c\geq A_h\).  Hence
\begin{equation}\label{eq:HS-second-difference-elliptic-ineq}
    \Delta u_h\leq -A_h
    \qquad\hbox{a.e. on }\{u_h<0\}.
\end{equation}
Kato's inequality and
\eqref{eq:HS-second-difference-elliptic-ineq} give
\begin{equation}\label{eq:HS-negative-second-difference-kato}
    -\Delta(u_h)^-
    \leq (A_h)^-
    \qquad\hbox{in }B_R,
    \qquad
    (u_h)^-=0\hbox{ on }\partial B_R .
\end{equation}

Let \(V\) be the time-measurable function obtained by applying, for a.e.
\(t\), the zero-boundary Dirichlet solution operator on \(B_R\) to
\((\partial_t c)^-(\cdot,t)\).  Thus \(V(\cdot,t)\geq0\) and
\[
    -\Delta V(\cdot,t)=(\partial_t c)^-(\cdot,t)
    \quad\hbox{in }B_R,
    \qquad
    V(\cdot,t)=0
    \quad\hbox{on }\partial B_R .
\]
The standard zero-boundary Dirichlet \(W^{2,r}\) estimate
\cite[Thm.~9.15 and Lem.~9.17]{GilTru}, with \(r= m\) if
\(m<\infty\) and any fixed \(r>d/2\) if \(m=\infty\), followed by
the second-order Sobolev embedding into \(L^{\bar m}\), gives
\begin{equation}\label{eq:HS-dirichlet-potential-bound}
    \norm{V(\cdot,t)}_{L^{\bar m}(B_R)}
    \leq
    C\norm{(\partial_t c)^-(\cdot,t)}_{L^m(B_R)} .
\end{equation}
By \eqref{eq:HS-source-second-average}, \(A_h=K_h*_t\partial_t c\) in the
distributional time sense, where \(K_h\) is a nonnegative time kernel of
mass \(1\), supported in \([0,2h]\), and $*_t$ denotes convolution in time.
Since \(\partial_t c\geq-(\partial_t c)^-\), we have
\((A_h)^-\leq K_h*_t(\partial_t c)^-\).  For each \(t\), the function
\((K_h*_tV)(\cdot,t)\) has zero boundary data on \(\partial B_R\) and solves
\[
    -\Delta (K_h*_tV)(\cdot,t)
    =K_h*_t(\partial_t c)^-(\cdot,t)
    \qquad\hbox{in }B_R.
\]
Comparing this equation with \eqref{eq:HS-negative-second-difference-kato}
gives \((u_h)^-\leq K_h*_tV\), hence
\begin{equation}\label{eq:HS-second-difference-lower}
    u_h\geq -K_h*_t V
    \qquad\hbox{in }B_R\times I.
\end{equation}
Letting \(h\downarrow0\) in \eqref{eq:HS-second-difference-lower} gives
\begin{equation}\label{eq:HS-pt-lower-by-potential}
    \partial_t p\geq -V
    \qquad\hbox{in }\mathcal D'(B_R\times I),
\end{equation}
since \(\delta_h^2w\to \partial_t p\) in distributions and \(K_h*_tV\to V\) strongly
in \(L^\beta\) if \(\beta<\infty\), and weak-* if \(\beta=\infty\).
Since \(I\Subset(0,\tau)\) was arbitrary, \eqref{eq:HS-pt-lower-by-potential}
holds in \(\mathcal D'(B_R\times(0,\tau))\).  Thus
\(\partial_t p+V\) is a positive distribution, hence a Radon measure, and \(\partial_t p\) is a
signed Radon measure on \(B_R\times(0,\tau)\).  Moreover
\(0\leq(\partial_t p)^-\leq V\), and
\eqref{eq:HS-dirichlet-potential-bound} gives the asserted bound.
\end{proof}

\begin{prop}[Canonical upper semicontinuous representative]\label{prop:HS-USC-representative}
Let \(p\) be a weak solution to \eqref{eq:intro-HS-law}, and assume
\eqref{eq:A-source}, \eqref{eq:A-initial}, and \eqref{eq:A-source-time}.
Then \(p\) has an
upper semicontinuous representative on $Q$, still denoted
by \(p\), characterized by
\begin{equation}\label{eq:HS-forward-average-representative}
    p(x,t)=\lim_{r\downarrow0}
    \frac{1}{r^2\abs{B_r}}
    \int_t^{t+r^2}\int_{B_r(x)}p(y,s)\,dy\,ds  \,\, \hbox{ for } (x,t)\in Q.
\end{equation}
Moreover, for a.e. \(t>0\),
\begin{equation}\label{eq:HS-spatial-envelope-representative}
    p(x,t)=
    \lim_{r\downarrow0}
    \operatorname*{ess\,sup}_{B_r(x)}p(\cdot,t)
    \qquad\hbox{for every }x\in\R^d,
\end{equation}
and \(p\) is right-continuous in time.
\end{prop}

\begin{proof}
Fix \(\tau>0\) and set
\[
    C_\tau:=\norm{c}_{L^\infty(\R^d\times(0,\tau+1))}.
\]
Since \(p\) is compactly supported and \(m>d/2\),
Lemma~\ref{lem:HS-finite-difference-pt-lower} gives
\begin{equation}\label{eq:HS-representative-time-lower}
    \partial_t p\geq -g
    \qquad\hbox{in }\mathcal D'(\R^d\times(0,\tau+1)),
    \qquad
    g:=(\partial_t p)^-\in L^1(0,\tau+1;L^\infty(\R^d)).
\end{equation}
Set
\[
    \gamma_\tau(r):=
    \sup_{\substack{I\subset(0,\tau+1)\\ \abs{I}\leq r^2}}
    \int_I\norm{g(\cdot,s)}_{L^\infty(\R^d)}\,ds,
\]
so that \(\gamma_\tau(r)\to0\) as \(r\downarrow0\) by absolute continuity
of the integral.  We then have, for every ball \(B_r(x)\) with \(0<r<1\) and every
interval \((a,b)\subset(0,\tau+1)\) with \(b-a\leq r^2\),
\begin{equation}\label{eq:HS-representative-pt-minus-modulus}
    \int_a^b\frac1{\abs{B_r}}\int_{B_r(x)}g(y,s)\,dy\,ds
    \leq \gamma_\tau(r).
\end{equation}

For fixed \((x,t)\) and \(0<\ell\) with \(t+\ell^2<\tau+1\), set
\[
    \Phi(r,\ell):=
    \frac{1}{\ell^2\abs{B_r}}
    \int_t^{t+\ell^2}\int_{B_r(x)}p(y,s)\,dy\,ds .
\]
We compare \(\Phi(r',r')\) and \(\Phi(r,r)\) for \(0<r'<r\) with
\(t+r^2<\tau+1\).  First keep the shorter time interval and enlarge only
the spatial ball.  By \eqref{eq:HS-lower-laplacian-bound}, the function
\(z\mapsto p(z,s)+(C_\tau/2d)\abs{z-x}^2\) is subharmonic in \(z\) for a.e.
\(s\in(0,\tau+1)\).  The mean-value inequality, averaged over
\(s\in(t,t+(r')^2)\), gives
\begin{equation}\label{eq:HS-representative-spatial-comparison}
    \Phi(r',r')\leq \Phi(r,r')+C_\tau r^2 .
\end{equation}
It remains to compare the two time lengths for the fixed ball \(B_r(x)\).
Set
\begin{equation}\label{eq:HS-representative-lambda-def}
    \lambda:=\left(\frac{r'}{r}\right)^2,
    \qquad
    (r')^2=\lambda r^2 .
\end{equation}
With this choice, write
\[
    F(s):=\frac1{\abs{B_r}}\int_{B_r(x)}p(y,s)\,dy,
    \qquad
    G(s):=\frac1{\abs{B_r}}\int_{B_r(x)}g(y,s)\,dy .
\]
By \eqref{eq:HS-representative-time-lower}, after approximating
\(\chi_{B_r(x)}\) by nonnegative smooth spatial cutoffs,
\begin{equation}
    F'\geq -G
    \qquad\hbox{in }\mathcal D'(0,\tau+1).
\end{equation}
Thus, after changing
\(F\) on a null set, for every \(0<a<b<\tau+1\),
\begin{equation}\label{eq:HS-representative-averaged-time-integral}
    F(a)-F(b)\leq \int_a^b G(\sigma)\,d\sigma .
\end{equation}
The choice of \(\lambda\) gives the identity
\begin{equation}\label{eq:HS-representative-time-difference-identity}
    \Phi(r,r')-\Phi(r,r)
    =
    \frac1{r^2}\int_0^{r^2}\bigl(F(t+\lambda s)-F(t+s)\bigr)\,ds .
\end{equation}
Since \(t+\lambda s\leq t+s\), applying
\eqref{eq:HS-representative-averaged-time-integral} with
\(a=t+\lambda s\) and \(b=t+s\), then using
\eqref{eq:HS-representative-time-difference-identity}, yields
\begin{equation}\label{eq:HS-representative-time-comparison}
    \Phi(r,r')-\Phi(r,r)
    \leq
    \frac1{r^2}\int_0^{r^2}\int_{t+\lambda s}^{t+s}G(\sigma)\,d\sigma\,ds
    \leq \int_t^{t+r^2}G(\sigma)\,d\sigma
    \leq \gamma_{\tau}(r)
\end{equation}
Combining \eqref{eq:HS-representative-spatial-comparison} and
\eqref{eq:HS-representative-time-comparison} with
\(\omega_\tau(r):=C_\tau r^2+\gamma_\tau(r)\) gives
\begin{equation}\label{eq:HS-forward-average-almost-monotone}
    \Phi(r',r')\leq \Phi(r,r)+\omega_\tau(r)
    \qquad\hbox{whenever }0<r'<r<1\hbox{ and }t+r^2<\tau+1 .
\end{equation}
Since \(\omega_\tau(r)\to0\) and \(\Phi(r,r)\geq0\),
\eqref{eq:HS-forward-average-almost-monotone} implies that
\(\lim_{r\downarrow0}\Phi(r,r)\) exists for every
\((x,t)\in\R^d\times(0,\tau)\).  Lebesgue differentiation identifies this
limit with the original \(p\) a.e., and we still denote it by \(p\).

To establish upper semicontinuity, fix \(r_\tau>0\) with
\(t+r_\tau^2<\tau+1\).  Sending \(r'\downarrow0\) in
\eqref{eq:HS-forward-average-almost-monotone} and taking the infimum over
\(r\in(0,r_\tau)\) yields
\begin{equation}\label{eq:HS-USC-inf-representation}
    p(x,t)=\inf_{0<r<r_\tau}
    \left(
    \frac{1}{r^2\abs{B_r}}
    \int_t^{t+r^2}\int_{B_r(x)}p(y,s)\,dy\,ds
    +\omega_\tau(r)\right).
\end{equation}
For each fixed \(r\), the function inside the infimum in
\eqref{eq:HS-USC-inf-representation} is continuous in \((x,t)\) by
translation continuity in \(L^1_{\operatorname{loc}}\).  Hence \(p\) is upper
semicontinuous as an infimum of continuous functions.

Right continuity in \(t\) follows directly from
\eqref{eq:HS-representative-time-lower}: since \(g\in L^1(0,\tau+1;L^\infty(\R^d))\),
\eqref{eq:HS-representative-time-lower} passes to the representative as
\[
    p(x,t)\leq p(x,s)+\int_t^s\norm{g(\cdot,r)}_{L^\infty(\R^d)}\,dr,
    \qquad 0<t<s<\tau .
\]
Letting \(s\downarrow t\) gives \(\liminf_{s\downarrow t}p(x,s)\geq
p(x,t)\), and upper semicontinuity provides the reverse inequality on the
limsup.

It remains to prove the spatial envelope formula
\eqref{eq:HS-spatial-envelope-representative}.  Fix a time
\(t\in(0,\tau)\) such that the slice of the representative agrees a.e. with
the original slice and \(t\) is a right Lebesgue point of \(p\) as a map into
\(L^1(B_R)\) for every integer \(R\geq1\).  These times have full measure.
Let \(x\in\R^d\) and \(0<r<1\) be such that \(t+r^2<\tau+1\).  Applying
\eqref{eq:HS-representative-spatial-comparison} with base point \((x,t)\)
and radii \(0<r'<r\), we send \(r'\downarrow0\).  The representative formula
\eqref{eq:HS-forward-average-representative} gives
\begin{equation}\label{eq:HS-shrinking-parabolic-average-limit}
    \lim_{r'\downarrow0}\Phi(r',r')=p(x,t).
\end{equation}
Choose an integer \(R_x\) such that \(B_r(x)\subset B_{R_x}\).  Since \(t\)
is a right Lebesgue point of \(p\) in \(L^1(B_{R_x})\),
\begin{equation}\label{eq:HS-fixed-radius-time-average-limit}
    \lim_{r'\downarrow0}\Phi(r,r')
    =
    \frac1{\abs{B_r}}\int_{B_r(x)}p(y,t)\,dy .
\end{equation}
The limits \eqref{eq:HS-shrinking-parabolic-average-limit} and
\eqref{eq:HS-fixed-radius-time-average-limit} in
\eqref{eq:HS-representative-spatial-comparison} give
\begin{equation}\label{eq:HS-fixed-time-spatial-envelope-upper}
    p(x,t)
    \leq
    \frac1{\abs{B_r}}\int_{B_r(x)}p(y,t)\,dy+C_\tau r^2
    \leq
    \operatorname*{ess\,sup}_{B_r(x)}p(\cdot,t)+C_\tau r^2 .
\end{equation}
Letting \(r\downarrow0\) in
\eqref{eq:HS-fixed-time-spatial-envelope-upper} yields one inequality in
\eqref{eq:HS-spatial-envelope-representative}.  The reverse follows from the
upper semicontinuity of \(p\).
\end{proof}

\section{Nondegeneracy and Hopf--Lax estimates}\label{sec:HS-nondegeneracy-hopf-lax}
The nondegeneracy and hitting-time
estimates of \S\ref{subsec:HS-endpoint-nondegeneracy} use the endpoint case \(m=\infty\) of
\eqref{eq:A-source-time}, while the Hopf--Lax inequality of
\S\ref{subsec:HS-Hopf-Lax-finite-q} requires only \(m>d/2\).
For the remainder of the paper, \(p\) is
identified with the canonical upper semicontinuous and right-continuous
representative from Proposition~\ref{prop:HS-USC-representative}.

\subsection{Nondegeneracy}
\label{subsec:HS-endpoint-nondegeneracy}

For weak solutions with bounded initial patch, the slice
comparison \eqref{eq:HS-slice-pressure-comparison}, applied to
\eqref{eq:HS-forward-average-representative}, gives the pointwise persistence
of the canonical pressure:
\begin{equation}\label{eq:HS-pointwise-phase-persistence}
    p(x,s)>0\quad\Longrightarrow\quad p(x,t)>0
    \qquad\hbox{for }x\in\R^d,\quad 0<s\leq t .
\end{equation}
We note that, since $p$ is, in general, discontinuous, the set $\Omega_t^{+}=\{p(\cdot,t)>0\}$ is not open in general. For this reason, we henceforth consider an open set, defined in terms of a more regular variable, which coincides a.e. with $\Omega_t^{+}$. Namely, we recall from \eqref{eq:intro-baiocchi} that
\begin{equation}\label{eq:HS-open-phase-bridge}
    w(x,t)=\int_{0}^tp(x,s)ds, \qquad \Omega_t:=\{w(\cdot,t)>0\},
    \qquad
    T(x):=\inf\{t>0:x\in\Omega_t\}.
\end{equation}
By \eqref{eq:HS-Baiocchi-equation}, $\Delta w(\cdot,t)\in L^\infty(\R^d)$, so elliptic regularity gives that $w(\cdot,t)$ is continuous and $\Omega_t$ is open. Moreover,
\begin{equation}\label{eq:HS-open-phase-identities}
    \Omega_t=\{T<t\},
    \qquad
    p(x,t)>0\quad\hbox{for }x\in\Omega_t,
    \qquad
    \rho(\cdot,t)=\chi_{\Omega_t}\quad\hbox{a.e.}.
\end{equation}
Indeed, since $w(x,\cdot)$ is nondecreasing and continuous, the definition of $T$ gives $\Omega_t=\{T<t\}$. If $x\in\Omega_t$, then $w(x,t)>0$, so $p(x,s)>0$ for some $s<t$; hence \eqref{eq:HS-pointwise-phase-persistence} gives $p(x,t)>0$. Finally, the a.e. identity follows from the monotonicity in \eqref{eq:HS-monotone-sets} and the $L^1$-continuity of $\rho$. In pointwise time
integrals below, we use the representative \(\rho(x,t)=\chi_{\{T(x)<t\}}\).
Thus \(\Omega_t\) is the positive-time geometric phase, while
\(\Omega_0^{+}\) denotes the initial density patch.

\begin{lem}[Initial nondegeneracy]\label{lem:HS-initial-nondegeneracy}
Assume \eqref{eq:A-source}, \eqref{eq:A-initial}, \eqref{eq:A-interior-ball}, and
\eqref{eq:A-source-time}, and let
\(p\) be a weak solution to \eqref{eq:intro-HS-law}.  Let \(t_*>0\), and set
\(\underline c:=\inf_{\R^d\times[0,t_*]}c\).
Then
\begin{equation}\label{eq:HS-initial-nondegeneracy}
    \Omega_0^{+}+\overline B_\vep\subset\Omega_t
    \qquad\hbox{whenever }\vep>0
    \hbox{ and } \frac{d}{r_0 \underline c}\vep\leq t<t_* .
\end{equation}
\end{lem}

\begin{proof}
Let \(B_r(a)\subset\Omega_0^{+}\) be an open ball.  Define
\begin{equation}\label{eq:HS-explicit-ball-solution}
    P(x,t):=\frac{\underline c}{2d}
    \bigl(R(t)^2-\abs{x-a}^2\bigr)_+,
    \qquad
    R(t):=re^{\underline c t/d}.
\end{equation}
The function \(P\) is the explicit radial classical (and hence weak) solution to
\eqref{eq:intro-HS-law} with constant source \(\underline c\) and initial
phase \(B_r(a)\).  Since \(\underline c\)\(\leq c\) on
\([0,t_*]\),
Theorem~\ref{thm:HS-distributional-comparison} gives
\(P\leq p\) a.e. in \(\R^d\times(0,t_*)\).  Since \(P\) is continuous and
\(p\) is upper semicontinuous,
the inequality holds everywhere.  If \(0<s<t\), positivity of \(p(y,s)\) and
the right-continuity from Proposition~\ref{prop:HS-USC-representative} imply
\(w(y,t)>0\). Thus \(B_{re^{\underline cs/d}}(a)\subset\Omega_t\) for every \(0<s<t\);
letting \(s\uparrow t\) gives
\begin{equation}\label{eq:HS-ball-contained}
    B_{re^{\underline c t/d}}(a)\subset\Omega_t
    \qquad\hbox{whenever }B_r(a)\subset\Omega_0^{+}
    \hbox{ and }0<t<t_* .
\end{equation}
Setting
\begin{equation}\label{eq:HS-sigma-expansion-def}
    \sigma(t):=r_0(e^{\underline c t/d}-1)> \frac{r_0 \underline ct}{d},
\end{equation}
we conclude that
\begin{equation}\label{eq:HS-initial-phase-expands}
    \Omega_0^{+}+B_{\frac{r_0 \underline ct}{d}}\subset \Omega_0^{+}+B_{\sigma(t)}\subset\Omega_t
    \qquad\hbox{for }0<t<t_* ,
\end{equation}
which readily yields the claim.
\end{proof}

\begin{lem}[Comparison with a time-rescaled solution]\label{lem:HS-time-rescaled-comparison}
Let \(\tau,\vep>0\), let \(\Theta\in C^1([0,\tau])\) satisfy
\(\Theta'>0\), \(\Theta(0)>0\), and \(\Theta(\tau)\geq\tau\).
Assume \eqref{eq:A-source}, \eqref{eq:A-initial}, and \eqref{eq:A-source-time}.
Let \(p\) be a weak solution to
\eqref{eq:intro-HS-law} on \([0,\Theta(\tau))\).
Assume
\begin{equation}\label{eq:HS-source-Theta-compatibility}
    \operatorname*{ess\,sup}_{\abs{z-x}\leq\vep}c(z,t)
    \leq \Theta'(t)c(x,\Theta(t))
    \qquad\hbox{for a.e. }(x,t)\in\R^d\times(0,\tau),
\end{equation}
and
\begin{equation}\label{eq:HS-initial-Theta-compatibility}
    \Omega_0^{+}+\overline B_\vep\subset\Omega_{\Theta(0)}.
\end{equation}
Then
\begin{equation}\label{eq:HS-time-rescaled-pressure-comparison}
    \sup_{\abs{z-x}\leq\vep}p(z,t)
    \leq \Theta'(t)p(x,\Theta(t))
    \qquad\hbox{for }x\in\R^d,\quad 0<t<\tau .
\end{equation}
\end{lem}

\begin{proof}
Fix \(h\in\overline B_\vep\) and set
\[
    p_h(x,t):=p(x+h,t),
    \qquad
    \rho_h(x,t):=\rho(x+h,t),
    \qquad
    c_h(x,t):=c(x+h,t).
\]
Then \(p_h\), restricted to \([0,\tau)\), is a weak solution to
\eqref{eq:intro-HS-law} with source \(c_h\) on \([0,\tau)\).  Define
\begin{equation}\label{eq:HS-bar-p-time-rescale}
    \bar p(x,t):=\Theta'(t)p(x,\Theta(t)),
    \qquad
    \bar c(x,t):=\Theta'(t)c(x,\Theta(t)),
    \qquad
    \bar\rho(x,t):=\rho(x,\Theta(t)).
\end{equation}
  The regularity and \(L^1\)-time continuity in
\eqref{eq:HS-weak-regularity}--\eqref{eq:HS-weak-time-traces} are preserved by
the continuous increasing reparametrization \(\Theta\).  Since \(\Theta\) is
strictly increasing, the change of variables \(s=\Theta(t)\) in
\eqref{eq:HS-weak-equation} shows that
\begin{equation}\label{eq:HS-bar-p-weak-equation}
    \partial_t\bar\rho-\Delta\bar p=\bar c\,\bar\rho
    \qquad\hbox{in }\mathcal D'(\R^d\times(0,\tau)).
\end{equation}
Together with the trace \(\bar\rho(\cdot,0)=\rho(\cdot,\Theta(0))\), this says
that \(\bar p\) is a weak solution
to \eqref{eq:intro-HS-law} with source \(\bar c\) on \([0,\tau)\).  By
\eqref{eq:HS-source-Theta-compatibility}, \(c_h\leq\bar c\), and by
\eqref{eq:HS-initial-Theta-compatibility},
\[
    \rho_h(\cdot,0)
    =\chi_{\Omega_0^{+}-h}\leq\chi_{\Omega_{\Theta(0)}}
    =\bar\rho(\cdot,0)\quad\hbox{a.e.},
\]
Theorem~\ref{thm:HS-distributional-comparison} gives \(p_h\leq\bar p\) a.e. in
\(\R^d\times(0,\tau)\).  Since \(p_h\) is represented by forward averages as in
\eqref{eq:HS-forward-average-representative} and \(\bar p\) is upper
semicontinuous, averaging this a.e. inequality over forward cylinders and
letting the radius vanish gives
\[
    p(x+h,t)\leq\Theta'(t)p(x,\Theta(t))
    \qquad\hbox{for }x\in\R^d,\quad 0<t<\tau .
\]
Taking the supremum over \(\abs{h}\leq\vep\) gives
\eqref{eq:HS-time-rescaled-pressure-comparison}.
\end{proof}

We now discuss how to choose $\Theta$ for given $c$ such that the compatibility condition \eqref{eq:HS-source-Theta-compatibility} is satisfied.
\begin{lem}[Source-adapted time rescaling]\label{lem:HS-source-envelope}
Let \(\tau>0\), and assume \eqref{eq:A-source} and \eqref{eq:A-source-time} with
\(m=\infty\).
Set
\begin{equation}\label{eq:HS-source-envelope-constants}
    M_x:=\norm{\nabla_x\log c}_{L^\infty(\R^d\times(0,\tau+1))},
    \qquad
    k(t):=\norm{(\partial_t\log c)^-(\cdot,t)}_{L^\infty(\R^d)},
    \qquad
    K(s):=\int_0^s k(r)\,dr .
\end{equation}
For
\(\delta,\vep>0\), set
\[
    H(s):=\int_0^s e^{-K(\sigma)}\,d\sigma
\]
and assume
\begin{equation}\label{eq:HS-envelope-size-condition}
    H(\delta)+e^{M_x\vep}H(\tau)\leq H(\tau+1).
\end{equation}
Define
\begin{equation}\label{eq:HS-Theta-def}
    \Theta(t):=
    H^{-1}\bigl(H(\delta)+e^{M_x\vep}H(t)\bigr),
    \qquad 0\leq t\leq \tau .
\end{equation}
Then \(\Theta\) is well-defined, \(\Theta(\tau)\leq \tau+1\), and
\begin{equation}\label{eq:HS-Theta-ode}
    \Theta'(t)=e^{M_x\vep+K(\Theta(t))-K(t)},
    \qquad
    \Theta(0)=\delta.
\end{equation}
Moreover, \(\Theta(t)>t\),
\eqref{eq:HS-source-Theta-compatibility} holds on \([0,\tau]\), and
\begin{equation}\label{eq:HS-Theta-lag-bound}
    \Theta(t)-t
    \leq e^{K(\tau+1)}
    \bigl(\delta+(e^{M_x\vep}-1)t\bigr)
    \qquad\hbox{for }0\leq t\leq \tau .
\end{equation}
\end{lem}

\begin{proof}
The function \(H\) is \(C^1\) and strictly increasing.  Condition
\eqref{eq:HS-envelope-size-condition} guarantees that
\eqref{eq:HS-Theta-def} is defined on \([0,\tau]\) and that
\(\Theta(\tau)\leq \tau+1\).  Differentiating \eqref{eq:HS-Theta-def} gives
\[
    \Theta'(t)=e^{M_x\vep}\frac{H'(t)}{H'(\Theta(t))}
    =e^{M_x\vep+K(\Theta(t))-K(t)},
\]
which is \eqref{eq:HS-Theta-ode}.  Since
\(H(\Theta(t))=H(\delta)+e^{M_x\vep}H(t)>H(t)\), we have \(\Theta(t)>t\).

To verify \eqref{eq:HS-source-Theta-compatibility},
\eqref{eq:HS-source-envelope-constants} gives
\[
    \operatorname*{ess\,sup}_{\abs{z-x}\leq\vep}c(z,t)
    \leq e^{M_x\vep}c(x,t).
\]
Since \(\Theta(t)>t\), \eqref{eq:HS-source-envelope-constants} gives
\[
    \log c(x,t)-\log c(x,\Theta(t))
    \leq \int_t^{\Theta(t)}k(r)\,dr
    =K(\Theta(t))-K(t).
\]
Exponentiating and combining with the spatial bound gives
\begin{equation}\label{eq:HS-source-multiplicative-control}
    \operatorname*{ess\,sup}_{\abs{z-x}\leq\vep}c(z,t)
    \leq e^{M_x\vep+K(\Theta(t))-K(t)}c(x,\Theta(t)).
\end{equation}
Using \eqref{eq:HS-Theta-ode} proves \eqref{eq:HS-source-Theta-compatibility}.

Let \(\ell(t):=\Theta(t)-t\).  From \eqref{eq:HS-Theta-def},
\begin{equation}\label{eq:HS-Theta-lag-H-identity}
    H(t+\ell(t))-H(t)=H(\delta)+(e^{M_x\vep}-1)H(t).
\end{equation}
Since \(K\) is nondecreasing, \(H(s)\leq s\) and
\[
    H(t+\ell)-H(t)
    =\int_t^{t+\ell}e^{-K(\sigma)}\,d\sigma
    \geq \ell e^{-K(t+\ell)}.
\]
Combining this with \eqref{eq:HS-Theta-lag-H-identity} gives
\begin{equation}\label{eq:HS-Theta-lag-prebound}
    \ell(t)
    \leq e^{K(t+\ell(t))}
    \bigl(\delta+(e^{M_x\vep}-1)t\bigr).
\end{equation}
Since \(t+\ell(t)\leq \tau+1\), \eqref{eq:HS-Theta-lag-prebound} proves
\eqref{eq:HS-Theta-lag-bound}.
\end{proof}

\begin{proof}[Proof of Theorem~\ref{thm:HS-hitting-time-lipschitz}]
Fix \(\tau>0\).  Set
\begin{equation}\label{eq:HS-Lip-proof-M-k}
    M_x:=\norm{\nabla_x\log c}_{L^\infty(\R^d\times(0,\tau+1))},
    \qquad
    k(t):=\norm{(\partial_t\log c)^-(\cdot,t)}_{L^\infty(\R^d)},
    \qquad
    K(s):=\int_0^s k(r)\,dr .
\end{equation}
By \eqref{eq:A-source} and \eqref{eq:A-source-time} with \(m=\infty\),
we have \(M_x<\infty\) and
\(k\in L^1(0,\tau+1)\).
Lemma~\ref{lem:HS-initial-nondegeneracy}, applied with \(t_*=1\), gives a
constant \(C_0>0\) such that
\begin{equation}\label{eq:HS-initial-compat-from-nondeg}
    \Omega_0^{+}+\overline B_\vep\subset\Omega_\delta
    \qquad\hbox{if }\vep>0
    \hbox{ and } C_0\vep\leq\delta<1 .
\end{equation}

Choose \(\vep_0\in (0,1)\) so that, whenever \(0<\vep<\vep_0\),
\eqref{eq:HS-envelope-size-condition} holds with \(\delta=C_0\vep\), and also
\(C_0\vep<1\).  For such \(\vep\), let \(\Theta\) be the
function from Lemma~\ref{lem:HS-source-envelope} with \(\delta=C_0\vep\).  Then
\eqref{eq:HS-initial-compat-from-nondeg} gives
\eqref{eq:HS-initial-Theta-compatibility}.  Lemma~\ref{lem:HS-source-envelope}
also gives \eqref{eq:HS-source-Theta-compatibility}, so
Lemma~\ref{lem:HS-time-rescaled-comparison} yields
\begin{equation}\label{eq:HS-pressure-spatial-time-comparison}
    \sup_{\abs{z-x}\leq\vep}p(z,t)
    \leq \Theta'(t)p(x,\Theta(t))
    \qquad\hbox{for }x\in\R^d,\quad 0<t<\tau.
\end{equation}
The lag bound \eqref{eq:HS-Theta-lag-bound} gives
\begin{equation}\label{eq:HS-Theta-linear-lag}
    \Theta(t)-t
    \leq e^{K(\tau+1)}
    \bigl(C_0+M_xe^{M_x}\tau\bigr)\vep
    =:L\vep .
\end{equation}

We first prove the one-sided estimate
\begin{equation}\label{eq:HS-T-one-sided-estimate}
    T(x)\leq T(y)+L\abs{x-y}
\end{equation}
whenever \(T(y)<\tau\) and \(\abs{x-y}<\vep_0\).  The case \(x=y\) is
trivial, so set \(\vep:=\abs{x-y}>0\).  Applying
\eqref{eq:HS-pressure-spatial-time-comparison} with base point \(x\), with
\(z=y\) and \(t=T(y)+\eta<\tau\), with \(\eta<1\) if \(T(y)=0\), gives
\[
    p(y,T(y)+\eta)\leq \Theta'(T(y)+\eta)p(x,\Theta(T(y)+\eta)).
\]
By \eqref{eq:HS-open-phase-identities}, \(y\in\Omega_{T(y)+\eta}\)
implies that the left-hand side is positive.  Hence
\(p(x,\Theta(T(y)+\eta))>0\). By the right-continuity from
Proposition~\ref{prop:HS-USC-representative} and \eqref{eq:HS-open-phase-identities},
\[
    T(x)\leq\Theta(T(y)+\eta).
\]
Letting \(\eta\downarrow0\) and using \eqref{eq:HS-Theta-linear-lag} proves
\eqref{eq:HS-T-one-sided-estimate}.

Since \(\tau>0\) was arbitrary, the family of estimates
\eqref{eq:HS-T-one-sided-estimate} readily gives the extended continuity of \(T\).  If \(\min\{T(x),T(y)\}<\tau\) and \(\abs{x-y}<\vep_0\), we may assume
\(T(y)\leq T(x)\); then \eqref{eq:HS-T-one-sided-estimate} gives
\(T(x)\leq T(y)+L\abs{x-y}<\infty\), which is
\eqref{eq:HS-T-local-Lipschitz-estimate}.  In particular \(\set{T<\infty}\) is
open.  Given \(x_0\) with \(T(x_0)<\tau\), every pair
\(x,y\in B_\rho(x_0)\) with
\(\rho:=\tfrac12\min\{\vep_0,(\tau-T(x_0))/L\}\) satisfies
\(\min\{T(x),T(y)\}\leq T(x_0)+L\rho<\tau\) and \(\abs{x-y}<\vep_0\), so
\(T\) is Lipschitz on \(B_\rho(x_0)\); this proves the local Lipschitz
assertion.

It follows from \eqref{eq:HS-pointwise-phase-persistence} and \eqref{eq:HS-open-phase-bridge} that for $x\in O$, we have $p(x,t) = 0$ for all $t < T(x)$ and $p(x,t) > 0$ for all $t > T(x)$. This immediately gives $(x, T(x))\in \partial \{ p > 0\}$. On the other hand, once we have that $T$ is continuous it follows that for any $x_0\in O$ and $t_0\neq T(x_0)$, there exists a spacetime ball around $(x_0,t_0)$ in which either $T < t$ (and thus $p>0$) or $T > t$ (and thus $p=0$).

Finally, let $x\in\partial\Omega_t$ and $0<r<r_*$.  If
$\operatorname{dist}(x,\Omega_0^{+})<r/2$, comparison with the tangent-ball solution
\eqref{eq:HS-explicit-ball-solution} gives \eqref{eq:HS-pressure-linear-nondegeneracy}.
Otherwise $B_{r/2}(x)\cap\Omega_0^{+}=\emptyset$, while $T(x)=t$ and
\eqref{eq:HS-T-local-Lipschitz-estimate} gives $t-T\leq Lr/2$ on
$B_{r/2}(x)\cap\Omega_t$.  After decreasing $r_*$, the source in
\eqref{eq:intro-w-obstacle} is at least $1/2$ there, so
Lemma~\ref{obst-quadratic-growth}\ref{item:obst-nondegeneracy} and
\eqref{eq:HS-slice-pressure-comparison} give
$r^2/(16d)\leq (\overline c_\tau/\underline c_\tau)Lr
\sup_{B_r(x)}p(\cdot,t)$, proving
\eqref{eq:HS-pressure-linear-nondegeneracy}.
\end{proof}

\begin{rem}\label{rem:HS-pointwise-pressure-equation}
Under \eqref{eq:A-source}, \eqref{eq:A-initial}, and
\eqref{eq:A-source-time}, the
elliptic equation for the pressure holds at every time on the open phase: if
\(t>0\), then
\begin{equation}\label{eq:HS-pointwise-pressure-equation}
    -\Delta p(\cdot,t)=c(\cdot,t)
    \qquad\hbox{in }\mathcal D'(\Omega_t).
\end{equation}
To prove \eqref{eq:HS-pointwise-pressure-equation}, let
\(\phi\in C_c^\infty(\Omega_t)\).  By
\eqref{eq:HS-open-phase-identities}, \(\supp\phi\subset\Omega_s\) for
every \(s>t\). Choose times \(s_j\downarrow t\) at which
\eqref{eq:HS-slice-variational} holds. Then
\[
    -\int_{\R^d}p(x,s_j)\Delta\phi(x)\,dx
    =
    \int_{\R^d}c(x,s_j)\phi(x)\,dx .
\]
The right-continuity of \(p\) from
Proposition~\ref{prop:HS-USC-representative}, the local boundedness of \(p\),
and the continuity of \(c\) allow \(j\to\infty\), giving
\eqref{eq:HS-pointwise-pressure-equation}.
\end{rem}

\subsection{Hopf--Lax inequality}
\label{subsec:HS-Hopf-Lax-finite-q}
In this subsection we prove Theorem~\ref{prop:HS-pressure-hopf-lax}. The proof
is a comparison argument in which the barrier is built from a single time slice of the
pressure.

\begin{proof}[Proof of Theorem~\ref{prop:HS-pressure-hopf-lax}]
Set \(h:=y-x\), and for \(t\in[t_0,t_1]\) define
\begin{equation}\label{eq:HS-Hopf-Lax-A-C}
    A(t):= A(|h|,t),
    \qquad
    C(t):=
    \frac{\abs{h}^2}{4(t_1-t_0)^2}\int_{t_0}^t e^{A(r)}\,dr,
\end{equation}
and
\begin{equation}\label{eq:HS-Hopf-Lax-subsolution}
    \bar p(z,t):=
    e^{-A(t)}
    \left(p\left(z+\frac{t-t_0}{t_1-t_0}h,t_0\right)-C(t)\right)_+ .
\end{equation}
 The key
observation is that formally \(\bar p\) is a weak subsolution
to \eqref{eq:intro-HS-law}, with source \(c\), on \((t_0,t_1]\), (the free boundary condition $\bar{p}_t \leq |D\bar{p}|^2$ can be checked via the equality \eqref{eq:HS-Hopf-Lax-square}). Since \(\bar p(\cdot,t_0)=p(\cdot,t_0)\),
Theorem~\ref{thm:HS-distributional-comparison} formally yields \(\bar p\leq p\)
a.e.\ in \(\R^d\times[t_0,t_1]\).
Evaluating this inequality at $(x,t_1)$ yields \eqref{eq:HS-pressure-hopf-lax}.

Making this simple observation rigorous turns out to be delicate due to the low regularity of \(\bar p\), since it is constructed only from
the single time slice \(p(\cdot,t_0)\).  We thus initially restrict ourselves to proving the comparison for ``good'' time slices.

We say \(t_0\) is {\it admissible} if the following collection of conditions holds:
\begin{equation}\label{eq:HS-good-slice-H1}
    p(\cdot,t_0)\in H^1(\R^d),
\end{equation}
the slice identity \eqref{eq:HS-slice-variational} holds at \(t=t_0\),
\begin{equation}\label{eq:HS-slice-elliptic-inequality}
    -\Delta p(\cdot,t_0)
    \leq c(\cdot,t_0)\chi_{\Omega_{t_0}}
    \qquad\hbox{in }\mathcal D'(\R^d),
\end{equation}
and
\begin{equation}\label{eq:HS-good-slice-envelope}
    p(y,t_0)=\lim_{r\downarrow0}
    \operatorname*{ess\,sup}_{B_r(y)}p(\cdot,t_0)
    \qquad\hbox{for every }y\in\R^d .
\end{equation}
These conditions hold for a.e.\ \(t_0\).  Indeed,
\eqref{eq:HS-good-slice-H1} follows from \eqref{eq:HS-weak-regularity}, and
\eqref{eq:HS-slice-variational} holds for a.e.\ \(t_0\) by
Lemma~\ref{lem:HS-basic-properties}.  
\eqref{eq:HS-slice-elliptic-inequality} follows from the monotonicity
\eqref{eq:HS-monotone-sets}, since it gives \(\partial_t\rho\geq0\) in distributions and thus
\eqref{eq:HS-weak-equation} yields \(-\Delta p\leq c\rho\) in space-time.  Lastly the
 envelope property \eqref{eq:HS-good-slice-envelope} holds for a.e.\
\(t_0\) by \eqref{eq:HS-spatial-envelope-representative}.

It suffices to prove \eqref{eq:HS-pressure-hopf-lax} for admissible \(t_0\). Indeed the estimate for an arbitrary
\(t_0\) follows  by applying it at admissible
times \(s_j\downarrow t_0\) and passing \(j\to\infty\), using the
right-continuity of \(p\) from Proposition~\ref{prop:HS-USC-representative}
and the absolute continuity of the integral in
\eqref{eq:HS-Hopf-Lax-Lambda-q}.

Hence we assume that $t_0$ is admissible. Let us first compare the source terms for $\bar{p}$ and $p$.
For each
\(t\in[t_0,t_1]\), define
\begin{equation}\label{eq:HS-Hopf-Lax-G-def}
    G(z,t):=(c(z,t_0)-c(z,t))_+ .
\end{equation}
Since both \(c(\cdot,t_0)\) and \(c(\cdot,t)\) are \(L\)-Lipschitz,
\begin{equation}\label{eq:HS-Hopf-Lax-G-Lipschitz}
    \abs{G(z,t)-G(w,t)}\leq 2L\abs{z-w}
    \qquad\hbox{for }z,w\in\R^d .
\end{equation}
Since \(\partial_t c\geq-(\partial_t c)^-\) in the distributional sense,
\begin{equation}\label{eq:HS-Hopf-Lax-G-pointwise-time}
    G(z,t)\leq\int_{t_0}^t(\partial_s c)^-(z,s)\,ds
    \qquad\hbox{for a.e. }z\in\R^d .
\end{equation}
Since \(x,y\in B_{R_*}\), we have \(\abs h\leq 2R_*\).  As
\(\bar p(z,t)>0\) forces \(z+\frac{t-t_0}{t_1-t_0}h\in\overline{\Omega_{t_0}}\subset B_{R_*}\),
the barrier \(\bar p(\cdot,t)\) is supported in \(B_{3R_*}\); it thus suffices to
bound \(G\) on \(B_{3R_*}\).  When \(d/2< m<\infty\), Minkowski's inequality
applied to \eqref{eq:HS-Hopf-Lax-G-pointwise-time} over \(B_{4R_*}\) gives
\(\norm{G(\cdot,t)}_{L^m(B_{4R_*})}\leq\Phi(t)\).  Since \(G(\cdot,t)\) is
\(2L\)-Lipschitz by \eqref{eq:HS-Hopf-Lax-G-Lipschitz} and \(B_r(z)\subset B_{4R_*}\)
for \(z\in B_{3R_*}\), \(r\in(0,1]\),
\[
    G(z,t)\leq \frac1{\abs{B_r}}\int_{B_r(z)}G(\xi,t)\,d\xi+2Lr
    \leq c_d\,r^{-d/ m}\norm{G(\cdot,t)}_{L^m(B_{4R_*})}+2Lr ;
\]
optimizing over \(r\in(0,1]\) gives
\[
    \norm{G(\cdot,t)}_{L^\infty(B_{3R_*})}
    \leq
    C_{d, m}\left(L^{\frac{d}{d+ m}}\Phi(t)^{\frac{m}{d+ m}}+\Phi(t)\right),
\]
so \(\underline c^{-1}\norm{G(\cdot,t)}_{L^\infty(B_{3R_*})}\leq\Lambda(t)\).  When
\(m=\infty\), \eqref{eq:HS-Hopf-Lax-G-pointwise-time} gives
\(\norm{G(\cdot,t)}_{L^\infty(B_{4R_*})}\leq\Phi(t)\) directly, so again
\(\underline c^{-1}\norm{G(\cdot,t)}_{L^\infty(B_{3R_*})}\leq\Lambda(t)\).  Recalling
\eqref{eq:HS-Hopf-Lax-G-def} and \eqref{eq:HS-Hopf-Lax-cinf-L-M},  we obtain
\[
    \log c\left(z+\frac{t-t_0}{t_1-t_0}h,t_0\right)-\log c(z,t)
    \leq M\frac{t-t_0}{t_1-t_0}\abs{h}+\underline c^{-1}G(z,t)
    \leq M\frac{t-t_0}{t_1-t_0}\abs{h}+\Lambda(t)
    =A(t)
\]
for a.e.\ \(z\in B_{3R_*}\).  Exponentiating gives
\begin{equation}\label{eq:HS-Hopf-Lax-source-control}
    e^{-A(t)}
    c\left(z+\frac{t-t_0}{t_1-t_0}h,t_0\right)
    \leq c(z,t)
    \qquad\hbox{for a.e. }(z,t)\in B_{3R_*}\times[t_0,t_1].
\end{equation}

We now proceed to show that $\bar{p}$ is a weak subsolution of \eqref{eq:intro-HS-law} with source $c$ in $\R^d\times [t_0,t_1]$.
Note that $\bar{p}(\cdot,t)>0$ has uniform compact support in \(t\in[t_0,t_1]\) due to Lemma~\ref{lem:HS-basic-properties}, hence
\(\bar\rho:=\chi_{\{\bar p>0\}} \in L^\infty(t_0,t_1;L^1(\R^d))\), while
\eqref{eq:HS-good-slice-H1} gives the \(H^1\) regularity of \(\bar p\).  It thus remains to verify the subsolution inequality.

Let \(P(z,t):=p\left(z+\frac{t-t_0}{t_1-t_0}h,t_0\right)\) and
\(\ell(z,t):=P(z,t)-C(t)\), so that $\bar{p}(t)=e^{-A(t)}\ell(t)_+$. Then
\(\ell_t=\frac{h}{t_1-t_0}\cdot\nabla P-C'(t)\) in $\mathcal{D}'$.
For a.e.\ \(t\), differentiating \eqref{eq:HS-Hopf-Lax-A-C} gives
\(C'(t)=\frac{\abs{h}^2e^{A(t)}}{4(t_1-t_0)^2}\), and completing the square yields
\begin{equation}\label{eq:HS-Hopf-Lax-square}
    e^{-A(t)}
    \left(\ell_t-e^{-A(t)}\abs{\nabla P}^2\right)
    =
    -\left|e^{-A(t)}\nabla P-\frac{h}{2(t_1-t_0)}\right|^2 .
\end{equation}
Now we introduce a mollified version of $\bar{p}$ and $\bar{\rho}$.
Let
\(\eta_\vep\in C^\infty(\R;[0,1])\) be nondecreasing, with
\(\eta_\vep=0\) on \((-\infty,0]\) and \(\eta_\vep=1\) on
\([\vep,\infty)\), and set
\[
    \beta_\vep(s):=\int_0^s\eta_\vep(a)\,da,\qquad
    \bar p_\vep(t):=e^{-A(t)}\beta_\vep(\ell(t)),\qquad
    \chi_\vep:=\eta_\vep(\ell).
\]

We will compute the subsolution-like inequality using $\bar p_\e$ and $\chi_\e$, and will send $\e\to 0$ to conclude. Fix \(t_0\leq a<b<t_1\).  For every nonnegative
\(\phi\in C_c^\infty(\R^d\times[a,b))\), test
\eqref{eq:HS-slice-elliptic-inequality} at time \(t_0\), for each
\(t\in[a,b]\), with
\[
    \zeta_t(x):=
    e^{-A(t)}
    \phi(z,t)
    \chi_\vep(z,t),\qquad z=z_{h,t} := x-\frac{t-t_0}{t_1-t_0}h.
\]
Note that $\zeta_t(x)\in H^1(\R^d)$, due to
\eqref{eq:HS-good-slice-H1}.  
After applying
\eqref{eq:HS-slice-elliptic-inequality}, changing variables
\(z=x-\frac{t-t_0}{t_1-t_0}h\) and integrating in \(t\in[a,b]\) and using \eqref{eq:HS-Hopf-Lax-source-control},  we obtain
\begin{equation}\label{eq:HS-Hopf-Lax-regularized-elliptic}
    \int_a^b\!\!\int_{\R^d}\nabla\phi\cdot\nabla \bar p_\vep\,dzdt
    \leq
    \int_a^b\!\!\int_{\R^d}\phi\,c\,\chi_\vep\,dzdt 
      -\int_a^b\!\!\int_{\R^d}
    e^{-A(t)}\phi\,\eta_\vep'(\ell)\abs{\nabla P}^2\,dzdt .
\end{equation}
  The chain rule in the time variable
gives, since \(\phi(\cdot,b)=0\),
\begin{equation}\label{eq:HS-Hopf-Lax-regularized-time}
    -\int_a^b\!\!\int_{\R^d}\chi_\vep\,\partial_t\phi\,dzdt
    =
    \int_{\R^d}\phi(z,a)\chi_\vep(z,a)\,dz
    +\int_a^b\!\!\int_{\R^d}
    \phi\,\eta_\vep'(\ell)\ell_t\,dzdt .
\end{equation}
Adding \eqref{eq:HS-Hopf-Lax-regularized-elliptic} and
\eqref{eq:HS-Hopf-Lax-regularized-time} leaves the boundary-condition term
\begin{equation}\label{eq:HS-Hopf-Lax-regularized-subsolution}
\begin{aligned}
    \int_a^b\!\!\int_{\R^d}
    \nabla\phi\cdot\nabla \bar p_\vep-\chi_\vep\,\partial_t\phi\,dzdt
    &\leq
    \int_{\R^d}(\phi\chi_\vep)(z,a)\,dz +\int_a^b\!\!\int_{\R^d}\phi\,c\chi_\vep\,dzdt \\
    &\quad
    +\int_a^b\!\!\int_{\R^d}
    \phi\,\eta_\vep'(\ell)
    \left(\ell_t-e^{-A(t)}\abs{\nabla P}^2\right)\,dzdt .
\end{aligned}
\end{equation}
Note that the last integrand is nonpositive due to  \eqref{eq:HS-Hopf-Lax-square}.
Lastly we send $\e\to 0$, using dominated convergence theorem, to obtain
\(\bar p_\vep\to \bar p\) in \(L^2(a,b;H^1(\R^d))\) and
\(\chi_\vep\to\bar\rho\) in \(L^1(\R^d\times(a,b))\) and we conclude that
 for every nonnegative
\(\phi\in C_c^\infty(\R^d\times[a,b))\),
\begin{equation}\label{eq:HS-Hopf-Lax-bar-p-distributional-subsolution}
    \int_a^b\!\!\int_{\R^d}
    \nabla\phi\cdot\nabla \bar p-\bar\rho\,\partial_t\phi\,dzdt
    \leq
    \int_{\R^d}\phi(z,a)\bar\rho(z,a)\,dz
    +\int_a^b\!\!\int_{\R^d}\phi\,c\bar\rho\,dzdt .
\end{equation}

It remains to verify the \(L^1\)-continuity condition
\eqref{eq:HS-weak-time-traces}.  We first claim that
\begin{equation}\label{eq:HS-Hopf-Lax-positive-levels-null}
    \abs{\{p(\cdot,t_0)=a\}}=0
    \qquad\hbox{for every }a>0 .
\end{equation}
By \eqref{eq:HS-weak-time-traces} and \eqref{eq:HS-monotone-sets}, the monotone
representative of the density satisfies
\(\rho(\cdot,t_0)=\chi_{\{w(\cdot,t_0)>0\}}\) a.e.: 
Integrating
\eqref{eq:HS-weak-equation} in time gives a locally bounded right-hand side for
\(\Delta w(\cdot,t_0)\), so elliptic regularity makes \(w(\cdot,t_0)\)
continuous.  If \(\abs{\{p(\cdot,t_0)=a\}}>0\) for some
\(a>0\), choose a ball \(B\Subset\{w(\cdot,t_0)>0\}\) with
\(\abs{B\cap\{p(\cdot,t_0)=a\}}>0\); there \(\rho(\cdot,t_0)=1\) a.e., so
\eqref{eq:HS-slice-variational} gives \(-\Delta p(\cdot,t_0)=c(\cdot,t_0)\) and
hence \(p(\cdot,t_0)\in H^2(B)\) by elliptic regularity.  The second derivatives
of an \(H^2\) function vanish a.e.\ on each of its level sets, so
\(\Delta p(\cdot,t_0)=0\) a.e.\ on \(B\cap\{p(\cdot,t_0)=a\}\), contradicting
\(-\Delta p(\cdot,t_0)=c(\cdot,t_0)>0\).  This proves
\eqref{eq:HS-Hopf-Lax-positive-levels-null}.

Since \(p(\cdot,t_0)\) is compactly supported and
\eqref{eq:HS-Hopf-Lax-positive-levels-null} holds, the map
\(a\mapsto\chi_{\{p(\cdot,t_0)>a\}}\) is continuous from \([0,\infty)\) into
\(L^1(\R^d)\); continuity at \(a=0\) follows by monotone convergence.  Together
with \(L^1\)-continuity of translations, this gives
\(\bar\rho\in C([t_0,t_1);L^1(\R^d))\).

We have verified that \(\bar p\) is a weak subsolution in $\R^d\times [t_0,t_1]$. Since
\(\bar p(\cdot,t_0)=p(\cdot,t_0)\), the comparison principle
Theorem~\ref{thm:HS-distributional-comparison} therefore yields
\begin{equation}\label{eq:HS-Hopf-Lax-ae-comparison}
    \bar p\leq p
    \qquad\hbox{a.e. in }\R^d\times[t_0,t_1].
\end{equation}

It remains to pass from the a.e.\ comparison in
\(\R^d\times[t_0,t_1]\) to the pointwise inequality at \((x,t_1)\).  Choose
\(\alpha_j,r_j,\eta_j\downarrow0\), with \(t_1-\eta_j>t_0\).  By
\eqref{eq:HS-good-slice-envelope}, each set
\[
    E_j:=\{w\in B_{r_j}(y):p(w,t_0)>p(y,t_0)-\alpha_j\}
\]
has positive measure.  Thus we may choose
\((z_j,s_j)\in\R^d\times(t_1-\eta_j,t_1)\) such that
\[
    z_j+\frac{s_j-t_0}{t_1-t_0}h\in E_j
\]
and \eqref{eq:HS-Hopf-Lax-ae-comparison} holds at \((z_j,s_j)\).  Then
\(s_j\to t_1\) and \(z_j\to y-h=x\) .  Moreover,
\begin{equation}\label{eq:HS-Hopf-Lax-endpoint-sequence-lower}
    p(z_j,s_j)\geq \bar p(z_j,s_j)
    \geq
    e^{-A(s_j)}
    \bigl(p(y,t_0)-\alpha_j-C(s_j)\bigr)_+ .
\end{equation}
Letting \(j\to\infty\) in
\eqref{eq:HS-Hopf-Lax-endpoint-sequence-lower}, using space-time upper semicontinuity of
\(p\) from \eqref{eq:HS-forward-average-representative} and continuity of \(A\) and \(C\), gives
\[
    p(x,t_1)\geq e^{-A(t_1)}\bigl(p(y,t_0)-C(t_1)\bigr)_+ .
\]
Substituting \(A(t_1)=M\abs{x-y}+\Lambda(t_1)\) together with the
definition of \(C(t_1)\) yields \eqref{eq:HS-pressure-hopf-lax}.
\end{proof}

\begin{rem}\label{rem:HS-Hopf-Lax-HJ}
Note that, when \(m=\infty\), formally setting \(y=x-v\delta\), \(t_1=t+\delta\),
\(t_0=t\), dividing \eqref{eq:HS-pressure-hopf-lax} by \(\delta\), and
optimizing in \(v\) gives the Hamilton--Jacobi inequality
\begin{equation}\label{eq:HS-formal-HJ}
    \partial_t p-\bigl(\abs{\nabla p}-M(t)p\bigr)_+^2+k(t)p\geq0,
\end{equation}
where
\[
    M(t):=\frac{\norm{\nabla_x c(\cdot,t)}_{L^\infty(\R^d)}}
    {\inf_{\R^d}c(\cdot,t)},
    \qquad
    k(t):=\frac{\norm{(\partial_t c)^-(\cdot,t)}_{L^\infty(\R^d)}}
    {\inf_{\R^d}c(\cdot,t)} .
\]

\end{rem}

\begin{rem}\label{rem:HS-hitting-time-holder}
The Hopf--Lax inequality, together with the annular barrier argument of
\cite[Lem.~4.1, Thm.~4.2]{CJK25}, yields H\"older continuity of the hitting time under
\eqref{eq:A-source-time} alone, with no nondegeneracy assumption on the initial patch:
\[T\in C^{0,\alpha}_{\mathrm{loc}}(\{0<T<\infty\})
\qquad\hbox{for every }\alpha<\alpha_d:=
\begin{cases}
2/e, & d=2,\\
2(2/d)^{d/(d-2)}, & d>2.
\end{cases}\]
The endpoint exponent \(\alpha_d\) is attained if the time integrability in
\eqref{eq:A-source-time} is strengthened to
\((\partial_t c)^-\in L_{\operatorname{loc}}^\beta([0,\infty);L_{\operatorname{loc}}^m(\R^d))\)
for some \(\beta>1\).
\end{rem}

\section{Static barriers for pressure at singular geometry}\label{sec:static-barriers}

In this section, we estimate the growth of the pressure away from the zero set in geometric configurations which model $\Omega_t$ shortly before a $k$th stratum singular point occurs, for $0\leq k \leq d-2$. 

In the following, we fix a $k$-dimensional hyperplane $V\subset \R^d$ containing the origin,
and work in the coordinate split $x = (x',x'')\in V^\perp\times V$. When $\dim V=k=0$, the second coordinate is trivial. For parameters $0 < r_0 < R_0$ and $\alpha\in (0,1)$, we set (see Figure~\ref{fig:U-region})
\begin{equation}\label{eq:u_alpha}
    U(V,\alpha,R_0,r_0) := \{ x : r_0 < |x| < R_0 \hbox{ and } |x'| > \alpha |x''| \}.
\end{equation}
When $k=0$ we have $U(V,\alpha, R_0, r_0)=U( R_0, r_0)= B_{R_0}\setminus \overline{B_{r_0}}$.

\begin{figure}[ht]
\centering
\begin{tikzpicture}[scale=0.672,>=Stealth]
    \fill[gray!15] (0,0) -- (70:3.3) arc (70:110:3.3) -- cycle;
    \fill[gray!15] (0,0) -- (250:3.3) arc (250:290:3.3) -- cycle;
    \fill[blue!10] (-70:3) arc (-70:70:3) -- (70:0.6) arc (70:-70:0.6) -- cycle;
    \fill[blue!10] (110:3) arc (110:250:3) -- (250:0.6) arc (250:110:0.6) -- cycle;
    \draw[->] (-3.6,0) -- (3.6,0) node[below right] {$V^\perp$};
    \draw[->] (0,-3.6) -- (0,3.6) node[above] {$V$};
    \draw[dashed] (110:3.3) -- (0,0) -- (70:3.3);
    \draw[dashed] (250:3.3) -- (0,0) -- (290:3.3);
    \draw[thick] (-70:3) arc (-70:70:3) -- (70:0.6) arc (70:-70:0.6) -- cycle;
    \draw[thick] (110:3) arc (110:250:3) -- (250:0.6) arc (250:110:0.6) -- cycle;
    \draw[->] (0,0) -- (160:3);
    \node at (160:1.7) [above=-2pt] {$R_0$};
    \draw[->] (0,0) -- (205:0.6);
    \node at (218:0.95) {\scriptsize $r_0$};
    \fill (1.9,0) circle (1.7pt);
    \node[below right=-2pt] at (1.9,0) {$re$};
    \node at (35:1.9) {$U$};
    \node at (0,2.9) {\tiny $|x'|\leq \alpha|x''|$};
\end{tikzpicture}
\hfil
\begin{tikzpicture}[scale=0.672,>=Stealth]
    \fill[green!10] plot [smooth cycle, tension=0.9] coordinates
        {(0:3.6) (40:3.95) (85:3.75) (130:3.55) (175:3.7) (220:3.5) (265:3.75) (315:3.55)};
    \draw[thin] plot [smooth cycle, tension=0.9] coordinates
        {(0:3.6) (40:3.95) (85:3.75) (130:3.55) (175:3.7) (220:3.5) (265:3.75) (315:3.55)};
    \filldraw[fill=white, draw=black, thin] plot [smooth cycle, tension=0.7] coordinates
        {(0.6,0) (0.38,0.28) (0.18,0.7) (0.45,2.0) (0.5,2.6) (0,2.95)
         (-0.5,2.6) (-0.45,2.0) (-0.18,0.7) (-0.4,0.25) (-0.44,0) (-0.4,-0.25)
         (-0.18,-0.7) (-0.45,-2.0) (-0.5,-2.6) (0,-2.95) (0.5,-2.6)
         (0.45,-2.0) (0.18,-0.7) (0.38,-0.28)};
    \fill[blue!10] (-70:3) arc (-70:70:3) -- (70:0.6) arc (70:-70:0.6) -- cycle;
    \fill[blue!10] (110:3) arc (110:250:3) -- (250:0.6) arc (250:110:0.6) -- cycle;
    \draw[dashed] (110:3.3) -- (0,0) -- (70:3.3);
    \draw[dashed] (250:3.3) -- (0,0) -- (290:3.3);
    \draw[thick] (-70:3) arc (-70:70:3) -- (70:0.6) arc (70:-70:0.6) -- cycle;
    \draw[thick] (110:3) arc (110:250:3) -- (250:0.6) arc (250:110:0.6) -- cycle;
    \draw[->] (0,0) -- (-30:0.6);
    \node at (-22:1.15) {\scriptsize $r_+(t)$};
    \node at (40:3.45) {$\Omega_t$};
\end{tikzpicture}
\caption{Left: the region $U(V,\alpha,R_0,r_0)$ (blue). Right: the setting of Section~\ref{sec:closing-rates}: $U(V,\alpha,R_0,r_+(t))\subset\Omega_t$ (green), and the zero set (white) lies in the union of the cone with $\overline{B_{r_+(t)}}$, touching $\partial B_{r_+(t)}$.}
\label{fig:U-region}
\end{figure}
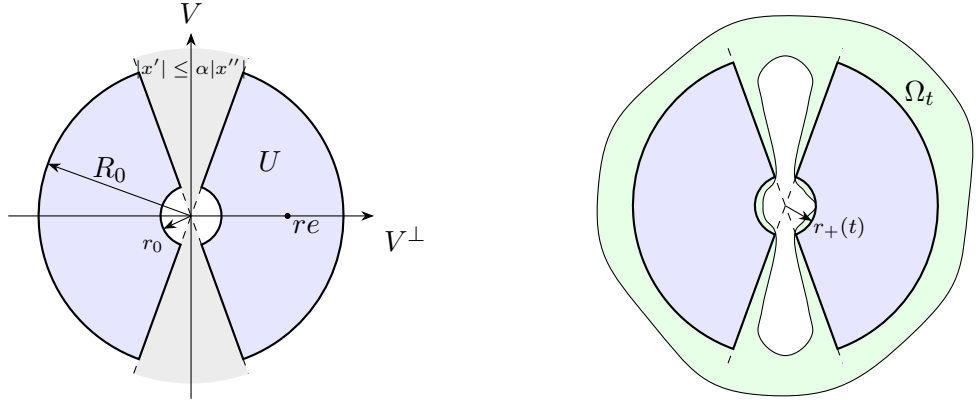

The main result of this section is the following growth estimate for $p$ at a fixed time:

\begin{prop}\label{prop:pressure_lower_bd}
Assume \eqref{eq:A-source}, \eqref{eq:A-initial}, and \eqref{eq:A-source-time}.
Let $V \subset \R^d$ be a subspace of dimension $0 \leq k \leq d-2$,  and let $\alpha\in (0, \alpha_0)$ with  $\alpha_0 = \alpha_0(d,k)\leq 1$. Suppose that for some $t$ and for $0 < r_0 < \frac{R_0}{16}$, we have
\[ U(V, \alpha, R_0, r_0) \subset \Omega_t. \]
Let $e\in V^\perp$ with $|e| = 1$. Then for $4r_0 < r < \frac{R_0}{4}$, we have
\[ p(re, t) \geq C(d,k)\underline{c}(R_0,t)R_0^2\cdot \begin{cases}
    1 & d\geq 3, k=0 \\
    (\log \frac{R_0}{r})^{-1} & d=2,k=0 \\
    (\tfrac{r}{R_0})^{\nu} & 1\leq k \leq d-2
\end{cases} \]
Here,
$\underline{c}(R_0, t) = \inf_{B_{R_0}\times \{t\}} c$, and $\nu=\nu(\alpha,d,k)\to 0^+$ as $\alpha \to 0$.
\end{prop}

We split the proof of Proposition~\ref{prop:pressure_lower_bd} into separate lemmas for the $k = 0$ and $1 \leq k \leq d-2$ cases. 

\begin{lem}\label{lem:radial_subbarrier}
Proposition~\ref{prop:pressure_lower_bd} holds in the case $k = 0$.
\end{lem}
\begin{proof}
For fixed $t$, let $\psi(x)$ solve
\begin{equation}\label{eq:zs_annular_subsolution_eqn}
    \begin{cases}
        -\Delta \psi = \underline{c}(R_0,t) & \hbox{ in }B_{R_0}\setminus \overline{B_{r_0}} \\
        \psi = 0 & \hbox{ on }\partial B_{R_0}\cup \partial B_{r_0} \\
    \end{cases}
\end{equation}
By assumption, $U = B_{R_0}\setminus \overline{B_{r_0}} \subset \Omega_t$, so $\Delta p(t) = -c \leq -\underline{c}(R_0,t)$ on $U$. We have $p(t)\geq 0$ on $\partial U$, so comparison gives $\psi \leq p(t)$ on $U$. We can compute $\psi$ explicitly as
\begin{equation}\label{eq:radial_barrier_explicit}
    \psi(x) = \begin{cases}
        \frac{\underline{c}}{4}\left(R_0^2 - r_0^2\right)\left(1 - \frac{|x|^2 - r_0^2}{R_0^2 - r_0^2} - \frac{\log \frac{R_0}{|x|}}{\log\frac{R_0}{r_0}}\right) & d=2 \\
        \frac{\underline{c}}{2d}\left(R_0^2 - r_0^2\right)\left(1 - \frac{|x|^2 - r_0^2}{R_0^2 - r_0^2} - \frac{|x|^{2-d} - R_0^{2-d}}{r_0^{2-d} - R_0^{2-d}}\right) & d\geq 3
    \end{cases}
\end{equation}

On the range $4r_0 \leq |x| \leq R_0/4$, elementary calculus bounds yield that
\[ 1 - \frac{|x|^2 - r_0^2}{R_0^2 - r_0^2} - \frac{\log \frac{R_0}{|x|}}{\log\frac{R_0}{r_0}} \geq \frac{15\log 4 }{34\log \frac{R_0}{|x|}}, \]
and
\[ 1 - \frac{|x|^2 - r_0^2}{R_0^2 - r_0^2} - \frac{|x|^{2-d} - R_0^{2-d}}{r_0^{2-d} - R_0^{2-d}} \geq \min\left(\frac{15}{16}, \frac{16}{17} - \frac{1}{4^{d-2}+1}\right) \geq \frac{63}{85}. \]
Putting these lower bounds into \eqref{eq:radial_barrier_explicit}, we conclude.

\end{proof}

We now turn to the case $1 \leq k \leq d-2$. Our aim is to construct a subsolution $\psi$ to the pressure equation, vanishing on the cone $|x'|=\alpha |x''|$ and satisfying an almost optimal growth rate in $\{4r_0<|x|<R_0/4\}$: see \eqref{subsolution}. To this end, we separate variables. Set $\rho:=|x|$ and $\varphi:= \arctan \frac{|x'|}{|x''|}$, so that 
\begin{equation}\label{eq:spherical_coordinates}
\begin{aligned}
    |x'| &= \rho\sin\varphi,\quad |x''| = \rho\cos\varphi,& \\
    \rho&\in (r_0, R_0), \quad \varphi\in (\varphi_\alpha, \pi/2] &\hbox{ on } U(V, \alpha, R_0, r_0), \quad \varphi_\alpha:=\arctan\alpha.
\end{aligned}
\end{equation}
Let $u(x) = u(\rho, \varphi)$, so that $u$ is invariant under rotations inside $V$ and $V^\perp$. We have
\begin{equation}\label{eq:cone_laplacian}
    \Delta_x u = \partial_{\rho\rho}u + \frac{d-1}{\rho}\partial_\rho u + \frac{1}{\rho^2}\left(\partial_{\varphi\varphi} + ((d-k-1)\cot\varphi - (k-1)\tan\varphi)\partial_\varphi\right)u.
\end{equation}
 For $u=\rho^{\nu}\Phi(\varphi)$, $u$ is harmonic if $\Phi''+\big((d-k-1)\cot\varphi-(k-1)\tan\varphi\big)\Phi'+\nu(\nu+d-2)\Phi=0$. This leads to the Sturm-Liouville problem
\begin{equation}\label{eq:sturm-liouville}
    -\frac{1}{w_{d,k}}\left(w_{d,k}\Phi'_\alpha\right)' = \lambda_\alpha\Phi_\alpha,  \Phi_\alpha(\varphi_\alpha) = 0, \Phi_\alpha'(\pi/2) = 0,
\hbox{ where } w_{d,k}(\varphi) := \sin^{d-k-1}(\varphi)\cos^{k-1}(\varphi).
\end{equation} We set $\lambda_\alpha > 0$ to be the smallest eigenvalue of \eqref{eq:sturm-liouville}, determined by the Rayleigh quotient
\begin{equation}\label{eq:rayleigh_quotient}
        \lambda_\alpha := \underset{\Phi(\varphi_\alpha) = 0, \Phi\not\equiv0}{\inf_{\Phi\in H^1([\varphi_\alpha,\pi/2], w_{d,k}d\varphi)}}\frac{\int_{\varphi_\alpha}^{\pi/2} |\Phi'|^2w_{d,k}(\varphi)d\varphi}{\int_{\varphi_\alpha}^{\pi/2} |\Phi|^2 w_{d,k}(\varphi)d\varphi}.
\end{equation}
We record a few  properties of $\Phi_{\alpha}$ in the next lemma.

\begin{lem}\label{lem:first_eigenfunction}
There exists a solution  $\Phi_\alpha$ of \eqref{eq:sturm-liouville} with $\Phi_\alpha(\pi/2) = 1$. Moreover $\Phi_\alpha'>0$ on $[\varphi_\alpha, \pi/2)$.
\end{lem}
\begin{proof}
    By the direct method, we can find a minimizer $\Phi_\alpha$ of \eqref{eq:rayleigh_quotient}. If $\Phi_\alpha$ minimizes \eqref{eq:rayleigh_quotient}, then so does $|\Phi_\alpha|$, so we may assume $\Phi_\alpha \geq 0$. Integrating \eqref{eq:sturm-liouville} backwards from $\pi/2$ then gives
    \begin{equation}\label{eq:sturm-liouville-1st-derivative}
        w_{d,k}(\varphi)\Phi_\alpha'(\varphi) = \lambda_\alpha \int_{\varphi}^{\pi/2} w_{d,k}(s)\Phi_\alpha(s) ds.
    \end{equation}
    We use \eqref{eq:sturm-liouville-1st-derivative} with $\Phi_\alpha \geq 0$, $\lambda_\alpha > 0$, and $w_{d,k} > 0$ to obtain that $\Phi_\alpha' \geq 0$. Since $\Phi_\alpha(\varphi_\alpha) = 0$ and $\Phi_\alpha\not\equiv 0$, it follows that the right side of \eqref{eq:sturm-liouville-1st-derivative} is positive for all $\varphi\in (\varphi_\alpha, \pi/2)$, and hence $\Phi_\alpha' > 0$ on $[\varphi_\alpha, \pi/2)$. We conclude by scaling $\Phi_\alpha$ to satisfy $\Phi_\alpha(\pi/2) = 1$.
\end{proof}

From \eqref{eq:cone_laplacian} and \eqref{eq:sturm-liouville}, direct computation shows that, for
\begin{equation}\label{eq:cone_exponent}
    \nu_\alpha := \lambda_\alpha[\tfrac{d-2}{2} + ((\tfrac{d-2}{2})^2 + \lambda_\alpha)^{1/2}]^{-1},
\end{equation}
we have
\begin{equation}\label{eq:cone_harmonics}
\Delta(\rho^{\nu_\alpha}\Phi_\alpha(\varphi)) = 0, \quad
    \Delta(\rho^{2-d -\nu_\alpha}\Phi_\alpha(\varphi)) =0, \quad
  \Delta(\rho^2\Phi_\alpha(\varphi)) = (2d-\lambda_\alpha)\Phi_\alpha.
\end{equation}

We will use the functions in \eqref{eq:cone_harmonics} to construct the subsolution $\psi$.  First, we estimate the decay of $\lambda_\alpha$ and $\nu_\alpha$ as $\alpha\to 0$, using \eqref{eq:rayleigh_quotient}.

\begin{lem}\label{lem:first_eigenvalue}
Assume $0 < \alpha \leq \alpha_0(d,k)$, and set
\begin{equation}\label{eq:sigma_d_minus_k}
    \sigma_{d-k}(\alpha) := \begin{cases}
    \frac{1}{\log(1/\alpha)} & d-k=2 \\ \alpha^{d-k-2} & d-k > 2
\end{cases}
\end{equation}

Then there exists a constant $C=C(d,k) > 0$ such that
    \[ C^{-1}\sigma_{d-k}(\alpha) \leq \lambda_\alpha, \nu_{\alpha} \leq C\sigma_{d-k}(\alpha) \leq 1. \]
 
\end{lem}
\begin{proof}
  Recalling $\varphi_\alpha := \arctan \alpha$, let $\alpha$ be sufficiently small that $\varphi_\alpha \leq \pi/8$. Define
    \[ \Phi(\varphi) := \begin{cases}
        \frac{\int_{\varphi_\alpha}^{\varphi} w_{d,k}(t)^{-1} dt}{\int_{\varphi_\alpha}^{\pi/4} w_{d,k}(t)^{-1} dt} & \varphi \leq \pi/4 \\
        1 & \varphi \geq \pi/4,
    \end{cases} \]
     Since $\Phi \geq 0$ with $\Phi\equiv 1$ on $[\pi/4, \pi/2]$, we have
    \[ \int_{\varphi_\alpha}^{\pi/2} |\Phi|^2w_{d,k}(\varphi)d\varphi \geq \int_{\pi/4}^{\pi/2} w_{d,k}(\varphi)d\varphi =: c_{d,k}. \]
    Testing the Rayleigh quotient \eqref{eq:rayleigh_quotient} with $\Phi$ then yields
    \[ \lambda_\alpha \leq \frac{1}{c_{d,k}}\int_{\varphi_\alpha}^{\pi/2} |\Phi'|^2w_{d,k}(\varphi)d\varphi = \frac{1}{c_{d,k}}\left(\int_{\varphi_\alpha}^{\pi/4} \frac{1}{\sin^{d-k-1}(\varphi)\cos^{k-1}(\varphi)}d\varphi\right)^{-1}. \]
    Estimating with $\sin\varphi \leq \varphi$ and $\cos\varphi \leq 1$ gives $\lambda_\alpha \leq \frac{C_1}{\sigma_{d-k}(\varphi_\alpha)^{-1} - C_2}$ for some $C_i(d,k)$. This yields the upper bound for $\lambda_\alpha$ once $\alpha$ is sufficiently small.

    Toward the reverse inequality for $\lambda_\alpha$, Cauchy-Schwarz gives
    \begin{equation*}
        |\Phi(\varphi)|^2 \leq \left(\int_{\varphi_\alpha}^{\varphi} w_{d,k}(t)^{-1}dt\right)\left(\int_{\varphi_\alpha}^{\varphi}|\Phi'(t)|^2w_{d,k}(t)dt\right)
    \end{equation*}
    for any admissible $\Phi$ and $\varphi\in (\varphi_\alpha, \pi/2)$. We integrate both sides against $w_{d,k}(\varphi)$ and bound the rightmost factor by its value at $\varphi = \pi/2$. Rearranging yields the Rayleigh quotient, and thus
    \begin{align*}
        \lambda_\alpha \geq \left(\int_{\varphi_\alpha}^{\pi/2}\int_{\varphi_\alpha}^{\varphi}  \frac{w_{d,k}(\varphi)}{w_{d,k}(t)}dtd\varphi\right)^{-1}
    \end{align*}
    We can estimate this using $\cos t \geq \cos\varphi \geq 0$, $\sin\varphi \leq \varphi$, and $\sin t \geq t/2$, since $0 \leq t \leq\varphi\leq \pi/2$. As before, this implies $\lambda_\alpha \geq \frac{C_3}{\sigma_{d-k}(\varphi_\alpha)^{-1} - C_4}$ for some $C_i(d,k)$, which implies the lower bound for $\lambda_\alpha$ once $\alpha$ is sufficiently small. The bounds for $\nu_\alpha$ follow directly from applying the bounds for $\lambda_\alpha$ to \eqref{eq:cone_exponent}.

\end{proof}

\begin{lem}\label{lem:cone_subbarrier}
Proposition~\ref{prop:pressure_lower_bd} holds in the case $1\leq k \leq d-2$. In fact, we have
\[ p(x,t) \geq C(d)\underline{c}R_0^{2-\nu_\alpha}\rho^{\nu_\alpha}\Phi_\alpha(\varphi), \]
whenever $4r_0 < \rho < R_0/4$ and $\varphi\in [\varphi_\alpha, \pi/2]$.
\end{lem}
\begin{proof}

    For fixed $t$, let $\psi(x)$ solve
    \begin{equation}\label{subsolution} 
    \Delta \psi = -\underline{c}\Phi_\alpha(\varphi) \hbox{ in } U(V,\alpha, R_0, r_0), \quad         \psi = 0 \hbox{ on } \partial U(V,\alpha, R_0, r_0).
    \end{equation}
 We require $\alpha \leq \alpha_0$, with $\alpha_0(d,k)$ chosen by Lemma~\ref{lem:first_eigenvalue} to ensure that $\nu_\alpha, \lambda_\alpha\in (0, 1]$.   
    From Lemma~\ref{lem:first_eigenfunction}, $0 \leq \Phi_\alpha \leq 1$, and thus $\Delta \psi \geq -\underline{c} \geq -c = \Delta p(t)$ in $U$. Since $\psi = 0 \leq p(t)$ on $\partial U$, elliptic comparison principle gives $\psi \leq p(t)$ in $U$.
    
    From \eqref{eq:cone_harmonics}, $\psi$ can be found explicitly in the form
    \[ \psi(\rho, \varphi) = \frac{\underline{c}}{2d-\lambda_\alpha}(A\rho^{\nu_\alpha} + B\rho^{2-d - \nu_\alpha} - \rho^2)\Phi_\alpha(\varphi), \]
    for constants
    \[ A = \frac{R_0^{d+\nu_\alpha} - r_0^{d+\nu_\alpha}}{R_0^{d-2+2\nu_\alpha}- r_0^{d-2+2\nu_\alpha}}, \quad B = r_0^{d + \nu_\alpha} - Ar_0^{d-2+2\nu_\alpha}. \]
    Let us write $A$ as $R_0^{2-\nu_\alpha}\left(1 + \frac{r_0^{d-2+2\nu_\alpha}(R_0^{2-\nu_\alpha}- r_0^{2-\nu_\alpha})}{R_0^{d+\nu_\alpha}- R_0^{2-\nu_\alpha}r_0^{d-2+2\nu_\alpha}}\right)$. Since $0 < \nu_\alpha < 2$, we have $R_0^{2-\nu_\alpha} - r_0^{2-\nu_\alpha} > 0$ and $R_0^{d+\nu_\alpha}- R_0^{2-\nu_\alpha}r_0^{d-2+2\nu_\alpha} > 0$. Thus, $A \geq R_0^{2-\nu_\alpha}$. On the other hand, using  $16r_0\leq R_0$,  we have $A \leq \frac{1}{1- (1/16)^{d-2+2\nu_\alpha}}R_0^{2-\nu_\alpha}$ and thus $B \geq -\frac{r_0^{d-2+2\nu_\alpha}R_0^{2-\nu_\alpha}}{1- (1/16)^{d-2+2\nu_\alpha}}$. Altogether,
    \[ \begin{aligned}
        \psi(\rho, \varphi) &\geq \frac{\underline{c}}{2d-\lambda_\alpha}\left(R_0^{2-\nu_\alpha}\rho^{\nu_\alpha} - \frac{r_0^{d-2+2\nu_\alpha}R_0^{2-\nu_\alpha}}{1- (1/16)^{d-2+2\nu_\alpha}}\rho^{2-d-\nu_\alpha} - \rho^2\right)\Phi_\alpha(\varphi)
        \\&= \frac{\underline{c}R_0^{2-\nu_\alpha}}{2d-\lambda_\alpha}\left(1 - \frac{1}{1- (1/16)^{d-2+2\nu_\alpha}}\left(\frac{r_0}{\rho}\right)^{d-2 + 2\nu_\alpha} - \left(\frac{\rho}{R_0}\right)^{2-\nu_\alpha}\right)\rho^{\nu_\alpha}\Phi_\alpha(\varphi).
    \end{aligned} \]
    Since $\nu_\alpha > 0$ and $d\geq 3$, we have $d-2+2\nu_\alpha \geq 1$. It follows that $\frac{1}{1- (1/16)^{d-2+2\nu_\alpha}}\left(\frac{r_0}{\rho}\right)^{d-2 + 2\nu_\alpha} \leq \frac{4}{15}$ for $\rho \geq 4r_0$. On the other hand, since $\nu_\alpha \leq 1$, we have $2-\nu_\alpha \geq 1$. For $\rho \leq R_0/4$, it follows that $\left(\frac{\rho}{R_0}\right)^{2-\nu_\alpha} \leq \frac{1}{4}$. Finally, we note that $0 < 2d-\lambda_\alpha < 2d$, since $\lambda_\alpha\in (0,1]$. We conclude that for $\rho\in (4r_0, R_0/4)$, we have
    \[ \psi(\rho, \varphi) \geq \frac{29}{60}\cdot \frac{\underline{c}R_0^{2-\nu_\alpha}}{2d}\rho^{\nu_\alpha}\Phi_\alpha(\varphi). \]
\end{proof}

\section{Closing rates and asymptotic behavior at singular points}\label{sec:closing-rates}

Throughout this section, we work at a singular point $0\in \Sigma_k(t_0)$, with $0 \leq k \leq d-2$. Let $p_2$ be the blowup of $w(\cdot, t_0)$ at $0$, and let $\lambda_1>0$ be its smallest eigenvalue. We set $V := \ker D^2 p_2$, and write $x = (x',x'')\in V^\perp\times V$.

Two hypotheses enter throughout this section: the time regularity
\eqref{eq:A-source-time}, through Theorem~\ref{prop:HS-pressure-hopf-lax}, and the local
H\"older continuity of the hitting time,
\begin{equation}\label{eq:closing-T-holder}
    T\in C^{0,\beta}_{\operatorname{loc}}(\{0<T<\infty\})
    \qquad\hbox{for some }\beta\in(0,1] .
\end{equation}
Under \eqref{eq:closing-T-holder}, the source $f$ of \eqref{eq:intro-w-obstacle} is locally H\"older in space, and
hence the blowup $p_2$ and the stratification \eqref{eq:intro-strata} are well-defined. The H\"older continuity of $f$
enters quantitatively to control the convergence rate of the blowup in Lemma~\ref{lem:r_plus_upper_bd}.  Under
\eqref{eq:A-interior-ball} and \eqref{eq:A-source-time} with $m=\infty$,
\eqref{eq:closing-T-holder} holds with $\beta=1$ by
Theorem~\ref{thm:HS-hitting-time-lipschitz}; alternatively, it holds under
\eqref{eq:A-source}, \eqref{eq:A-initial}, and \eqref{eq:A-source-time} alone, by
Remark~\ref{rem:HS-hitting-time-holder}.

For $R_0>0$ and $0<\alpha \leq \alpha_0(d,k)$ such that all of the results of Section~\ref{sec:static-barriers} hold, define
\begin{equation}\label{eq:r_plus_def}
    r_+(t) =r_+(t;R_0, \alpha):= \inf \{ r > 0 : U(V, \alpha, R_0,r) \subset \Omega_t\}.
\end{equation}
(See Figures~\ref{fig:U-region} and~\ref{fig:r-perp}). When $k=0$, $r_+(t)$ is simply the outer radius of the zero set in $B_{R_0}$.
We measure the transverse size of the same portion of the zero set by
\begin{equation}\label{eq:r-perp-def}
    r_\perp(t)=r_\perp(t;R_0,\alpha)
    :=\sup\bigl\{\abs{x'}:x\in U(V,\alpha, R_0,0)\setminus\Omega_t\bigr\},
\end{equation}
Let us note that for $t < t_0$, we have $r_+(t), r_\perp(t)\in (0, R_0]$. Here, the lower bounds both follow from $0\notin \overline{\Omega_t}$, and the upper bounds are trivial (since $U(V, \alpha, R_0, R_0) = \emptyset$ and $U(V, \alpha, R_0,0)\subset B_{R_0}$). Let us also remark that the definition of $r_\perp$ here is equivalent to the definition given in \eqref{eq:intro-r-perp}, since $\{ x\in B_{R_0}\setminus \Omega_t: |x'| > \alpha |x''| \} = U(V,\alpha, R_0,0)\setminus\Omega_t$.

We have that $r_+, r_\perp$ are comparable, with the estimate
\begin{equation}\label{eq:r-perp-r-plus-comparison}
    \frac{\alpha}{\sqrt{1+\alpha^2}}\,r_+(t)
    \leq r_\perp(t)\leq r_+(t).
\end{equation}
Indeed, since $U(V, \alpha, R_0, r_+(t)) = \bigcup_{r > r_+(t)} U(V, \alpha, R_0, r)$, \eqref{eq:r_plus_def} gives $U(V, \alpha, R_0, r_+(t)) \subset \Omega_t$. Thus, $U(V, \alpha, R_0,0)\setminus \Omega_t \subset \overline{B_{r_+(t)}}$, which gives the upper bound on $r_\perp(t)$. Conversely, for each $r < r_+(t)$, there exists $x\in U(V, \alpha, R_0, r)\setminus \Omega_t$. This $x$ satisfies $|x| > r$ and $|x'| > \alpha |x''|$, which implies $r_\perp(t) \geq |x'| > \alpha (1 + \alpha^2)^{-1/2} r$. We conclude the lower bound by sending $r \uparrow r_+(t)$.

\subsection{Almost-sharp closing rates}
\label{subsec:almost-sharp-closing-rates}

In this section, we estimate $r_+(t)$ for $t<t_0$.
Equation~\eqref{eq:r-perp-r-plus-comparison} then transfers these
bounds to $r_\perp(t)$.
We begin by showing that $r_+(t)\to0$ as $t\to t_0^-$.

\begin{lem}\label{lem:r_plus_upper_bd}
Assume \eqref{eq:A-source}, \eqref{eq:A-initial}, \eqref{eq:A-source-time}, and
\eqref{eq:closing-T-holder}. Let $R_0=R_0(d,\lambda_1, \alpha,\|c\|_{L^1_t C^{0,\beta}_x}, \|T\|_{C^{0,\beta}})$ be sufficiently small. Then
\begin{equation}\label{eq:upper_bd_r}
    r_+(t) \leq C(t_0 - t)^{1/2},
\end{equation}
where $C = C(\lambda_1, \alpha, \|p\|_{L^\infty(B_{R_0}\times (t,t_0))})$.
\end{lem}
\begin{proof}
    By definition of $w$, we have the Lipschitz estimate
    \[ w(x,t) \geq w(x,t_0) - \|p\|_{L^\infty(B_{R_0}\times (t,t_0))}(t_0 - t). \]
    Using $p_2$ to Taylor-expand $w(x,t_0)$ at $x=0$ leads to
    \[ w(x,t) \geq \lambda_1|x'|^2 - \sigma(|x|)|x|^2 - \|p\|_{L^\infty_{x,t}}(t_0 - t). \]
    Here, $\sigma(|x|)|x|^2 = o(|x|^2)$ bounds the remainder for the quadratic blowup from Lemma~\ref{lem:stationary-blowup}. The coefficient $\sigma$ can be taken as a universal modulus of continuity depending only on dimension and the $C^{0,\beta}$ norm of $f(\cdot, t_0)$ in a neighborhood of 0, for $\beta\in (0,1]$ and $f$ from \eqref{eq:intro-w-obstacle} (see \cite[Lemma 6.13]{CJK25}). This norm of $f$ is controlled by $\|c\|_{L^1_tC^{0,\beta}_x}$ and $\|T\|_{C^{0,\beta}}$.
    
    In the $k=0$ case we have $|x'|^2 = |x|^2$. Thus, for $R_0$ sufficiently small that $\sigma(R_0) \leq \frac{\lambda_1}{2}$, we have
    \[ w(x,t) \geq \frac{\lambda_1}{2}|x|^2 - \|p\|_{L^\infty(B_{R_0}\times (t,t_0))}(t_0 - t), \]
     which is positive when $|x|^2 > C(t_0-t)$ with $C=2\lambda_1^{-1}\|p\|_{L^\infty(B_{R_0}\times (t,t_0))}$, so \eqref{eq:upper_bd_r} follows with this $C$.

    In the $1\leq k \leq d-2$ case, we choose $R_0$ small so that $\sigma(R_0) \leq \lambda_1\frac{\alpha^2/2}{1 + \alpha^2/2}$. Then
    \begin{equation}\label{eq:w_lower_bound_cone}
    w(x,t) \geq \frac{\lambda_1}{1 + \alpha^2/2}\left(|x'|^2 - \frac{\alpha^2}{2}|x''|^2 - \frac{1 + \alpha^2/2}{\lambda_1}\|p\|_{L^\infty}(t_0 - t)\right).
    \end{equation}
    Inside the cone $\{ |x'| \geq \alpha |x''|\}$, which is $\{ \varphi \geq \arctan \alpha\}$ in the $(\rho, \varphi)$-coordinates of \eqref{eq:spherical_coordinates}, we have
    \[ \begin{aligned}
    |x'|^2 - \frac{\alpha^2}{2}|x''|^2 
    = |x|^2\left(\sin^2\varphi - \frac{\alpha^2}{2}\cos^2\varphi\right)
    \geq \frac{\alpha^2/2}{1 + \alpha^2}|x|^2.
\end{aligned}\]
    Thus, for $x$ with $|x'| \geq \alpha |x''|$, we have
    \[ w(x,t) \geq \frac{\lambda_1(\alpha^2/2)}{(1 + \alpha^2/2)(1 + \alpha^2)}\left(|x|^2 - \frac{(1 + \alpha^2/2)(1 + \alpha^2)}{\lambda_1(\alpha^2/2)}\|p\|_{L^\infty}(t_0 - t)\right), \]
    Using $\alpha < 1$, it follows that $w > 0$ on $U(V,\alpha, R_0, r)$ for $r > \sqrt{\frac{6\|p\|_{L^\infty}}{\lambda_1 \alpha^2}(t_0 - t)}$, so we conclude.
\end{proof}

Next we obtain a lower bound for $r_+$ by combining growth estimate on pressure with Hopf-Lax. We set
\begin{equation}\label{eq:hopf-lax-const-reminder}
    C_{\mathrm{Hopf-Lax}} := \frac{1}{4}\exp(M + \|\Lambda\|_{L^\infty(0,t_0)})
\end{equation}
where $M, \Lambda$ from \eqref{eq:HS-Hopf-Lax-cinf-L-M}--\eqref{eq:HS-Hopf-Lax-Lambda-q} depend on \eqref{eq:A-source} and \eqref{eq:A-source-time}.

Let $(x_1, t_1)$ and $(x_2,t_2)$ with $t_1 < t_2$, $|x_1 - x_2| \leq 1$. Then if $p(x_2, t_2) = 0$, Theorem~\ref{prop:HS-pressure-hopf-lax} yields
\[ p(x_1, t_1) \leq C_{\mathrm{Hopf-Lax}}\frac{|x_2-x_1|^2}{t_2 - t_1}. \]

\begin{lem}\label{lem:r_plus_lower_bd}
Assume \eqref{eq:A-source}, \eqref{eq:A-initial}, and \eqref{eq:A-source-time}. For $t < t_0$ sufficiently close that $\frac{r_+(t)}{R_0} < \frac{1}{16}$ and $4r_+(t) \leq 1$, we have the estimate
\begin{equation}
    t_0 - t \leq \frac{C(d,k)C_{\mathrm{Hopf-Lax}}}{\underline{c}(R_0,t)R_0^2}r_+(t)^2\cdot\begin{cases}
        1 & d\geq 3, k=0 \\
        \log\frac{R_0}{r_+(t)} & d=2, k=0 \\
        (\tfrac{R_0}{r_+(t)})^{\nu} & 1 \leq k \leq d-2
    \end{cases}
\end{equation}
Here, $C_{\mathrm{Hopf-Lax}}$ is from \eqref{eq:hopf-lax-const-reminder} and $\underline{c},\nu$ are the same as in Proposition~\ref{prop:pressure_lower_bd}.
\end{lem}
\begin{proof}
Let $e\in V^\perp$ with $|e| = 1$. We would like to apply Hopf-Lax to compare the values of $p$ at $(4r_+(t)e, t)$ and $(0, t_0)$, but this encounters issues since the pressure is upper semicontinuous and may have $p(0, t_0) > 0$. Instead, we take $t_1 \in (t, t_0)$. We have $p(0, t_1) = 0$, so Theorem~\ref{prop:HS-pressure-hopf-lax} yields
\begin{equation}\label{eq:r_plus_hopf_lax}
    p(4r_+(t)e, t) \leq C_{\mathrm{Hopf-Lax}}\frac{16r_+(t)^2}{t_1 - t}.
\end{equation}
Sending $t_1\to t_0$, we conclude \eqref{eq:r_plus_hopf_lax} with $t_0$ in place of $t_1$. By the definition of $r_+$, we have $U(V, \alpha, R_0,r_+(t))\subset \Omega_t$, and thus we can use Proposition~\ref{prop:pressure_lower_bd} to get a lower bound for the left side of \eqref{eq:r_plus_hopf_lax}. Rearranging gives the result.
\end{proof}

\subsection{Asymptotic shape of the zero set}
\label{subsec:asymptotic-zero-set-shape}

Next, we show that the zero set, rescaled by
$\frac{1}{r_\perp(t)}$, converges to a $k$-dimensional cylinder over a
$(d-k)$-dimensional ellipsoid as $t\to t_0^-$.
The proof relies on blowup techniques and the classification of global solutions for the obstacle problem. Let us recall that $v$ is a global solution to obstacle problem if
\begin{equation}\label{eq:global-obstacle}
v\geq 0 \hbox{ and } \Delta v = \chi_{\{v>0\}}, \quad x\in \R^d.
\end{equation}
\begin{prop}\label{prop:cone_moving_blowup}
Assume \eqref{eq:A-source}, \eqref{eq:A-initial}, \eqref{eq:A-source-time}, and
\eqref{eq:closing-T-holder}. Let $R_0$ be as given in Lemma~\ref{lem:r_plus_upper_bd}, and set
$v_t(x):=r_\perp(t)^{-2}w(r_\perp(t)x,t)$ for $t < t_0$.
Then
\begin{enumerate}[label=(\alph*)]
\item There exists a solution $v_0$ to \eqref{eq:global-obstacle} such that $v_t\to v_0$ in $C^{1,1-}_{\mathrm{loc}}(\R^d)$ as $t\to t_0^-$.
\item $\{ v_0 = 0\} = E \times V$,
where $E\subset V^\perp$ is a $(d-k)$-dimensional ellipsoid depending only on $p_2$,
normalized so that $\max_{x'\in E}\abs{x'}=1$.
\item\label{item:cone-blowup-hausdorff-dist} The sets $\{ v_t = 0\} \to E\times \R^{k}$ locally in Hausdorff distance.
\end{enumerate}
\end{prop}
\begin{proof}

Set
\[
    S_t:=\overline{U(V,\alpha,R_0,0)\setminus\Omega_t}.
\]
The justification for \eqref{eq:r-perp-r-plus-comparison} gives
$S_t\subset\overline{B_{r_+(t)}}\setminus\Omega_t$, so $S_t$ is compact.
Since taking the closure does not change the supremum of $|x'|$,
\eqref{eq:r-perp-def} gives
\begin{equation}\label{eq:r-perp-attainment}
    r_\perp(t)=\max\{|x'|:x\in S_t\}.
\end{equation}

Thus $r_\perp(t)$ is attained by some point
$r_\perp(t)x_t\in S_t$ with $|x_t'|=1$.
Since $r_\perp(t)\leq r_+(t)$ by \eqref{eq:r-perp-r-plus-comparison},
Lemma~\ref{lem:r_plus_upper_bd} gives $r_\perp(t)\to0$ as $t\to t_0^-$.
At the same time, we have $|x_t'|\geq\alpha|x_t''|$, and thus
\[
    1\leq|x_t|\leq\frac{\sqrt{1+\alpha^2}}{\alpha}.
\]

Now, since $v_t(0) = 0$ and $\lim_{t\to t_0^{-}}r_\perp(t) = 0$, Lemma \ref{lem:obst-blowup} gives that $v_t$ is precompact in $C^{1,1-}_{\mathrm{loc}}(\R^d)$ as $t\to t_0^-$. We select a subsequential limit $v_0$, and then a subsequence $t_n\to t_0^-$ along which $v_{t_n}\to v_0$ and along which also $x_{t_n}$ converges to some $x_0$. Let us remark that Lemma \ref{lem:obst-blowup} guarantees that $v_0$ is convex and solves \eqref{eq:global-obstacle}; we also have $v_0(0) = 0$ from the convergence of the $v_{t_n}$.

A few observations are in order. First, since $w$ is increasing in time, we have
\[
    v_{t_n}(x)\leq
    r_\perp(t_n)^{-2}w(r_\perp(t_n)x,t_0)
\]
for each $n$.
The right side of this inequality converges to $p_2(x)$ as $n\to\infty$,
so $v_0\leq p_2$.
Since $v_0\geq 0$, this implies that $v_0$ vanishes on $\{p_2 = 0\}$. By \cite[Lem.~2.7]{EFW25}, a solution to \eqref{eq:global-obstacle} which vanishes on a line is constant in the direction of that line. Thus, $v_0$ depends only on the $x'$ component. We set
\[ E := \{ x'\in V^\perp : v_0(x',0) = 0\}.\]
Note that $E$ is convex since $v_0$ is convex, and we have $0\in E$ and $x_0'\in E$. Since $\abs{x_{t_n}'}=1$ for all $n$, we
have $|x_0'|=1$.
Since $E$ has two distinct points and $v_0 \leq p_2$, where $p_2$ is positive-definite on $V^\perp$, Lemma~\ref{lem:ellipsoid_blowup}\ref{item:quadratic_interior_criterion} yields that $E$ has nonempty interior.

We claim that $E\subset\{|x'|\leq1\}$. If not, since $E$ is convex with
nonempty interior, we can choose $y=(y',0)$ with $|y'|>1$ such that
$v_0(x)=v_0(x',0)$ vanishes in a ball around $y$. By the nondegeneracy
estimate Lemma~\ref{obst-quadratic-growth}\ref{item:obst-nondegeneracy}, it
follows that $v_{t_n}(y)=0$ for sufficiently large $n$.
Now, since $v_{t_n}(y)=0$, we have
$r_\perp(t_n) y\notin U(V,\alpha,R_0,r_+(t_n))$.
Since $y''=0$ and $r_\perp(t_n)\to0$, this implies
$r_\perp(t_n)|y|\leq r_+(t_n)$, and hence
$r_\perp(t_n)y\in S_{t_n}$, for all sufficiently large $n$.
But the optimality condition in \eqref{eq:r-perp-attainment} gives
\[
    r_\perp(t_n)|y'|
    \leq r_\perp(t_n)|x_{t_n}'|
    = r_\perp(t_n),
\]
and thus $|y'|\leq1$, which is a contradiction.

Invoking Lemma~\ref{lem:ellipsoid_blowup}\ref{item:ellipsoid_unique} with $v_0 \leq p_2$, $E\subset \overline{B_1'}$, $E\cap \partial B_1'\neq\emptyset$, we conclude that $E$ is an ellipsoid determined by $p_2$. Consequently, $v_0$ is determined by $p_2$, and did not depend on the choice of sequence $t_n$. Thus, $v_t \to v_0$, and the local convergence of zero sets in Hausdorff distance holds by Lemma~\ref{lem:obst-blowup}\ref{item:obst-hausdorff-dist}.

\end{proof}

We can now collect the results of this section into the statement announced in
the introduction.

\begin{proof}[Proof of Theorem~\ref{thm:closing-rates}]

Proposition~\ref{prop:cone_moving_blowup}\ref{item:cone-blowup-hausdorff-dist} is exactly
\eqref{eq:intro-cylinder-limit}, so we restrict attention to items \ref{item:intro-closing-high-d}-\ref{item:intro-closing-intermediate}. Let $R_0$ be small enough that Lemma~\ref{lem:r_plus_upper_bd} applies, and thus $r_+(t)\leq C(t_0-t)^{1/2}$. By \eqref{eq:r-perp-r-plus-comparison}, $r_\perp(t)\leq r_+(t)$, so we conclude the upper bounds in items~\ref{item:intro-closing-high-d}
and~\ref{item:intro-closing-intermediate}.

The upper bound in \ref{item:intro-closing-d2} requires a more subtle argument to get the extra logarithmic factor. Since Proposition~\ref{prop:cone_moving_blowup}\ref{item:cone-blowup-hausdorff-dist} gives that $\frac{1}{r_\perp(t)}(\R^2\setminus \Omega_t)$ converges locally in Hausdorff distance to the cylinder $E \times V$, we deduce from Lemma~\ref{obst-quadratic-growth}\ref{item:obst-nondegeneracy} that there exists $a > 0$ such that $B_{a r_\perp(t)} \subset \R^2\setminus \Omega_t$ for all $t < t_0$ sufficiently close to $t_0$. By Lemma~\ref{lem:HS-basic-properties}(i), choose $R_*>R_0$
such that $\Omega_s\subset B_{R_*/2}$ for $0\leq s\leq t_0+1$, and set
$\bar c:=\norm{c}_{L^\infty(\R^2\times(0,t_0+1))}$. Fix $t_1\in (t_0-1, t_0)$, and let $\bar{p}$ solve \eqref{eq:intro-HS-law} with constant source $\bar c$ and annular initial phase $B_{R_*}\setminus \overline{B_{ar_\perp(t_1)}}$. Let $\bar \Omega_s$ denote the open phase at positive-time, and note by radial symmetry that there exist monotone functions $\bar R, \bar r$ such that $\bar \Omega_s = B_{\bar R(s)}\setminus \overline{B_{\bar r(s)}}$ for $0 < s < \bar \tau$, where $\bar \tau$ is the extinction time of the radial inner hole.

By Theorem~\ref{thm:HS-distributional-comparison}, we have $p(x,t_1 +s) \leq \bar p(x, s)$ for a.e. $(x,s)\in \R^d\times (0,\min\{\bar\tau,1\})$, and thus $ \Omega_{t_1+s} \setminus \bar \Omega_{s} $ has measure zero, by \eqref{eq:HS-open-phase-identities}. It follows that $\Omega_{t_1 + s}\subset \bar \Omega_s$ for such $s$, since both sets are open and $\bar \Omega_s$ has smooth boundary. Since $T(0)=t_0$, this inclusion implies $\bar\tau\leq t_0-t_1<1$. Now, $\bar p$ is given by \eqref{eq:radial_barrier_explicit}, so the motion law yields
\[
    -\dot{\bar r}
    =\frac{\bar c}{4}\left(
        \frac{\bar R^2-\bar r^2}{\bar r\log(\bar R/\bar r)}
        -2\bar r\right),
    \qquad
    0\leq\dot{\bar R}\leq\frac{\bar c}{2}\bar R.
\]
Since $\bar R(0)=R_*$ and $\bar\tau\leq t_0-t_1<1$, we have
$R_*\leq\bar R(s)\leq R_*e^{\bar c/2}$ for $0\leq s\leq\bar\tau$, and hence
$-\dot{\bar r}(s)\leq C/(\bar r(s)\log(R_*/\bar r(s)))$. Direct integration
now gives
\[
    t_0-t_1\geq\bar\tau
    \geq C^{-1}\int_0^{ar_{\perp}(t_1)}s\log\frac{R_*}{s}\,ds
    \geq C^{-1}r_{\perp}(t_1)^2\log\frac{R_0}{r_{\perp}(t_1)},
\]
which proves the upper bound in
\ref{item:intro-closing-d2}.

For the lower bounds in \ref{item:intro-closing-high-d}-\ref{item:intro-closing-intermediate}, \eqref{eq:r-perp-r-plus-comparison} gives that $\theta_\alpha r_+(t) \leq r_\perp(t)$, for $\theta_\alpha := \alpha(1 + \alpha^2)^{-1/2}$. We then take $t < t_0$ close enough that $r_+(t)/R_0<1/16$ and $4r_+(t)\leq1$, in order to invoke Lemma~\ref{lem:r_plus_lower_bd}. In the case $d\geq3$ and $k=0$, the lemma gives $t_0-t\leq Cr_+(t)^2$, and thus the lower bound in \ref{item:intro-closing-high-d} holds.
In the case $d=2$ and $k=0$, the lemma gives $t_0 - t \leq C r_+(t)^2 \log \frac{R_0}{r_+(t)}$, and so we have \[ r_\perp(t)^2 \log \frac{R_0}{r_\perp(t)} \geq \theta_\alpha^{2}r_+(t)^2\log\frac{R_0}{\theta_\alpha r_+(t)} \geq \theta_\alpha^{2}r_+(t)^2\log\frac{R_0}{r_+(t)} \geq \frac{\theta_\alpha^2 }{C}(t_0 - t), \]
using that $r_\perp(t) \leq r_+(t) \leq \frac{R_0}{16}$. Thus, the lower
bound in item~\ref{item:intro-closing-d2} also holds.

For the case $1 \leq k \leq d-2$, let $\epsilon > 0$ and choose $\tilde{\alpha} \in (0, \alpha)$ such that $\frac{1}{2 - \nu_{\tilde{\alpha}}} \leq \frac{1}{2} + \epsilon$, as is possible by
Lemma~\ref{lem:first_eigenvalue} and \eqref{eq:cone_exponent}. Let 
$\widetilde R_0\leq R_0$ such that Lemmas~\ref{lem:r_plus_upper_bd} and
\ref{lem:r_plus_lower_bd} and Proposition~\ref{prop:cone_moving_blowup} apply
with $(\widetilde\alpha,\widetilde R_0)$ in place of $(\alpha,R_0)$, and let
$\widetilde r_\perp$ be the radius \eqref{eq:r-perp-def} defined from
$(\widetilde\alpha,\widetilde R_0)$. Then Lemma~\ref{lem:r_plus_lower_bd} yields
\[
    \widetilde r_\perp(t)
    \geq\widetilde C^{-1}(t_0-t)^{\frac{1}{2-\nu_{\widetilde\alpha}}}
    \geq\widetilde C^{-1}(t_0-t)^{\frac12+\epsilon}.
\]
Moreover, $r_\perp(t)\geq(1-o(1))\,\widetilde r_\perp(t)$ as $t\uparrow t_0$.
Indeed, by Proposition~\ref{prop:cone_moving_blowup}(c), applied with
$(\widetilde\alpha,\widetilde R_0)$, the zero set rescaled by
$\widetilde r_\perp(t)$ converges to the normalized cylinder $E \times V$, which does
not depend on the aperture $\alpha$; local Hausdorff convergence near a point $(e,0)$
with $e\in E$ and $\abs e=1$ then produces zero points with
$\abs{x'}\geq(1-o(1))\,\widetilde r_\perp(t)$ and
$\abs{x''}=o(\widetilde r_\perp(t))$, which enter the supremum
\eqref{eq:r-perp-def} defining $r_\perp(t)$. This gives the lower bound in
\ref{item:intro-closing-intermediate} for $t_0-t$ sufficiently small depending
on $\epsilon$; since $r_\perp$ is positive and nonincreasing near $t_0$ while
$(t_0-t)^{\frac12+\epsilon}\leq1$, enlarging $C_\epsilon$ extends the bound to
the full stated range of $t$. Thus, we conclude the lower bound in \ref{item:intro-closing-intermediate}.

\end{proof}

\subsection{Pressure gradient blowup and hitting time flatness}
\label{subsec:pressure-gradient-hitting-time-flatness}

In this last subsection, we use the closing rates and asymptotic shape of the zero set to show blowup of the pressure gradient and one-sided flatness of the hitting time. These results amount to different ways of measuring the sense in which the free boundary moves with infinite speed near a $k$-dimensional singular point, when $k < d-1$. We begin with the proof of Corollary~\ref{prop:gradient-lq-blowup}, where we use a capacity-type estimate to show that $\nabla p$ blows up in $L^q_{x,t}$ for $q$ sufficiently large.

\begin{proof}[Proof of Corollary~\ref{prop:gradient-lq-blowup}]

    Fix $0<\delta<t_0$. For the first claim, let $\beta\in (0,1)$, and choose $\alpha$ sufficiently small so that the $\nu$ from Proposition~\ref{prop:pressure_lower_bd} is strictly smaller than $\beta$. Combining Lemma~\ref{lem:r_plus_upper_bd} and either Lemma~\ref{lem:radial_subbarrier} (in the case $k=0$) or Lemma~\ref{lem:cone_subbarrier} (in the case $1\leq k \leq d-2$), there exist constants $C_0, C_1, \delta_t > 0$ and $\varphi_0\in (\varphi_\alpha, \pi/2)$ such that $p\geq C_1^{-1}(t_0 - t)^{\beta/2}$ on $\partial B_{C_0(t_0-t)^{1/2}}\cap \{ \arctan(|x'|/|x''|) > \varphi_0\}$, whenever $0 < t_0 - t < \delta_t$. Since $p(0, t) = 0$, fundamental theorem of calculus along radial lines yields that
    \[ \int_{B_{C_0(t_0-t)^{1/2}}} |\nabla p|^q \gtrsim (t_0 - t)^{\tfrac{d - q(1-\beta)}{2}}.  \]
    Set $\delta':=\min\{\delta_t,\delta,C_0^{-2}\delta^2\}$. Then
    \[
    \int_{t_0-\delta}^{t_0}\int_{B_{\delta}}|\nabla p|^q
    \gtrsim
    \int_{t_0-\delta'}^{t_0}
    (t_0-t)^{\tfrac{d-q(1-\beta)}2}\,dt.
    \]
    This diverges whenever $q \geq \frac{d+2}{1-\beta}$. Since $\beta\in (0,1)$ is arbitrary, we conclude that $\int_{t_0-\delta}^{t_0}\int_{B_{\delta}} |\nabla p|^q = \infty$ for all $q > d + 2$.
    
    Moving to the second claim, by Theorem~\ref{thm:hitting-interface-regularity}\ref{i:singular-set-c1-manifold}, there exists a $C^1$ $k$-manifold $M$ which contains $\Sigma_k(t_0)$ in a neighborhood of 0, such that $T_xM = \ker D^2 p_2$ whenever $x\in \Sigma_k\cap M$ and $p_2$ is the blowup of $w(\cdot, T(x))$ at $x$. We shrink $\delta$ if necessary so that $M$ is the graph of a function $f$ over $V$ with $\|Df\|_{L^\infty} \leq \delta_M$ in $B_{\delta/3}$, for a small $\delta_M$ to be determined.
    
    For $x\in \Sigma_k(t_0)$, set $\lambda_1(x)$ as the smallest positive eigenvalue of the blowup of $w(t_0)$ at $x$, and set $A := \{x \in \Sigma_k(t_0)\cap M\cap B_{\delta/4} : \lambda_1(x) \geq \frac{1}{2}\lambda_1(0)\}$. By \cite[Proposition 5.13]{CJK25}, $\lambda_1(x)$ is continuous on $\Sigma_k(t_0)$. Thus, after decreasing $\delta$, the assumption that $\Sigma_k(t_0)$ has positive lower density at 0 implies that $\mathcal{H}^k(A) > 0$. Replacing $A$ by a compact subset of positive measure, we may assume that $A$ is compact.

    Fix $\beta\in (0,1)$, and as before, take $\alpha$ sufficiently small that $\nu < \beta$. Shrink $R_0$ to allow applying Lemma~\ref{lem:r_plus_upper_bd} at each point in $A$. Since $\lambda_1$ is uniformly positive on $A$, Lemma~\ref{lem:r_plus_upper_bd} gives a constant $C_0$ such that for each $x\in A$, we have $x + U(T_xM, \alpha, R_0, a/5) \subset \Omega_t$ for $a :=  C_0\sqrt{t_0  -t}$. Proposition~\ref{prop:pressure_lower_bd} then gives a $C_1$ and a $\delta_1$ such that for all $t$ with $0 < t_0 - t < \delta_1$, we have
    \[ p(x + ae,t) \geq C_1^{-1}(t_0 - t)^{\beta/2}, \qquad \forall e\in (T_xM)^\perp, |e| = 1. \]
    Moreover, we have $B_{a/2}(x + ae)\subset x + U(T_xM, \alpha, R_0, a/4)$ since for $y\in B_{a/2}(ae)$, we have $|y'| > \frac{a}{2} > \frac{\alpha a}{2} > \alpha|y''|$.
    The inhomogeneous Harnack inequality gives
    \[
    p(\cdot,t) \geq C(d)^{-1}C_1^{-1}(t_0-t)^{\beta/2}
        -C(d)\|c\|_\infty a^2
    \quad\hbox{in }B_{a/4}(x+ae).
    \]
    Since $a^2=C_0^2(t_0-t)=o((t_0-t)^{\beta/2})$, after decreasing $\delta_1$ this is at least $C_2^{-1}(t_0-t)^{\beta/2}$ for some $C_2>0$.
    
    Set $u(x) := \frac{C_2}{(t_0-t)^{\beta/2}}p(ax,t)$, so that $u = 0$ on $A_0 := a^{-1} A$ and $u\geq 1$ on $A_1 := \{x + e + y : x\in A_0, e\in (T_x(a^{-1}M))^\perp, |e|=1, |y| < \frac{1}{4}\}$. The zero value on $A_0$ is also the Sobolev trace, since continuity of $T$ gives $p(\cdot,t)=0$ near the compact set $A$. If $M$ was the graph of a function $f$, then the rescaled manifold is the graph of $a^{-1}f(a\cdot)$, and the $\delta_M$-Lipschitz bound is preserved under scaling. In Lemma~\ref{lem:cylindrical-q-capacity-est}, we will show that, after choosing $\delta_M = \frac{1}{8}$ and fixing a parameter $\eta > 0$, there exists a constant $C=C(d,k,q,\eta)$ such that $\int_{B_3(x)} |\nabla u|^q \geq C$ for any $x$ such that $\mathcal{H}^k(A_0\cap B_1(x)) \geq \eta$. 

    The $\delta_M$-Lipschitz bound for $M$ gives upper $\mathcal{H}^k$-density estimates for $A_0$. Thus, for $\eta$ sufficiently small, there is a constant $N(\delta_M,\eta)>0$ such that there exist at least $Na^{-k}$ disjoint unit balls $B_j$, centered in $A_0$, with $\mathcal{H}^k(A_0\cap B_j)\geq\eta$. For $a<\delta/12$, the concentric balls of radius $3$ lie in $B_{\delta/a}$ and have bounded overlap. If $u\notin W^{1,q}(B_{\delta/a})$, the desired lower bound is automatic; otherwise Lemma~\ref{lem:cylindrical-q-capacity-est}, applied after translation on each $B_j$, gives $\int_{B_{\delta/a}}|\nabla u|^q\gtrsim a^{-k}$. Substituting the definition of $u$ and integrating in time, we conclude that
    \[ \int_{t_0-\delta}^{t_0} \int_{B_{\delta}} |\nabla p|^q \gtrsim \int_{t_0 - \delta}^{t_0} (t_0-t)^{\tfrac{d-k-q(1 - \beta)}{2}}dt. \]
    This last integral diverges whenever $q \geq \frac{d-k+2}{1-\beta}$. Since $\beta\in (0,1)$ is arbitrary, it follows that $\int_{t_0-\delta}^{t_0} \int_{B_{\delta}} |\nabla p|^q = \infty$ whenever $q > d-k + 2$.
\end{proof}

It remains to check the following capacity estimate.
\begin{lem}\label{lem:cylindrical-q-capacity-est}
     Let $B''\subset V$ be the unit $k$-ball, and $f: B''\to V^\perp$ satisfy $0<\delta\leq1/8$ and $\|f\|_{C^1(B'')}\leq\delta$. Set $M = \{ (f(x''),x'') : x''\in B''\}$ so that $M$ is a $k$-manifold in $\R^d$. Let $A_0 \subset M$ with $\mathcal{H}^k(A_0) \geq \eta > 0$, and set $A_1 := \{ x + e + y : x\in A_0, e\in (T_xM)^\perp, |e| = 1, |y| < 2\delta \}$. Let $q>d-k$, and suppose $u\in W^{1,q}(B_3)$ has zero trace on $A_0$ and satisfies $u\geq1$ a.e. on $A_1$. Then there exists $C=C(d,k,q,\delta,\eta)>0$ such that
    \[ \int_{B_3} |\nabla u|^q \geq C. \]
\end{lem}
\begin{proof}
    Let $\theta$ be a unit vector in $V^\perp$, and note that $e = \frac{(\theta,-Df(x'')^T\theta)}{\sqrt{1 + |Df(x'')^T\theta|^2}}$ is normal to $M$ at $x=(f(x''),x'')$. This yields $|(\theta,0) - e|^2 = \frac{2|Df(x'')^T\theta|^2}{(1 + \sqrt{1 + |Df(x'')^T\theta|^2})\sqrt{1 + |Df(x'')^T\theta|^2}} \leq \delta^2$, using that $|Df|\leq \delta$. For $x\in A_0$, it follows that $x + (\theta,0) = x + e + ((\theta,0) - e)\in A_1$.

    Let $B'$ denote the unit $(d-k)$-ball, and observe that the map $B'\times B''\ni (y', y'')\mapsto (f(y'') + y', y'')\in B_3$ is a diffeomorphism onto its image, with Jacobian 1. Thus, we change coordinates to get
    \begin{align*}
        \int_{B_3} |\nabla u|^q &\geq \int_{B''} \int_{\partial B'}\int_0^1 |\nabla u((f(y'') + r\theta, y''))|^q r^{d-k-1}dr\, d\mathcal{H}^{d-k-1}(\theta)\,dy''.
    \end{align*}
    Let $E'' := \{ y''\in B'' : (f(y''), y'') \in A_0 \}$. For a.e. $y''\in E''$ and $\theta\in\partial B'$, the trace theorem and absolute continuity on lines give $u(f(y''),y'')=0$, $u(f(y'')+\theta,y'')\geq1$, and $\int_0^1 \nabla u(f(y'')+r\theta,y'')\cdot(\theta,0)\,dr\geq1$. With H\"older, this implies
    \begin{align*}
        1 &\leq \left( \int_0^1|\nabla u((f(y'') + r\theta, y''))|^q r^{d-k-1}dr \right)^{1/q}\left(\int_0^1 r^{-\frac{d-k-1}{q-1}}dr\right)^{\frac{q-1}{q}}
    \end{align*}
    Applying this to the main estimate yields
    \[ \int_{B_3} |\nabla u|^q \geq \left(\frac{q-d+k}{q - 1}\right)^{q-1}\mathcal{H}^{d-k-1}(\partial B')\mathcal{H}^k(E''). \]
    Since $A_0$ is the graph of $f$ over $E''$, we have $\mathcal{H}^k(A_0) = \int_{E''} \sqrt{\det(I + Df^TDf)}$. Thus, $\mathcal{H}^k(E'') \geq (1 + \delta^2)^{-k/2}\eta$, so we conclude.
\end{proof}

\begin{rem}\label{rem:lq-sharpness}
By the result of \cite{DavidPerthame}, which we adapt to our setting in Lemma~\ref{lem:tumor-pressure-L4}, we have $\nabla p\in L^4(\R^d\times (0,\tau))$ for every $\tau > 0$. Corollary~\ref{prop:gradient-lq-blowup} with $k = d-2$ shows that this bound is sharp.

The threshold $q= d-k+2$ is also critical for $k < d-2$. To see this, suppose $\Omega_t=(B_{R(t)}^{d-k}\setminus \overline{B_{r(t)}^{d-k}})\times \mathbb{T}^k$, and the source $c$ is constant. Then the pressure is $p(x,t) = \phi_t(x')$, where $\phi_t$ is given by \eqref{eq:radial_barrier_explicit} with $R_0 := R(t), r_0 := r(t), d:= d-k$. Since $d-k \geq 3$, we have $\nabla \phi_t \sim r(t)^{-1}$ at scale $r(t)$. This implies $r(t) \sim (t_0 - t)^{1/2}$ and thus $\iint |\nabla \phi_t|^qdxdt \gtrsim \int (t_0-t)^{\frac{d-k-q}{2}}dt$. Thus, $\iint |\nabla \phi_t|^{q}dxdt = \infty$ for $q\geq d-k+2$, diverging logarithmically in the critical case $q = d-k+2$. This differs from the critical behavior when $d-k = 2$, where $\iint |\nabla \phi_t|^4dxdt < \infty$ due to extra logarithmic factors from the Green's function (see \cite[Appendix C]{DavidPerthame}).
\end{rem}

\begin{prop}[Flatness of the hitting time]\label{prop:cone_hitting_time_upper_bd}
  Assume \eqref{eq:A-source}, \eqref{eq:A-initial}, \eqref{eq:A-source-time}, and
  \eqref{eq:closing-T-holder}.
  For any $\beta\in(0,1)$, there exists $r_{\beta}>0$  and $C > 0$ such that, for $|x|<r_{\beta}$,
  \begin{equation*}
    t_0-T(x)\leq C|x|^{2-\beta}.
  \end{equation*}
\end{prop}
\begin{proof}
    Use Lemma~\ref{lem:r_plus_lower_bd} and
    \eqref{eq:r-perp-r-plus-comparison} to choose $\alpha>0$, $R_0 > 0$,
    $C > 0$, and $t_1 < t_0$ such that
    \begin{equation}\label{eq:cone_r_perp_decay_gamma}
        t_0 - t \leq Cr_\perp(t)^{2-\beta},
        \quad t\in (t_1, t_0)
    \end{equation}

    Let $E$ be as in Proposition~\ref{prop:cone_moving_blowup}, so
    that we have
    \[
        \{ x : w(r_\perp(t)x, t) = 0\}
        \longrightarrow E\times V \hbox{ as }t\to t_0^-,
    \]
    where the convergence holds locally in Hausdorff distance between sets.
    The nondegeneracy estimate
    Lemma~\ref{obst-quadratic-growth}\ref{item:obst-nondegeneracy} then
    implies that there exists $b>0$, depending only on $E$, such that
    \begin{equation}\label{eq:cone_inner_ball_r_perp}
     B_{br_\perp(t)}\subset \{ w(\cdot,t) = 0\},
    \end{equation}
    for all $t$ sufficiently close to $t_0$.
    
    Fix $t_*<t_0$ sufficiently close to $t_0$ such that
    \eqref{eq:cone_inner_ball_r_perp} and
    \eqref{eq:cone_r_perp_decay_gamma} hold for $t\in[t_*,t_0]$. We now fix
    $r_{\beta}=\frac{b}{2}r_\perp(t_*)>0$.
    Let $0<|x|<r_\beta$ and set $\rho=2|x|/b$. By monotonicity of
    $r_\perp$, there exists $\tau\in[t_*,t_0)$ such that
    $r_\perp(t)\geq\rho$ for $t<\tau$ and
    $r_\perp(t)<\rho$ for $t>\tau$. Hence
    \eqref{eq:cone_inner_ball_r_perp} gives $T(x)\geq\tau$, while for
    $t>\tau$, \eqref{eq:cone_r_perp_decay_gamma} gives
    $t_0-t\leq C\rho^{2-\beta}$. Letting $t\downarrow\tau$, we conclude that
    \begin{equation*}
        t_0-T(x)\leq t_0-\tau\leq C'|x|^{2-\beta}.
    \end{equation*}
\end{proof}

\section{\texorpdfstring{Improvement of flatness and $C^1$ regularity}{Improvement of flatness and C1 regularity}}\label{sec:improvement-flatness}
In this section we prove Theorems~\ref{prop:cone_T_differentiability} and
\ref{thm:no_top_stratum_C1}. The proof has two ingredients, one local and one global. The
local ingredient is an intrinsic gradient estimate near a lower-stratum singularity,
obtained by comparing $p$ with directional derivatives of $w$, which yields
Theorem~\ref{prop:cone_T_differentiability}. The global ingredient is that lower-stratum
singularities are too small, in the sense of capacity, to affect the pressure nonlocally,
so that the spaces $H^1_0(\Omega_t)$ vary continuously in the Mosco sense; combined with
the local estimate, this yields Theorem~\ref{thm:no_top_stratum_C1}.

\subsection{Improvement of flatness for the hitting time near a singular point}

Throughout this subsection we work at a singular point \(0\in\Sigma_k(t_0)\),
\(0\le k\le d-2\). We continue to use coordinates, written
\[
    x=(x',x'')\in V^\perp\times V,
\]
where $V = \{ p_2 = 0\} = \ker D^2 p_2$, for $p_2$ the quadratic blowup \(p_2\) of
\(w(\cdot,t_0)\) at the origin. Let \(\lambda_1>0\) denote
the smallest positive eigenvalue of \(p_2\) on the \(x'\)-space. Fix a small parameter
\(0<\alpha<1\), with \(0<\alpha\leq\alpha_0(d,k)\) when \(k\geq1\).
Choose \(R_*>0\) so small that, for every \(0<R_0<R_*\),
one has \(t_0-R_0^2>0\), the hypotheses on \(R_0\) in
Lemma~\ref{lem:r_plus_upper_bd} are satisfied, and

\begin{equation}\label{eq:cone_local_mass_bound}
\begin{gathered}
    B_{R_0}\cap\Omega_0^{+}=\emptyset,\\
    0\leq \int_{T(x)\wedge s}^s c(x,\tau)\,d\tau\leq\frac12
    \qquad x\in B_{R_0},\quad |s-t_0|\leq R_0^2 .
\end{gathered}
\end{equation}

For \(0<R_0<R_*\) and \(0<r<R_0\), recall that \(U(\alpha,R_0,r)\) is the conical annulus
defined in \eqref{eq:u_alpha}.
The \(C^1\) convergence of \(w(\cdot,t_0)\) to \(p_2\), together
with \(p_2\simeq |x'|^2\) and \(|\nabla p_2|\leq C|x'|\), gives, uniformly
for \(x\in B_r\),
\begin{equation}\label{eq:cone_w0_bounds}
    w(x,t_0)\leq C|x'|^2+o(r^2),
    \qquad
    |\nabla w(x,t_0)|\leq C|x'|+o(r),
    \qquad r\downarrow0.
\end{equation}
The set \(U(\alpha,R_0,h)\) is also strictly
inside the positive phase at time \(t_0\):
\begin{equation}\label{eq:cone_Ur_positive_at_t0}
    w(x,t_0)\geq c_\alpha |x|^2,
    \qquad x\in B_{R_0},\ |x'|\geq \alpha |x''|.
\end{equation}
Thus \(U(\alpha,R_0,h)\subset \Omega_{t_0}\subset\Omega_t\) for every
\(0<h<R_0\) and \(t\geq t_0\).

In \(B_{R_0}\), the equation for \(w\), \eqref{eq:HS-Baiocchi-equation}, gives
\[
    \Delta w(\cdot,s)=f(\cdot,s)
    \qquad\hbox{in }\Omega_s\cap B_{R_0},
\]
where
\begin{equation}\label{eq:cone_f_definition}
    f(x,s):=1-\int_{T(x)\wedge s}^s c(x,\tau)\,d\tau .
\end{equation}

The bound \eqref{eq:cone_local_mass_bound} gives \(1/2\leq f\leq1\),
while \eqref{eq:cone_f_definition}, Theorem~\ref{thm:HS-hitting-time-lipschitz},
and the local spatial Lipschitz bound for \(c\) give the gradient estimate in
\begin{equation}\label{eq:cone_f_slice}
    \frac12\leq f(x,s)\leq 1,
    \qquad
    |\nabla_x f(x,s)|\leq L_f,
    \qquad
    (x,s)\in B_{R_0}\times [t_0-R_0^2,t_0+R_0^2].
\end{equation}
Set
\[
I_0= [t_0-R_0^2,t_0+R_0^2],\qquad \underline c:=\inf_{\R^d\times I_0}c,\qquad
    \overline c:=\sup_{\R^d\times I_0}c.
\]
Then \(0<\underline c\leq\overline c<\infty\).
\begin{lem}[Pressure lower bound off the spine]\label{lem:cone_pressure_away_spine}
Assume the hypotheses of Theorem~\ref{thm:HS-hitting-time-lipschitz}, and let
\(0\in\Sigma_k(t_0)\) with \(0\leq k\leq d-2\).
Write \(U(\alpha,R_0,r):=U(V,\alpha,R_0,r)\), as in \eqref{eq:u_alpha}. Set \(\gamma=\nu_\alpha\), with \(\nu_\alpha\) defined in \eqref{eq:cone_exponent}, if \(k\geq1\), and fix \(\gamma\in(0,1)\) if \(k=0\). Then for \(\delta\in(4\alpha,1)\), there exists $c_*=c_*(d,k,\alpha,\delta,\gamma,R_0, \underline c)>0$ with the following
property:  for any $r\in (0,\frac{R_0}{128})$, if 
\[
    U(\alpha,R_0,r)\subset\Omega_s \quad \hbox{ for some  }
s\in [t_0-R_0^2, t_0+R_0^2], \]
then, for any $R\in [16r, 32r]$, 
\begin{equation}\label{eq:cone_pressure_lower_away_spine}
    p(x,s)\geq c_*R^\gamma \hbox{ in } \quad \{R/2\leq |x|\leq R\}\cap \{|x'| \geq \delta R\}.
\end{equation}
\end{lem}
\begin{proof}
If \(k=0\), then \(U(\alpha,R_0,r)=B_{R_0}\setminus B_r\).  Comparing
\(p(\cdot,s)/\underline c\) with the radial annular barrier used in
\eqref{eq:radial_barrier_explicit}, with \(r\) in place of \(r_+(t)\), gives
\[
    p(x,s)\geq c\underline c R_0^2
    \begin{cases}
        \bigl(\log(R_0/R)\bigr)^{-1}, & d=2,\\
        1, & d\geq3,
    \end{cases}
\]
for \(R/2\leq |x|\leq R\).  After
decreasing the constant, this is at least \(c_*R^\gamma\) for
\(0<R<R_0/4\).

We may therefore assume \(k\geq1\). The assumptions give \(4r<|x|<R_0/4\). Also,
\(|x'|\geq\delta R\), \(|x|\leq R\), and \(\delta>4\alpha\) imply
\(|x'|>\alpha |x''|\). Thus \(x\in U(\alpha,R_0,r)\), and
Lemma~\ref{lem:cone_subbarrier} applies to \(p(\cdot,s)/\underline c\) in
\(U(\alpha,R_0,r)\), using \eqref{eq:HS-pointwise-pressure-equation}.

Since \(\{\varphi:\sin\varphi\geq\delta\}\) is compactly contained in
\((\varphi_\alpha,\pi/2]\), Lemma~\ref{lem:first_eigenfunction} gives
\(\Phi_\alpha\geq m_\delta>0\) on this set. The lower bound follows from
\(|x|\geq R/2\), after decreasing \(c_*\).
\end{proof}

\begin{prop}[Intrinsic window gradient estimate]\label{lem:cone_gradient_intrinsic_window}
Assume the hypotheses of Theorem~\ref{thm:HS-hitting-time-lipschitz}, let
\(0\in\Sigma_k(t_0)\) with \(0\leq k\leq d-2\), and let \(0<R_0<R_*\) be as fixed above.
There exists
\(\bar\alpha\in(0,1)\), depending only on \(d,k\), the constant in
\eqref{eq:cone_w0_bounds}, and the constants in \eqref{eq:cone_f_slice}, such
that if \(0<\alpha\leq\bar\alpha\), then the following holds.
Let \(\gamma=\nu_\alpha\) if \(k\geq1\), and let
\(\gamma\in(0,1)\) if \(k=0\).
For every \(\sigma\in(\gamma,1)\), there exist constants
\(c_\sigma\in(0,1)\) and \(r_\sigma\in(0,R_0)\) such that, for every \(R\in(0,r_\sigma)\),
\begin{equation}\label{eq:cone_gradient_intrinsic_window}
    |\nabla w(x,t)|\leq R^{1-\sigma} p(x,t),
    \qquad
    (x,t)\in B_{R/4}\times [t_0-c_\sigma R^2,t_0+c_\sigma R^{2-\sigma}].
\end{equation}
\end{prop}
\begin{proof}
Fix \(\sigma\in(\gamma,1)\) and fix a unit vector \(e\). It suffices to estimate the directional derivative $\partial_e w$, with constants independent of \(e\). We begin by choosing some parameters. Let
\[
    K:=
    \left(\frac{6}{\lambda_1\alpha^2}
    \|p\|_{L^\infty(B_{R_0}\times(t_0-R_0^2,t_0))}
    \right)^{1/2}.
\]
Choose \(c_\sigma\in(0,1)\) so small that
\begin{equation}\label{eq:cone_csigma_choice}
    K\sqrt{c_\sigma}\leq \frac1{32}.
\end{equation} Next choose \(r_\sigma\in(0,R_0/4)\) so small that, whenever
\(0<R<r_\sigma\),
\[
    [t_0-c_\sigma R^2,t_0+c_\sigma R^{2-\sigma}]
    \subset (t_0-R_0^2,t_0+R_0^2).
\]
Thus \eqref{eq:cone_f_slice} applies for all times in this interval.

Fix \(x_0\in B_{R/4}\) and \(t\in[t_0-c_\sigma R^2,t_0+c_\sigma R^{2-\sigma}]\). If
\(w(x_0,t)=0\), then \(\nabla w(x_0,t)=0\), and there is nothing to prove. We therefore assume
\(w(x_0,t)>0\). Let \(\varepsilon_j\downarrow0\), with
\(\varepsilon_j<w(x_0,t)\), and choose smooth open sets $\Omega_j$ such that 
\begin{equation}\label{eq:cone_Omega_j}
    B_{R/2}\cap\{w(\cdot,t)>\varepsilon_j\}\subset\Omega_j
    \subset B_R\cap\Omega_t.
\end{equation}

For nonnegative \(C^{1,1}\) functions we use
\begin{equation}\label{eq:cone_c11_sqrt}
    |\nabla u|^2\leq C u,
\end{equation}
with \(C\) controlled by the local \(C^{1,1}\) norm. Applying this to
\(u=w(\cdot,t)\) gives
\begin{equation}\label{eq:cone_inner_boundary_smallness}
    0\leq w(x,t)\leq\varepsilon_j,
    \qquad
    |\partial_e w(x,t)|\leq C\sqrt{\varepsilon_j},
    \qquad x\in\partial\Omega_j\cap B_{R/2}.
\end{equation}

Let \(G_{\Omega_j}\) be the Green function for \(-\Delta\) in \(\Omega_j\), and
set
\[
    I_j:=\int_{\Omega_j}G_{\Omega_j}(x_0,y)\,dy,
\]
\[
    c_j:=
    \frac{\int_{\Omega_j}f(y,t)G_{\Omega_j}(x_0,y)\,dy}{I_j},
    \qquad
    a_j:=
    \frac{\int_{\Omega_j}\partial_e f(y,t)G_{\Omega_j}(x_0,y)\,dy}{I_j}.
\]
We define the correctors \(\psi_j\) and \(\psi_j^e\) by
\[
    \Delta\psi_j=f(\cdot,t)-c_j,
    \qquad
    \Delta\psi_j^e=\partial_e f(\cdot,t)-a_j,
    \qquad
    \psi_j=\psi_j^e=0\hbox{ on }\partial\Omega_j.
\]
By \eqref{eq:cone_f_slice}, \(c_j\in[1/2,1]\) and \(|a_j|\leq L_f\).  The
normalization by the Green function gives
\begin{equation}\label{eq:cone_corrector_center}
    \psi_j(x_0)=\psi_j^e(x_0)=0.
\end{equation}

Define
\[
    Q(x):=\frac{|x-x_0|^2}{2d},
    \qquad
    q_j:=\partial_e w(\cdot,t)-\psi_j^e-a_jQ,
\]

With \(\kappa>0\) to be chosen,
we consider the following function:

\begin{equation}\label{eq:cone_barrier}
    z_j:=p(\cdot,t)-R^{-(1-\sigma)}q_j
    +\kappa R^{-(2-\sigma)}\bigl(c_jQ-w(\cdot,t)+\psi_j\bigr).
\end{equation}
Note that $z_j$ is superharmonic in $\Omega_j$, since  \(\Delta w(\cdot,t)=f(\cdot,t)\) and
the pressure equation \eqref{eq:HS-pointwise-pressure-equation}
give
\[
    \Delta q_j=0,
    \qquad
    \Delta\bigl(c_jQ-w(\cdot,t)+\psi_j\bigr)=0 \qquad \hbox{ in } \Omega_j.
\]
Our goal is to estimate the lower bound of $z_j$ on $\partial\Omega_j$, which then yields the desired estimate at $x=x_0$ by superharmonicity.
On \(\partial\Omega_j\), the correctors vanish, and therefore
\[
    z_j=p(\cdot,t)-R^{-(1-\sigma)}\partial_e w(\cdot,t)
    +\bigl(\kappa c_jR^{-(2-\sigma)}+a_jR^{-(1-\sigma)}\bigr)Q
    -\kappa R^{-(2-\sigma)}w(\cdot,t).
\]
After \(\kappa\) is fixed, we decrease \(r_\sigma\) so that
\(|a_j|R\leq\kappa/4\) for \(R<r_\sigma\). Since \(c_j\geq1/2\),
\begin{equation}\label{eq:cone_boundary_common}
    z_j\geq p(\cdot,t)-R^{-(1-\sigma)}\partial_e w(\cdot,t)
    +\frac{\kappa}{4}R^{-(2-\sigma)}Q
    -\kappa R^{-(2-\sigma)}w(\cdot,t)
\end{equation}
on \(\partial\Omega_j\). On \(\partial\Omega_j\cap B_{R/2}\), this and
\eqref{eq:cone_inner_boundary_smallness} give
\begin{equation}\label{eq:cone_inner_boundary_zj}
    z_j\geq
    -CR^{-(1-\sigma)}\sqrt{\varepsilon_j}
    -\kappa R^{-(2-\sigma)}\varepsilon_j.
\end{equation}

It remains to estimate \(z_j\) on the outer boundary 
\(\partial\Omega_j\cap\{|x|\geq R/2\}\), where $Q\geq c_dR^2$. For this part we proceed by different arguments at past times $t<t_0$, using results from the barrier argument, and at future times $t>t_0$, using the right continuity of $p$.

\textbf{\emph{Backward window.}}
Let $x\in \partial\Omega_j\cap \{|x|\geq R/2\}$.
Lemma~\ref{lem:r_plus_upper_bd} gives
\begin{equation}\label{eq:cone_r_plus_small_on_window}
    r_+(t)\leq K\sqrt{t_0-t}\leq \frac{R}{32},
    \qquad t\in[t_0-c_\sigma R^2,t_0].
\end{equation}

 Define
\[
    v:=w(\cdot,t_0)-w(\cdot,t)\geq0.
\]
Using \eqref{eq:cone_c11_sqrt} for \(v\), followed by Young's inequality, we have
\begin{equation}\label{eq:cone_backward_absorb}
    \kappa R^{-(2-\sigma)}v+R^{-(1-\sigma)}\partial_ev
    \geq -C\kappa^{-1}R^\sigma.
\end{equation}
Combining \eqref{eq:cone_boundary_common}, \(w(\cdot,t)=w(\cdot,t_0)-v\),
\eqref{eq:cone_w0_bounds}, and \eqref{eq:cone_backward_absorb}, we obtain
\begin{equation}\label{eq:cone_backward_outer_common}
    \begin{aligned}
    z_j(x)\geq{}&p(x,t)+c_d\kappa R^\sigma
    -CR^{-(1-\sigma)}|x'|
    -C\kappa R^{-(2-\sigma)}|x'|^2\\
    &-o(R^\sigma)-\kappa o(R^\sigma)
    -C\kappa^{-1}R^\sigma-CR^{1+\sigma}.
    \end{aligned}
\end{equation}

We now choose the threshold \(\delta_0\) separating the near-spine and
away-from-spine parts of the boundary in our case analysis. Let \(C_*\) bound
the coefficient of \(\kappa R^{-(2-\sigma)}|x'|^2\) in
\eqref{eq:cone_backward_outer_common}. Take \(\bar\alpha\) small enough that \(8\bar\alpha<1\) and
\(C_*(8\bar\alpha)^2\leq c_d/4\), and choose
\(\delta_0\in(4\alpha,8\bar\alpha)\). Then
\begin{equation}\label{eq:cone_delta_choice}
    4\alpha<\delta_0,
    \qquad
    C_*\delta_0^2\leq c_d/4 .
\end{equation}
Let \(c_*\) be the constant from
Lemma~\ref{lem:cone_pressure_away_spine} with \(\delta=\delta_0\). If \(|x'|\leq\delta_0R\), then
\begin{equation}\label{eq:cone_near_backward}
    z_j(x)\geq p(x,t)
    +\bigl(c_d\kappa-C\delta_0-C\kappa\delta_0^2-C\kappa^{-1}-o(1)-\kappa o(1)\bigr)R^\sigma
    -CR^{1+\sigma}.
\end{equation}
By \eqref{eq:cone_delta_choice}, the coefficient of \(R^\sigma\) in
\eqref{eq:cone_near_backward} is at least
\[
    \frac{c_d}{2}\kappa-C\delta_0-C\kappa^{-1}-o(1)-\kappa o(1),
\]
which is positive once \(\kappa\) is chosen large and \(r_\sigma\) is
decreased so that the \(o(1)\) terms are small for \(R<r_\sigma\).

If \(|x'|\geq\delta_0R\), set
\(r:=r_+(t)+R/32\). By
\eqref{eq:cone_r_plus_small_on_window}, \(R/32\leq r\leq R/16\), and by
the definition of \(r_+\), \(U(\alpha,R_0,r)\subset\Omega_t\).
Lemma~\ref{lem:cone_pressure_away_spine}, applied with \(s=t\), gives
\[
    p(x,t)\geq c_*R^\gamma.
\]
Since \(|x'|\leq R\), \eqref{eq:cone_backward_outer_common} yields
\begin{equation}\label{eq:cone_far_backward}
    z_j(x)\geq
    c_*R^\gamma-C(\kappa+\kappa^{-1})R^\sigma-CR^{1+\sigma}.
\end{equation}
Because \(\gamma<\sigma\), \eqref{eq:cone_near_backward} and
\eqref{eq:cone_far_backward} are nonnegative once \(r_\sigma\) is small enough.

The minimum principle and \eqref{eq:cone_inner_boundary_zj} give
\[
    z_j(x_0)\geq
    -CR^{-(1-\sigma)}\sqrt{\varepsilon_j}
    -\kappa R^{-(2-\sigma)}\varepsilon_j.
\]
By \eqref{eq:cone_corrector_center},
\[
    z_j(x_0)=p(x_0,t)-R^{-(1-\sigma)}\partial_e w(x_0,t)
    -\kappa R^{-(2-\sigma)}w(x_0,t).
\]
Letting \(j\to\infty\) gives
\[
    \partial_e w(x_0,t)\leq R^{1-\sigma}p(x_0,t).
\]
Replacing \(e\) by \(-e\), we conclude \eqref{eq:cone_gradient_intrinsic_window} for $t\leq t_0$.

\textbf{\emph{Forward window.}}
We next estimate $z_j$ on $\partial\Omega_j \cap \{|x|\geq R/2\}$ for the forward times $t>t_0$ with the same choice of $\delta_0$ and $\kappa$ as for backward times, but with potentially smaller choices of \(c_\sigma\) and
\(r_\sigma\) made below. 
If \(t=t_0+\tau\) with
\(\tau\leq c_\sigma R^{2-\sigma}\), set
\[
    v:=w(\cdot,t)-w(\cdot,t_0)=\int_{t_0}^t p(\cdot,s)\,ds\geq0.
\]
Integrating the slice comparison
\eqref{eq:HS-slice-pressure-comparison} in \(s\in(t_0,t)\), and passing to the
canonical representative from Proposition~\ref{prop:HS-USC-representative}, gives
\begin{equation}\label{eq:cone_forward_v_bound}
    v(x)=\int_{t_0}^t p(x,s)\,ds
    \leq \frac{\overline c}{\underline c}\tau p(x,t).
\end{equation}
By \eqref{eq:cone_forward_v_bound} and \(\tau\leq c_\sigma R^{2-\sigma}\), we have
\(v\leq(\overline c/\underline c)\,c_\sigma R^{2-\sigma}p(\cdot,t)\). We decrease
\(c_\sigma\), depending on \(\kappa\) and \(\overline c/\underline c\), so that
\(C\kappa\,(\overline c/\underline c)\,c_\sigma\leq\tfrac12\), where \(C\) is the constant
in \eqref{eq:cone_forward_absorb} below.
Since the \(C^{1,1}\) estimate \eqref{eq:cone_c11_sqrt} for \(v\) gives
\begin{equation}\label{eq:cone_forward_absorb}
    -\kappa R^{-(2-\sigma)}v-R^{-(1-\sigma)}\partial_ev
    \geq
    -C\kappa^{-1}R^\sigma-C\kappa R^{-(2-\sigma)}v,
\end{equation}
it follows 
\begin{equation}\label{eq:cone_forward_outer_common}
    \begin{aligned}
    z_j(x)\geq{}&\frac12p(x,t)+c_d\kappa R^\sigma
    -CR^{-(1-\sigma)}|x'|
    -C\kappa R^{-(2-\sigma)}|x'|^2\\
    &-o(R^\sigma)-\kappa o(R^\sigma)
    -C\kappa^{-1}R^\sigma-CR^{1+\sigma}.
    \end{aligned}
\end{equation}

When \(|x'|\leq\delta_0R\), the choices of \(\delta_0\) in
\eqref{eq:cone_delta_choice}, of \(\kappa\) after
\eqref{eq:cone_near_backward}, and of \(c_\sigma\)  make the right-hand side of
\eqref{eq:cone_forward_outer_common} nonnegative, after the final decrease of the range of $R$ given by 
\(r_\sigma\). When \(|x'|\geq\delta_0R\), set \(h=R/20\). By
\eqref{eq:cone_Ur_positive_at_t0} and monotonicity,
\[
    U(\alpha,R_0,h)\subset\Omega_{t_0}\subset\Omega_t,
\]
and Lemma~\ref{lem:cone_pressure_away_spine}, applied with \(s=t\) and \(r=h\),
gives
\[
    p(x,t)\geq c_*R^\gamma.
\]
Since \(\gamma<\sigma\), the positive pressure term in
\eqref{eq:cone_forward_outer_common} dominates the remaining \(R^\sigma\) and
\(R^{1+\sigma}\) errors for \(R<r_\sigma\). From here the minimum
principle and the passage \(j\to\infty\) proceed exactly as in the backward
window, yielding \(\partial_e w(x_0,t)\leq R^{1-\sigma}p(x_0,t)\); replacing \(e\)
by \(-e\) and taking the supremum over unit directions proves
\eqref{eq:cone_gradient_intrinsic_window}.
\end{proof}

To apply Proposition~\ref{lem:cone_gradient_intrinsic_window} at progressively smaller spatial scales, we first narrow down, by bootstrapping, the time range containing \(T(B_r)\). This also improves the resulting Lipschitz modulus at each step.
\begin{proof}[Proof of Theorem~\ref{prop:cone_T_differentiability}]
Fix \(\varepsilon\in(0,1)\). Choose the cone parameter
\(0<\alpha\leq\bar\alpha\), and also \(\alpha\leq\alpha_0(d,k)\) when
\(k\geq1\).  If \(k\geq1\), choose \(\alpha\) so small that
\(\nu_\alpha<\varepsilon\), which is possible by
Proposition ~\ref{lem:first_eigenvalue}, and set \(\gamma=\nu_\alpha\).  If \(k=0\),
choose \(\gamma\in(0,\varepsilon)\).  With this choice of \(\alpha\), take
\(R_0<R_*\).  Finally choose
\(\sigma\) with
\[
    \gamma<\sigma<\varepsilon .
\]
Let \(c_\sigma,r_\sigma\) be the constants from
Proposition~\ref{lem:cone_gradient_intrinsic_window}.  Define
\[
    \Phi_\sigma(\eta):=\frac{(1-\sigma)(1+\eta)}{2-\sigma}.
\]
We will repeatedly use the following elementary observation. Since $w$ is
locally Lipschitz in spacetime, if $|\nabla w|\leq A p=A \partial_t w$ in a spacetime
cylinder $U\times I$ and $[x,y]\subset U$, then
\[
    s\mapsto w(x+s(y-x),t+sA|x-y|)
\]
is nondecreasing whenever $t+sA|x-y|\in I$ for $s\in[0,1]$. Equivalently,
\begin{equation}\label{eq:cone_spacetime_monotonicity}
    w(y,t+A|x-y|)\geq w(x,t)
\end{equation}
provided $t,t+A|x-y|\in I$. Thus, if $[x,y]\subset U$, $T(x),T(y)\in I$, and
$\max\{T(x),T(y)\}+A|x-y|\in I$, then
\begin{equation}\label{eq:cone_T_lipschitz_from_gradient}
    |T(x)-T(y)|\leq A|x-y|.
\end{equation}
In the two applications that follow, \eqref{eq:cone_gradient_intrinsic_window}
supplies this hypothesis with \(A=R^{1-\sigma}\) on
\[
    B_{R/4}\times[t_0-c_\sigma R^2,t_0+c_\sigma R^{2-\sigma}].
\]

The rest of the proof is a bootstrap on the exponent \(\eta\) in estimates of the
form
\begin{equation}\label{eq:cone_future_bound_eta}
    T(z)\leq t_0+r^{1+\eta},\qquad z\in B_{r/8},
\end{equation}
valid for all sufficiently small \(r\). We first establish
\eqref{eq:cone_future_bound_eta} with \(\eta=0\). Choose \(M_0\) so large that
$c_\sigma M_0^{2-\sigma}>2$, and set
\(R=M_0r^{1/(2-\sigma)}\). For \(r\) small, \(R<r_\sigma\),
$B_{r/8}\subset B_{R/4}$, and \eqref{eq:cone_gradient_intrinsic_window} applies
up to time \(t_0+r\). If \(z\in B_{r/8}\), applying
\eqref{eq:cone_spacetime_monotonicity} with $x=0$, $y=z$, and
$t=t_0+\tau$, then letting $\tau\downarrow0$, gives
\[
    T(z)\leq t_0+R^{1-\sigma}|z|\leq t_0+r.
\]
Thus \eqref{eq:cone_future_bound_eta} holds with \(\eta=0\).

Assume now that \eqref{eq:cone_future_bound_eta} holds
for some \(\eta\in[0,1-\sigma)\) and all sufficiently small \(r\), and let
\(b<\Phi_\sigma(\eta)\). Put
\[
    q:=\frac{1+\eta}{2-\sigma},\qquad R:=Mr^q,\qquad A:=R^{1-\sigma},
\]
where $M$ is fixed so large that $c_\sigma M^{2-\sigma}>4$. Since $q<1$, for
small $r$ we have
$R<r_\sigma$ and $B_{r/8}\subset B_{R/4}$. Choose $\theta\in(0,1)$ so small
that $2-\theta>2q$. Proposition~\ref{prop:cone_hitting_time_upper_bd} gives
\[
    T(z)\geq t_0-C_\theta r^{2-\theta}\geq t_0-\frac12c_\sigma R^2,
    \qquad z\in B_{r/8},
\]
after reducing $r$. On the future side, \eqref{eq:cone_future_bound_eta} gives
\[
    \max\{T(x),T(y)\}+A|x-y|
    \leq t_0+r^{1+\eta}+\frac14M^{1-\sigma}r^{1+\Phi_\sigma(\eta)}
    \leq t_0+\frac12c_\sigma R^{2-\sigma}
\]
for $x,y\in B_{r/8}$, after decreasing $r$. Hence
\eqref{eq:cone_T_lipschitz_from_gradient}, used with
\eqref{eq:cone_gradient_intrinsic_window} at scale $R$, gives
\[
    |T(x)-T(y)|\leq A|x-y|.
\]
Since \(A=M^{1-\sigma}r^{\Phi_\sigma(\eta)}\leq r^b\) for small \(r\), we obtain
\begin{equation}\label{eq:cone_T_bootstrap_modulus}
    |T(x)-T(y)|\leq r^b|x-y|,
    \qquad x,y\in B_{r/8}.
\end{equation}
Taking $y=0$ in \eqref{eq:cone_T_bootstrap_modulus} also gives
\eqref{eq:cone_future_bound_eta} with \(b\) in place of \(\eta\).

Let \(\eta_0=0\) and
\(\eta_{n+1}=\Phi_\sigma(\eta_n)\). The preceding
step shows by induction that, for each \(n\geq1\) and each
\(b<\eta_n\), both
\eqref{eq:cone_future_bound_eta} with \(\eta=b\) and
\eqref{eq:cone_T_bootstrap_modulus} with exponent \(b\) hold. Since
\[
    \eta_n=(1-\sigma)\left(1-\left(\frac{1-\sigma}{2-\sigma}\right)^n\right)
    \uparrow 1-\sigma,
\]
and \(1-\varepsilon<1-\sigma\), \eqref{eq:cone_T_near_sharp_modulus} follows.

Taking $y=0$ in \eqref{eq:cone_T_near_sharp_modulus}, with $r=16|x|$, yields
\[
    |T(x)-T(0)|\leq (16|x|)^{1-\varepsilon}|x|=o(|x|).
\]
Therefore $T$ is differentiable at $0$ and $\nabla T(0)=0$. If $x_j\to0$ and
$T$ is differentiable at $x_j$, applying
\eqref{eq:cone_T_near_sharp_modulus} with $r=16|x_j|$ in a neighborhood of
$x_j$ gives
\[
    |\nabla T(x_j)|\leq (16|x_j|)^{1-\varepsilon}\to0.
\]
\end{proof}

\subsection{\texorpdfstring{$C^1$}{C1} regularity without top-stratum singularities}

We recall from \S\ref{subsec:intro-setting} the notation
\(O_I=\{x:T(x)\in I\}\) for \(I\Subset(0,\infty)\), the regular set
\(\mathcal R\), the singular set \(\Sigma\), and the stratification
\eqref{eq:intro-strata}.  For \(t>0\), we write
\(\Sigma(t):=\Sigma\cap\{T=t\}\) and
\(\Sigma_k(t):=\Sigma_k\cap\{T=t\}\) for the corresponding time slices.  We also write
\[
    \Sigma_{\leq d-2}
    :=\{q\in\R^d:T(q)<\infty,\ q\in \Sigma_k
    \text{ for some }0\leq k\leq d-2\},
\]
for the singular set below the top stratum.

Next we discuss the $C^1$ regularity of $T$ away from the hitting times of $\Sigma_{d-1}$-points. To this end we need to analyze the behavior of the velocity field $\nabla p$ in space-time neighborhoods of regular points. We begin by recalling the notion of Mosco convergence \cite{Mosco69}.

\begin{defn}[Mosco convergence]\label{defn:mosco}
Let $V_j,V$ be closed subspaces of a Hilbert space $H$. We say that $V_j\to V$ in the Mosco sense if:
\begin{enumerate}
\item[(i)] every $v\in V$ admits $v_j\in V_j$ with $v_j\to v$ strongly in $H$;
\item[(ii)] every weak limit of a sequence $v_{j_k}\in V_{j_k}$ belongs to $V$.
\end{enumerate}
\end{defn}

We view $H^1_0(D)$, for $D\subset\R^d$ open, as a subspace of $H^1(\R^d)$ by zero extension.

The main reason for introducing Mosco convergence is the following elementary stability property for the Dirichlet energy (see also \cite[Thm.~3.3]{Daners03}).

\begin{lem}[Dirichlet stability under Mosco convergence]\label{lem:mosco-dirichlet-stability}
Let $D_j,D\subset B_R$ be open, and suppose that $H^1_0(D_j)\to H^1_0(D)$ in the Mosco sense in $H^1(\R^d)$. Let $f_j\to f$ strongly in $H^{-1}(\R^d)$. If $u_j\in H^1_0(D_j)$ and $u\in H^1_0(D)$ solve
\[
    \int \nabla u_j\cdot\nabla\varphi=\langle f_j,\varphi\rangle
    \quad\hbox{for all }\varphi\in H^1_0(D_j),
    \qquad
    \int \nabla u\cdot\nabla\varphi=\langle f,\varphi\rangle
    \quad\hbox{for all }\varphi\in H^1_0(D),
\]
then $u_j\to u$ strongly in $H^1(\R^d)$.
\end{lem}
\begin{proof}
Since $H^1_0(D_j)\subset H^1_0(B_R)$, the Poincar\'e inequality and the
equation tested with $u_j$ show that $(u_j)$ is bounded in $H^1(\R^d)$.
If, along a subsequence, $u_j\rightharpoonup \bar u$ in $H^1(\R^d)$, then
Mosco condition~(ii) gives $\bar u\in H^1_0(D)$.  For each
$\varphi\in H^1_0(D)$, condition~(i) provides
$\varphi_j\in H^1_0(D_j)$ with $\varphi_j\to\varphi$ strongly in
$H^1(\R^d)$.  Passing to the limit in the equation for $u_j$ shows that
$\bar u$ solves the limiting problem, so uniqueness gives $\bar u=u$.
Thus $u_j\rightharpoonup u$ in $H^1(\R^d)$.  Testing again with $u_j$ yields
\[
    \int |\nabla u_j|^2=\langle f_j,u_j\rangle
    \longrightarrow \langle f,u\rangle=\int |\nabla u|^2.
\]
Hence $\nabla u_j\to\nabla u$ strongly in $L^2(\R^d)$, and the conclusion follows from the Poincar\'e
inequality.
\end{proof}
\begin{lem}[Mosco convergence without top stratum]\label{lem:no_top_mosco}
Assume the hypotheses of Theorem~\ref{thm:HS-hitting-time-lipschitz}.
Let $I=(a,b)\Subset(0,\infty)$, with \(\Omega_t\) as in
\eqref{eq:HS-open-phase-bridge}, and assume that no point with hitting time in $I$ belongs to $\Sigma_{d-1}$. Fix $t_*\in I$. Then
\[
    H^1_0(\Omega_t)\to H^1_0(\Omega_{t_*})
\]
in the Mosco sense in $H^1(\R^d)$ as $t\to t_*$.
Moreover, for each $t\in I$,
\[
    \{u\in H^1(\R^d):u=0\hbox{ a.e. on }\R^d\setminus\Omega_t\}
    =H^1_0(\Omega_t).
\]
\end{lem}
\begin{proof}
For every open \(D\subset\R^d\), the capacity characterization of
\cite[Thm.~9.1.3]{AdamsHedberg} gives
\begin{equation}\label{eq:h10-qe-characterization}
    H^1_0(D)=\{v\in H^1(\R^d):v=0\hbox{ q.e. on }\R^d\setminus D\}.
\end{equation}
Here q.e. means {\it quasi-everywhere,} that is, outside a set of Sobolev
\(2\)-capacity zero; \(H^1\) functions are understood through their
quasi-continuous representatives.  We recall also that Sobolev capacity is outer regular by
\cite[Prop.~2.3.5]{AdamsHedberg}.  Moreover, an \(H^1\) function vanishing on
one side of a Lipschitz graph has zero trace on the graph.
To prove (i), note that continuity of $w$ implies that each $\phi\in C_c^\infty(\Omega_{t_*})$ also belongs to $C_c^\infty(\Omega_t)$ when $t$ is close to $t_*$. Density of these test functions in $H^1_0(\Omega_{t_*})$ gives the required strong approximation.

To prove (ii), let $t_j\to t_*$ and let $u_j\in H^1_0(\Omega_{t_j})$ converge weakly in $H^1(\R^d)$ to $u$. If $t_j\leq t_*$ along a subsequence, monotonicity gives $\Omega_{t_j}\subset\Omega_{t_*}$, hence $u\in H^1_0(\Omega_{t_*})$. Only the case $t_j>t_*$ remains.

Choose $R$ so large that $\Omega_t\Subset B_R$ for $t$ near $t_*$. Then $u_j=0$ q.e. in $\R^d\setminus B_R$ for all large $j$, hence also $u=0$ q.e. in $\R^d\setminus B_R$. By Theorem~\ref{thm:hitting-interface-regularity} and the absence of \(\Sigma_{d-1}\), \(\Sigma(t_*)\) is locally contained in countably many \(C^1\) manifolds of dimension at most \(d-2\), $\Sigma(t_*)$ has zero capacity \cite[Cor. 5.1.15]{AdamsHedberg}.

By outer regularity, choose an open set $U\Subset B_R$ containing $\Sigma(t_*)$ with arbitrarily small relative capacity. The compact set $(\partial\Omega_{t_*}\cap B_R)\setminus U$ consists of regular points. Finitely many graph cylinders $Q_\ell$, chosen compactly inside the open regular set, cover it. After shrinking the cylinders and the time window, each $Q_\ell$ carries a single graph representation of $\partial\Omega_t$ for all nearby times, with $C^1$ spatial regularity and Lipschitz dependence on $t$. The graph heights converge uniformly by the Lipschitz bound, while normal continuity along the hitting graph gives convergence of the normals; hence the graphs converge to $\partial\Omega_{t_*}$ in $C^1$.

By continuity of $T$, the weak limit $u$ vanishes q.e. in the open exterior
$B_R\setminus\overline{\Omega_{t_*}}$. In each $Q_\ell$, applying the trace
statement for $H^1$ functions across the Lipschitz graph
$\partial\Omega_{t_*}\cap Q_\ell$ shows that the quasi-continuous representative
of $u$ also vanishes q.e. on that graph. Since the cylinders cover
$(\partial\Omega_{t_*}\cap B_R)\setminus U$, we get $u=0$ q.e. on
$B_R\setminus(\Omega_{t_*}\cup U)$.

Letting the relative capacity of $U$ tend to zero, we get $u=0$ q.e. in $B_R\setminus\Omega_{t_*}$. Together with the vanishing outside $B_R$, \eqref{eq:h10-qe-characterization} gives $u\in H^1_0(\Omega_{t_*})$.

For the a.e.-vanishing characterization, fix $s\in I$ and let $v\in H^1(\R^d)$ vanish a.e. on $\R^d\setminus\Omega_s$. Quasi-continuity gives $v=0$ q.e. in the open exterior of $\Omega_s$. The regular part of $\partial\Omega_s$ is locally covered by $C^1$ graph cylinders, and the one-sided vanishing gives zero trace on the regular graph pieces. As before, the singular part $\Sigma(s)$ has zero capacity. Hence $v=0$ q.e. on $\R^d\setminus\Omega_s$, so \eqref{eq:h10-qe-characterization} gives $v\in H^1_0(\Omega_s)$. The opposite inclusion follows from the zero extension definition of $H^1_0(\Omega_s)$.
\end{proof}

\begin{lem}[Pressure continuity in regular patches]\label{lem:pressure_C1_no_top}
Under the hypotheses of Lemma~\ref{lem:no_top_mosco}, let $t_*\in I$. Then
\[
    p(\cdot,t)\to p(\cdot,t_*)\quad\hbox{strongly in }H^1_{\rm loc}
\]
as $t\to t_*$. Moreover, if $x_*\in\mathcal R$ and $t_*=T(x_*)$, then $\nabla p$ extends continuously from the positive phase $\{(x,t):x\in\Omega_t\}$ to the regular free boundary near $(x_*,t_*)$.
\end{lem}
\begin{proof}
Choose \(R\) so large that \(\Omega_t\Subset B_R\) for \(t\) near \(t_*\).
For each \(t\), let \(u^{(t)}\in H^1_0(\Omega_t)\) be the unique minimizer of
\begin{equation}\label{eq:pressure_variational_energy}
    v\mapsto \frac12\int |\nabla v|^2-\int c(\cdot,t)v
    \qquad\hbox{over }H^1_0(\Omega_t).
\end{equation}
For a.e. \(t\), Lemma~\ref{lem:HS-basic-properties} gives
\[
    \int_{\R^d}\nabla p(\cdot,t)\cdot\nabla\phi\,dx
    =
    \int_{\R^d}c(\cdot,t)\phi\,dx
    \qquad\hbox{for every }\phi\in H^1_0(\Omega_t).
\]
Together with the a.e.-vanishing characterization in Lemma~\ref{lem:no_top_mosco}, this gives \(p(\cdot,t)=u^{(t)}\) in $H^1(\R^d)$ for a.e. $t$.
By the continuity of \(c(\cdot,t)\mathbf 1_{B_R}\) in $H^{-1}(\R^d)$, Lemma~\ref{lem:no_top_mosco}, and Lemma~\ref{lem:mosco-dirichlet-stability}, we have
\[
    u^{(t)}\to u^{(t_*)}\quad\hbox{strongly in }H^1(B_R),
\]
and hence strongly in $H^1_{\rm loc}$. For a.e. $s$, the averages in
\eqref{eq:HS-forward-average-representative} may be computed
using \(u^{(s)}\). The strong continuity of
\(s\mapsto u^{(s)}\) sends these averages to
\(u^{(t)}\) in $H^1_{\rm loc}$ as the averaging scale shrinks.
Hence the upper semicontinuous representative satisfies
\(p(\cdot,t)=u^{(t)}\) as an $H^1$ class for every $t$,
which gives the strong convergence for every \(t\to t_*\).

Fix \(x_*\in\mathcal R\) with \(T(x_*)=t_*\), and increase \(R\), if necessary,
so that a neighborhood of \(x_*\) is contained in \(B_R\).  By
Theorem~\ref{thm:hitting-interface-regularity}, we may choose coordinates near
\((x_*,t_*)\) in which \(\partial\Omega_t\) is a uniformly
\(C^{1,\beta}\) graph, for some \(\beta\in(0,1)\), the graph height is
Lipschitz in \(t\), and the graph gradients converge uniformly as
\(t\to t_*\). Thus the moving domains converge in $C^1$.
Since \(c\) is locally bounded for positive times,
\cite[Cor.~8.36]{GilTru}, applied after flattening the uniformly
\(C^{1,\beta}\) graph patches, gives uniform \(C^{1,\beta}\) bounds up to the
boundary in smaller cylinders for
\[
    -\Delta u^{(t)}=c(\cdot,t)\quad\hbox{in }\Omega_t,\qquad u^{(t)}=0\quad\hbox{on }\partial\Omega_t
\]
Arzelà--Ascoli then implies that every sequence \(t_j\to t_*\) has a subsequence
for which \(u^{(t_j)}\) converges in \(C^1\) up to the flat boundary. The strong
\(H^1\) convergence identifies the subsequential limit with \(u^{(t_*)}\), so
\(\nabla u^{(t)}\to\nabla u^{(t_*)}\) locally uniformly up to the moving regular
boundary. Since the upper semicontinuous representative \(p\) is defined by
forward time averages, the local uniform convergence of \(u^{(t)}\) and \(\nabla
u^{(t)}\) identifies \(p\) with \(u^{(t)}\) pointwise in a regular neighborhood of
\((x_*,t_*)\). Thus \(\nabla p\) has the desired continuous extension.
\end{proof}

\begin{proof}[Proof of Theorem~\ref{thm:no_top_stratum_C1}]
By Theorem~\ref{thm:hitting-interface-regularity}, \(\Sigma_{d-1}\) is
relatively closed.  Since \(T\) is continuous and the slice \(\{T=t_0\}\) is
compact, we may choose \(I\Subset(0,\infty)\) with \(t_0\in I\) and
\(\Sigma_{d-1}\cap\{T\in I\}=\emptyset\).
Let $\mathcal R$ denote the regular set. By Lemma~\ref{lem:pressure_C1_no_top}, $\nabla p$ is continuous in spacetime near every regular boundary point. Thus (see, for instance \cite[Lem.~5.18]{CJK25}), one has
\begin{equation}\label{eq:regular_grad_T_formula}
    \nabla T(x)=-\frac{\nabla p(x,T(x))}{|\nabla p(x,T(x))|^2},
    \qquad x\in\mathcal R\cap\{T\in I\}.
\end{equation}
Hence $T$ is $C^1$ on $\mathcal R\cap\{T\in I\}$, and $\nabla T\neq0$ there.
The relative openness of \(\mathcal R\) in
Theorem~\ref{thm:hitting-interface-regularity} shows that points of $\{T\in I\}$ sufficiently close to a regular point are regular.

Let $q$ be singular with $T(q)\in I$. Since the top stratum is absent,
$q\in\Sigma_{\leq d-2}$. Theorem~\ref{prop:cone_T_differentiability},
applied after translating the base point to $q$, gives differentiability at
$q$ and $\nabla T(q)=0$. If $x_j\to q$ with $T(x_j)\in I$, then each $x_j$
is regular or lower-stratum singular, so $T$ is differentiable at $x_j$
(by \eqref{eq:regular_grad_T_formula} in the regular case and by
Theorem~\ref{prop:cone_T_differentiability} in the singular case).
The stability conclusion of
Theorem~\ref{prop:cone_T_differentiability}, applied at $q$, then gives
$\nabla T(x_j)\to0$. Thus $\nabla T$, set equal to
zero on the singular set, is continuous on $\{T\in I\}$, and
$T\in C^1(\{T\in I\})$.
Finally, Theorem~\ref{prop:cone_T_differentiability} gives $\nabla T=0$
at singular points. Therefore the singular set in $\{T\in I\}$ is exactly
$\{x:\nabla T(x)=0\}$.
\end{proof}

\begin{rem}
Although the theorem starts from a single time slice, the pressure argument uses the resulting exclusion of $\Sigma_{d-1}$ on the whole interval $I$, not merely near the point where one wants $C^1$ regularity. A far-away top-stratum merger can change the scalar speed $|\nabla p|$ at otherwise regular points \cite[Rem.~5.7]{CJK25}. Lower-stratum changes have zero $H^1$ capacity, so they do not change the Dirichlet problem discontinuously.
\end{rem}

\section{Applications to tumor growth}\label{sec:tumor-applications}

In this section we consider the tumor growth model
\eqref{eqn:density}--\eqref{eqn:nutrient} with patch initial data; the pressure of the patch
solution is then a weak solution of \eqref{eq:intro-HS-law} with source $c$, in the sense of
Definition~\ref{def:HS-weak-solution} \cite{JKT21}. We verify that in the physical
dimensions $d\leq3$ this source
satisfies the source hypotheses of Theorem~\ref{thm:HS-hitting-time-lipschitz}, so that the
results of the preceding sections apply to the tumor patch. The time-regularity hypothesis
\eqref{eq:A-source-time} with $m=\infty$ will follow from the pressure estimate
in Lemma~\ref{lem:tumor-pressure-L4}.

\begin{lem}[$L^4$ estimate for the tumor pressure]
\label{lem:tumor-pressure-L4}
Let $(\rho,p,c)$ be a patch solution of
\eqref{eqn:density}--\eqref{eqn:nutrient}, with
$\rho(\cdot,0)=\chi_{\Omega_0^{+}}$.  Suppose that $\Omega_0^{+}$ is bounded
and satisfies a uniform interior ball condition, and that
$c_0\in W^{2,\infty}(\R^d)$ with $\inf_{\R^d}c_0>0$.  Then, for every
$\tau>0$,
\begin{equation}\label{eq:tumor-pressure-L4}
    \nabla p\in L^4(\R^d\times(0,\tau)).
\end{equation}
\end{lem}

\begin{proof}
Fix $\tau>0$ and take the smooth approximations
$(\rho_\gamma,p_\gamma,c_\gamma)$ of \cite[Prop.~3.1]{CJK25}, with
$\nabla p_\gamma(\cdot,0)$ uniformly bounded in $L^2(\R^d)$. The maximum principle,
\eqref{eq:HS-compact-support-barrier}, and local parabolic maximal regularity give
uniform $L^\infty$ bounds for $c_\gamma,p_\gamma$, a uniform
$L^1(B_R\times(0,\tau))$ bound for $(c_\gamma)_t$, and a fixed ball $B_R$
containing $\supp p_\gamma$.

The $L^4$ energy argument in the proof of
\cite[Thm.~3.2]{DavidPerthame} applies with the primitive
$p_\gamma c_\gamma$.  Differentiating this primitive contributes the term
\[
 \iint p_\gamma |(c_\gamma)_t|
 \leq \|p_\gamma\|_{L^\infty}
       \|(c_\gamma)_t\|_{L^1(B_R\times(0,\tau))}.
\]
Thus \cite[proof of Thm.~3.2]{DavidPerthame} bounds
$\nabla p_\gamma$ uniformly in $L^4(\R^d\times(0,\tau))$; the convergence in
\cite[Prop.~3.1]{CJK25} and weak lower semicontinuity give
\eqref{eq:tumor-pressure-L4}.
\end{proof}

\begin{proof}[Proof of Theorem~\ref{thm:tumor-application}]
Fix $\tau>0$. The Duhamel representation of \eqref{eqn:nutrient} is
\begin{equation}\label{eq:tumor-nutrient-duhamel}
 c(\cdot,t)=e^{t\Delta}c_0
 -\int_0^t e^{(t-s)\Delta}(c\rho)(\cdot,s)\,ds.
\end{equation}
The maximum principle and \eqref{eq:tumor-nutrient-duhamel} give
\begin{equation}\label{eq:tumor-basic-nutrient-bounds}
 \underline c_0e^{-\tau}\leq c\leq\|c_0\|_{L^\infty},
 \qquad
 \sup_{0\leq t\leq\tau}\|\nabla c(\cdot,t)\|_{L^\infty}
 \leq \|\nabla c_0\|_{L^\infty}
      +C_d\sqrt{\tau}\,\|c_0\|_{L^\infty}.
\end{equation}
By \eqref{eq:HS-compact-support-comparison}, there  exists $R>0$ such that $p$ and
$\rho$ are supported in $B_R$ for $0\leq t\leq\tau$.  Moreover,
Lemma~\ref{lem:tumor-pressure-L4} gives \eqref{eq:tumor-pressure-L4}.

Set $u:=\partial_t c$.  Differentiating \eqref{eqn:nutrient} in time and using
$\partial_t\rho=\Delta p+c\rho$ and $\Delta c=u+c\rho$, we obtain
\begin{equation}\label{eq:tumor-ct-equation}
 u_t-\Delta u=(p-\rho)u-\Delta(cp)
       +2\nabla c\cdot\nabla p+c\rho(p-c),
 \qquad u(\cdot,0)=\Delta c_0-c_0\rho_0\in L^\infty.
\end{equation}
Time difference quotients in \eqref{eqn:nutrient}, together with
\eqref{eqn:density}, justify \eqref{eq:tumor-ct-equation} in distributions.
Writing its Duhamel formula and moving one derivative from $\Delta(cp)$ to the
heat kernel, the heat-semigroup estimates and
\eqref{eq:tumor-basic-nutrient-bounds} give
\begin{align}
 \|e^{t\Delta}\Delta(cp)\|_{L^\infty}
 &\leq Ct^{-1/2-d/8}\|\nabla(cp)\|_{L^4},\notag\\
 \|e^{t\Delta}(2\nabla c\cdot\nabla p+c\rho(p-c))\|_{L^\infty}
 &\leq Ct^{-d/8}
 \|2\nabla c\cdot\nabla p+c\rho(p-c)\|_{L^4}.
 \label{eq:tumor-ct-semigroup-estimates}
\end{align}
By \eqref{eq:tumor-basic-nutrient-bounds}, the support inclusion, and
Lemma~\ref{lem:tumor-pressure-L4}, the two spatial norms in
\eqref{eq:tumor-ct-semigroup-estimates} belong to $L^4(0,\tau)$. For $d\leq3$
both time kernels belong to $L^1(0,\tau)$ (the first precisely when $d<4$), so
Young's inequality and Gronwall's inequality for $(p-\rho)u$ give
\begin{equation}\label{eq:tumor-ct-L4-Linfty}
    \partial_t c\in L^4\bigl(0,\tau;L^\infty(\R^d)\bigr).
\end{equation}

The bounds in \eqref{eq:tumor-basic-nutrient-bounds} give local uniform
positivity and spatial Lipschitz continuity of $c$, while
\begin{equation}\label{eq:tumor-ct-negative-bound}
 \|(\partial_t c)^-\|_{L^1(0,\tau;L^\infty)}
 \leq \tau^{3/4}\|\partial_t c\|_{L^4(0,\tau;L^\infty)}<\infty.
\end{equation}
Equation~\eqref{eq:tumor-nutrient-duhamel} and the boundedness of $c\rho$
show that $c$ is continuous.  Together with this continuity, equations
\eqref{eq:tumor-basic-nutrient-bounds} and
\eqref{eq:tumor-ct-negative-bound} verify \eqref{eq:A-source} and
\eqref{eq:A-source-time} with $m=\infty$. Under \eqref{eq:A-interior-ball}, the hypotheses
of Theorems~\ref{thm:HS-hitting-time-lipschitz}, \ref{thm:closing-rates},
\ref{prop:cone_T_differentiability} and \ref{thm:no_top_stratum_C1} therefore hold, and
Theorem~\ref{thm:HS-hitting-time-lipschitz} gives the local Lipschitz continuity of $T$ in
$\{T<\infty\}$. This supplies the hypothesis of \cite[Prop.~5.17]{CJK25}, and hence
$\mathcal H^{d-2}(\Sigma(t))=0$ for almost every $t>0$.
\end{proof}

\subsection{A discussion on the size of the tumor region}
\label{subsec:tumor-region-size}
Theorem~\ref{thm:tumor-application},
Theorem~\ref{prop:cone_T_differentiability}, and
Theorem~\ref{thm:no_top_stratum_C1} concern the finite-time geometry of the
interface; they do not determine the geometry of the eventual tumor region
$\Omega_\infty:=\{T<\infty\}= \limsup_{t\to\infty} \Omega_t$. Below we provide a volume estimate for the growth of the tumor region, which in particular states that the region does not stay bounded: this 
contrasts with the zero-nutrient-diffusion model studied in \cite{JKT23}, where $\Omega_t$ may stay bounded with low nutrient level, or $|\Omega_t|$ grows exponentially with high nutrient level. Let us also mention that the asymptotic shape of the $\Omega_t$ remains an intriguing open question:
the tumor may still grow through thin fingers and leave behind nutrient-poor
regions.

\begin{prop}[Lower volume growth]\label{prop:tumor-volume-growth}
Let $(\rho,p,c)$ be a patch solution of
\eqref{eqn:density}--\eqref{eqn:nutrient}, with bounded initial patch
$\Omega_0^{+}$ and $\underline c_0:=\inf_{\R^d}c_0>0$.  There exist
$C_{\rm vol}>0$ and $t_*>0$, depending only on $d$,
$\underline c_0$, and $|\Omega_0^{+}|$, such that
\begin{equation}\label{eq:tumor-volume-growth}
    |\Omega_t|\geq |\Omega_0^{+}|+C_{\rm vol}t^{d/2}, \quad t\geq t_*.
\end{equation}
\end{prop}
\begin{proof}
Let
\begin{equation}\label{eq:tumor-mass-functions}
    g(t)=\int_{\Omega_t}c(x,t)dx, \qquad G(t)=\int_{0}^tg(s)ds.
\end{equation}
Integrating \eqref{eqn:density} over time and space and using
$\rho(\cdot,t)=\chi_{\Omega_t}$ gives
\begin{equation}\label{eq:tumor-volume-mass-identity}
    |\Omega_t|-|\Omega_0^{+}|= G(t),
\end{equation}
so it suffices to bound $G(t)+|\Omega_0^{+}|$ from below. Fix
$t_0\geq0$ and use \eqref{eq:tumor-nutrient-duhamel} with
$e^{t\Delta}c_0\geq\underline c_0$. For $0<s<t_0<t$,

\begin{equation}\label{eq:tumor-old-source-heat-bound}
\|e^{(t-s)\Delta}(c\rho)(\cdot,s)\|_{L^{\infty}(\R^d)}
\leq (4\pi (t-t_0))^{-d/2}
  \|c\rho(\cdot,s)\|_{L^{1}(\R^d)}\\
=(4\pi (t-t_0))^{-d/2}g(s).
\end{equation}
and, for $t_0<s<t$,
\begin{equation}\label{eq:tumor-recent-source-L1-bound}
\int_{\Omega_{t_0}}e^{(t-s)\Delta}(c\rho)(x,s)\,dx
\leq \|c\rho(\cdot,s)\|_{L^1(\R^d)}=g(s).
\end{equation}
Integrating \eqref{eq:tumor-nutrient-duhamel} over $\Omega_{t_0}$, using
\eqref{eq:tumor-old-source-heat-bound},
\eqref{eq:tumor-recent-source-L1-bound}, and the monotonicity $\partial_t\rho\geq0$,
gives
\begin{equation}\label{eq:replenishment}
g(t)\geq \int_{\Omega_{t_0}}c(x,t)\,dx \geq |\Omega_{t_0}|\underline c_0
 -(4\pi(t-t_0))^{-d/2}G(t_0)|\Omega_{t_0}|
 -(G(t)-G(t_0)).
\end{equation}
Set
\begin{equation}\label{eq:tumor-replenishment-scale}
\tau(t_0):=\max \left(1, c_2G(t_0)^{2/d} \right), \qquad c_2:=(4\pi)^{-1}(4/\underline c_0)^{2/d},
\end{equation}
so that
$(4\pi\tau(t_0))^{-d/2}G(t_0)\leq \underline c_0/4$. Now
\eqref{eq:replenishment} gives
\begin{equation}\label{eq:tumor-g-lower-window}
g(t)\geq \frac{3\underline c_0}{4}|\Omega_{t_0}|
 -(G(t)-G(t_0)) \quad \hbox{ for } t\in t_0 + (\tau(t_0),2\tau(t_0)).
\end{equation}
Integrating \eqref{eq:tumor-g-lower-window} over its time interval,
the monotonicity of $G$ gives
\begin{equation}\label{eq:tumor-G-window-growth}
(G(t_0+2\tau(t_0))-G(t_0))(1+\tau(t_0))
\geq \frac{3\underline c_0 \tau(t_0)}{4}|\Omega_{t_0}|.
\end{equation}
Since $\tau(t_0)\geq1$, equations \eqref{eq:tumor-G-window-growth} and
\eqref{eq:tumor-volume-mass-identity} imply, with
$a:=3\underline c_0/8$,
\begin{equation}\label{eq:tumor-G-geometric-step}
 G(t_0+2\tau(t_0))+|\Omega_0^{+}|
 \geq (1+a)\bigl(G(t_0)+|\Omega_0^{+}|\bigr).
\end{equation}
Set $s_0=0$, $s_{j+1}=s_j+2\tau(s_j)$, and
$H_j:=G(s_j)+|\Omega_0^{+}|$. Then
$H_{j+1}\geq(1+a)H_j$ by \eqref{eq:tumor-G-geometric-step}, while
\eqref{eq:tumor-replenishment-scale} gives
$s_{j+1}-s_j\leq CH_j^{2/d}$. Summing and using the geometric growth yields
$s_{j+1}\leq CH_j^{2/d}$. Hence, for $t\in[s_j,s_{j+1}]$, monotonicity gives
$G(t)+|\Omega_0^{+}|\geq ct^{d/2}$; together with
\eqref{eq:tumor-volume-mass-identity}, this proves
\eqref{eq:tumor-volume-growth} after increasing $t_*$.
\end{proof}
\begin{rem}
The exponent
$d/2$ in Proposition~\ref{prop:tumor-volume-growth} agrees with the radial
scaling in the nutrient-limited regime $c_0<1$. Among the nutrient assumptions in
\eqref{eq:tumor-initial-nutrient-assumptions}, the volume estimate in
Proposition~\ref{prop:tumor-volume-growth} uses only that 
$\inf_{\R^d}c_0>0$, which is a natural assumption when nutrient is supplied from infinity at
a positive rate.    
\end{rem}

\section{Application to the injection problem}\label{sec:injection}

In this final section we briefly discuss the potential extension of our analysis to the {\it injection problem}, that is, \eqref{eq:intro-HS-law} with $c\equiv 0$ in $\R^d\setminus B_r(0)$, with fixed boundary data $p=1$ on $\partial B_r(0)$. While this article focuses on the effect of the source term in \eqref{eq:intro-HS-law}, the injection problem merits discussion, as it is the classical form of \eqref{eq:intro-HS-law} and has an extensive literature. In addition to the regularity results recalled in the introduction, the large-time behavior of the flow is known to be simpler than in the general setting of the previous sections: the free boundary becomes smooth in finite time \cite{Kim06,kim06long}, and, as $t\to\infty$, the set $\{p>0\}$ exhausts $\R^d$ and the free boundary approaches a spherical profile in the normalized scale \cite{qv01}. In particular, for the injection problem the set $\{0<T(x)<\infty\}$ is exactly $\R^d \setminus \Omega_0^{+}$, whereas the analogous question for the tumor growth model remains open (Subsection~\ref{subsec:tumor-region-size}). Even in this well-studied setting, however, a description of the {\it space-time} profile of the free boundary valid for {\it all} times would be new.

We expect that our results extend to the injection problem without major challenges. Indeed, much of our analysis would be considerably simpler in this setting, owing partly to the availability of the comparison principle and partly to the uniform positivity and homogeneity of the fixed boundary data. Rather than carrying out proofs, we confine ourselves to indicating the parts of our analysis where modifications are necessary; the full verification remains open.

$\circ$ {\it Theorem~\ref{thm:HS-hitting-time-lipschitz}.} The Lipschitz regularity of the hitting time $T$ follows from the barrier argument of \cite{caffarelli1978} (which is a much simpler version of the proof of Proposition~\ref{lem:cone_gradient_intrinsic_window}), with no assumptions on the initial data, but only at a positive distance from the initial patch. For nondegenerate initial data, Theorem~\ref{thm:HS-hitting-time-lipschitz} as stated follows from the comparison estimate below (see \cite[proof of Theorem~3.1]{kim03}): for any $\e>0$, we have
$$
\sup_{|x-y|\leq \e } p(y,t) \leq (1+C\e)p(x,(1+C\e)t+C\e),
$$
where $C>0$ is chosen so that the ordering holds at $t=0$. Formally, this states that $V_{x,t} \geq \frac{1}{C (t+1)}$, and thus $|T(x)-T(y)|\leq C(t+1)|x-y|$.

$\circ$ {\it Theorem}~\ref{prop:HS-pressure-hopf-lax}. Here we use the fact that $p$ is smooth near the fixed boundary, and that the boundary gradient $|Dp|$ is largest at the initial time, due to the monotone expansion of the support of $p$. More precisely, the function $\bar{p}$ in \eqref{eq:HS-Hopf-Lax-subsolution} remains a weak subsolution of the injection problem, and hence lies below $p$, in $(\R^d\setminus B_r(0))\times(t_0,t_1]$, provided we take $\Lambda(t)\equiv 0$ and $M=\sup_{x\in K} |Dp_0|(x)$ in the definition of $A(t)$. Formally, we then have $\partial_t p - |Dp|^2 +C_M p \geq 0$.

$\circ$ {\it Proposition~\ref{prop:pressure_lower_bd}.}  Using Harnack chains from the fixed boundary, one obtains a positive lower bound on $p$ at any point away from the zero set. In the case $k=0$, this implies that $p$ is uniformly positive on $\partial B_{R_0}$, and one can then use the fundamental solution as a subbarrier. In the case $1 \leq k \leq d-2$, it implies that $p$ is uniformly positive on $\partial B_{R_0}\cap \{ |x'| \geq \beta |x''| \}$ for fixed $\beta > \alpha$; standard comparison and boundary Harnack arguments, based on the harmonic functions from \eqref{eq:cone_harmonics}, then yield an analogue of Lemma~\ref{lem:cone_subbarrier}: $p(\rho, \varphi) \gtrsim \rho^{\nu_\alpha}\Phi_\alpha(\varphi)$ on $U(V,\alpha, C^{-1}R_0,Cr_0)$ for $C=C(d,k,\alpha,\beta)$.

\appendix

\section{Comparison principle for weak solutions}\label{app:distributional-comparison}

For \(\tau>0\), write \(Q_\tau:=\R^d\times(0,\tau)\).
The results of this appendix require only a nonnegative source
\(c\in L^\infty(Q_\tau)\).

\begin{defn}[Relaxed weak subsolutions and supersolutions]
\label{def:app-HS-relaxed-subsuper}
Let \(c\in L^\infty(Q_\tau)\) be nonnegative.  A pair
\((\rho,p)\) is a relaxed weak subsolution to \eqref{eq:intro-HS-law}, with
source \(c\), on \([0,\tau)\) if
\begin{equation}\label{eq:app-HS-relaxed-graph}
    0\leq\rho\leq1,\qquad p\geq0,\qquad p(1-\rho)=0
    \qquad\hbox{a.e. in }Q_\tau,
\end{equation}
\begin{equation}\label{eq:app-HS-relaxed-regularity}
    p\in L^2_{\operatorname{loc}}([0,\tau);H^1(\R^d)),
    \qquad
    \rho\in L^\infty_{\operatorname{loc}}([0,\tau);L^1(\R^d)),
\end{equation}
and, for the chosen \(L^1\)-valued representative of \(\rho\),
\begin{equation}\label{eq:app-HS-relaxed-time-traces}
    \rho\in C([0,\tau);L^1(\R^d)),
\end{equation}
and
\begin{equation}\label{eq:app-HS-relaxed-subsolution-inequality}
    \partial_t\rho-\Delta p\leq c\rho
    \qquad\hbox{in }\mathcal D'(Q_\tau).
\end{equation}
A relaxed weak supersolution is defined by reversing the inequality in
\eqref{eq:app-HS-relaxed-subsolution-inequality}.
\end{defn}

\begin{thm}[Relaxed comparison principle]\label{thm:app-HS-relaxed-comparison}
Let \(0<\tau<\infty\), and let
\(c^-,c^+\in L^\infty(Q_\tau)\) satisfy \(0\leq c^-\leq c^+\).  Let
\((\rho^-,p^-)\) and \((\rho^+,p^+)\) be, respectively, a relaxed weak
subsolution with source \(c^-\) and a relaxed weak supersolution with source
\(c^+\) on \([0,\tau)\).  If
\begin{equation}\label{eq:app-HS-relaxed-initial-order}
    \rho^-(\cdot,0)\leq\rho^+(\cdot,0)
    \qquad\hbox{a.e. in }\R^d,
\end{equation}
then
\begin{equation}\label{eq:app-HS-relaxed-density-comparison}
    \rho^-\leq\rho^+
    \qquad\hbox{a.e. in }Q_\tau .
\end{equation}
Consequently,
\begin{equation}\label{eq:app-HS-relaxed-pressure-comparison}
    p^-\leq p^+
    \qquad\hbox{a.e. in }Q_\tau .
\end{equation}
\end{thm}

The proof of the density comparison in
\eqref{eq:app-HS-relaxed-density-comparison} proceeds by duality.  Writing
\(u:=\rho^- -\rho^+\) and \(v:=p^- -p^+\) for the differences of two
density--pressure pairs, the
difference of the sub- and supersolution inequalities involves the linear
combination \(A\partial_t\psi+B\Delta\psi\), where the coefficients
\(A,B\in[0,1]\) depend on \((u,v)\) and are therefore only measurable.
Closing the estimate requires testing against \(\psi\) that approximately
solves the degenerate backward adjoint
\[
    A\partial_t\psi+B\Delta\psi=-AG
    \qquad\hbox{in }D\times(a,b),
\]
for a prescribed nonnegative source \(G\), with
\(\norm{\psi}_{L^\infty}\) controlled independently of the coefficients.  The
next lemma constructs such an approximating sequence against the fixed weight
that appears in the comparison proof.

\begin{lem}[Approximate adjoint test functions]\label{lem:app-HS-dual-tests}
Let \(D\subset\R^d\) be a smooth bounded domain, let \(0\leq a<b\), and
set \(Q:=D\times(a,b)\).  Let \(A,B\in L^\infty(Q;[0,1])\), and let
\(G\in C_c^\infty(Q)\) be nonnegative.  For every \(r\in L^2(Q)\), there are
nonnegative
\(\psi_j\in C^\infty(\overline D\times[a,b])\) vanishing on the parabolic
boundary \((\partial D\times[a,b])\cup(D\times\{b\})\), with the uniform
\(L^\infty\) bound
\begin{equation}\label{eq:app-HS-dual-test-bound}
    0\leq\psi_j\leq (b-a)\norm{G}_{L^\infty(Q)}
    \qquad\hbox{in }Q,
\end{equation}
that approximately solve the backward adjoint
\(A\partial_t\psi+B\Delta\psi=-AG\) in the sense that
\begin{equation}\label{eq:app-HS-dual-test-limit}
    \lim_{j\to\infty}
    \int_Q r\left(A\partial_t\psi_j+B\Delta\psi_j+AG\right)\,dxdt
    =0.
\end{equation}
If, in addition, \(D=B_R\), \(G\) is supported in
\(B_{R_0}\times(a,b)\), and \(R\geq2R_0\), then the sequence may be chosen
so that
\begin{equation}\label{eq:app-HS-dual-normal-bound}
    -C R^{1-d}\leq \partial_\nu\psi_j\leq0
    \qquad\hbox{on }\partial B_R\times(a,b),
\end{equation}
where \(C\) depends only on \(d,R_0,b-a\), and \(G\), but not on \(j\) or
\(R\).
\end{lem}

\begin{proof}

The construction is the regularized-adjoint, weighted-duality argument of
\cite[\S\S3,5]{PQV14} (see also
\cite[Prop.~5.1]{GKM22} and \cite[App.~A]{JKT23}); we recall only the properties
used below.  Approximate \(A,B\) by smooth \(A_j,B_j\) valued in \([j^{-1},1]\),
and let \(\psi_j\) solve the nondegenerate backward equation
\begin{equation}\label{eq:app-HS-regularized-adjoint}
    \partial_t\psi_j+\tfrac{B_j}{A_j}\Delta\psi_j=-G
    \quad\hbox{in }Q,\qquad
    \psi_j=0\hbox{ on }(\partial D\times[a,b])\cup(D\times\{b\}).
\end{equation}
By the maximum principle \(\psi_j\in C^\infty(\overline D\times[a,b])\) is
nonnegative and obeys the uniform bound \eqref{eq:app-HS-dual-test-bound}; and if
the approximation is taken with
\(\sqrt j\,(\norm{r(A_j-A)}_{L^2(Q)}+\norm{r(B_j-B)}_{L^2(Q)})\to0\), a weighted
energy estimate for \(\Delta\psi_j\) gives the error limit
\eqref{eq:app-HS-dual-test-limit}.

For \eqref{eq:app-HS-dual-normal-bound}, assume \(D=B_R\),
\(\supp G\subset B_{R_0}\times(a,b)\), and \(R\geq2R_0\). Let \(h_R\) be radial
and harmonic on \(B_R\setminus B_{R_0}\), with values
\((b-a)\norm{G}_{L^\infty(Q)}\) and \(0\) on the inner and outer boundaries.
Its explicit formula gives \(-CR^{1-d}\leq\partial_\nu h_R<0\) on
\(\partial B_R\). Since \(G=0\) outside \(B_{R_0}\), comparison gives
\(0\leq\psi_j\leq h_R\), hence
\(\partial_\nu h_R\leq\partial_\nu\psi_j\leq0\), as required.

\end{proof}

\begin{proof}[Proof of the density comparison in \eqref{eq:app-HS-relaxed-density-comparison}]
Set
\[
    u:=\rho^- -\rho^+,\qquad v:=p^- -p^+ .
\]
The goal is \(u\leq0\) a.e.  The short-time step below is controlled by
\begin{equation}\label{eq:app-HS-M-def}
    M:=\norm{c^-}_{L^\infty(Q_\tau)}.
\end{equation}
The graph constraints imply that \(u\) and \(v\) have the same sign:
if \(u>0\), then \(\rho^+<1\), so \(p^+=0\) and \(v=p^-\geq0\); if \(u<0\),
then \(\rho^-<1\), so \(p^-=0\) and \(v=-p^+\leq0\). Thus, for
\(r:=u+v\), set
\begin{equation}\label{eq:app-HS-AB-def}
    A:=\frac{u}{r},\qquad B:=\frac{v}{r}
    \qquad\hbox{on }\{r\neq0\},
\end{equation}
extended by \(A=B=0\) on \(\{r=0\}\), satisfy
\begin{equation}\label{eq:app-HS-AB-properties}
    0\leq A,B\leq1,\qquad u=rA,\qquad v=rB .
    \end{equation}

Fix \(0\leq a<b<\tau\), assume \(u(\cdot,a)\leq0\) a.e. in \(\R^d\), and
  let \(R>0\).  The continuity in
\eqref{eq:app-HS-relaxed-time-traces} allows tests touching \(t=a\) in
\eqref{eq:app-HS-relaxed-subsolution-inequality}.  After subtracting the
supersolution inequality for \(p^+\) from the subsolution inequality for
\(p^-\), using \(u(\cdot,a)\leq0\), and bounding
\(c^+\rho^+\geq c^-\rho^+\), we obtain
\begin{equation}\label{eq:app-HS-difference-weak-ineq}
    \int_a^b\!\!\int_{B_R}
    \nabla\phi\cdot\nabla v
    -u\,\partial_t\phi
    -c^-u\phi\,dxdt
    \leq0
\end{equation}
for every nonnegative \(\phi\in C_c^\infty(B_R\times[a,b))\).
The same inequality holds for nonnegative
\(\psi\in C^\infty(\overline{B_R}\times[a,b])\) vanishing on
\((\partial B_R\times[a,b])\cup(B_R\times\{b\})\): spatial and terminal-time
cutoffs approximate \(\psi\) in
\(L^2(a,b;H^1_0(B_R))\cap H^1(a,b;L^2(B_R))\).

For such a \(\psi\), integrating the gradient term in
\eqref{eq:app-HS-difference-weak-ineq} by parts in space gives
\begin{equation}\label{eq:app-HS-AB-ineq}
    \int_a^b\!\!\int_{B_R}
    r\left(A\partial_t\psi+B\Delta\psi\right)\,dxdt
    \geq
    \mathcal B_R[\psi]
    -\int_a^b\!\!\int_{B_R} c^-u\psi\,dxdt,
\end{equation}
where
\begin{equation}\label{eq:app-HS-boundary-term}
    \mathcal B_R[\psi]
    :=
    \int_a^b\!\!\int_{\partial B_R}
    v\,\partial_\nu\psi\,d\sigma dt .
\end{equation}

Let \(G\in C_c^\infty(\R^d\times(a,b))\) be nonnegative, and choose
\(R_0>0\) with \(\supp G\subset B_{R_0}\times(a,b)\).  For
\(R\geq2R_0\), let \(\psi_{j,R}\) be the sequence from
Lemma~\ref{lem:app-HS-dual-tests} on \(B_R\times(a,b)\).  By
\eqref{eq:app-HS-dual-normal-bound},
\begin{equation}\label{eq:app-HS-boundary-error}
\begin{aligned}
    \abs{\mathcal B_R[\psi_{j,R}]}
    &\leq
    C R^{1-d}\int_a^b\!\!\int_{\partial B_R}\abs{v}\,d\sigma dt  \\
    &\leq
    C R^{-(d-1)/2}
    \left(\int_a^b\!\!\int_{\partial B_R}\abs{v}^2\,d\sigma dt\right)^{1/2}
    =:\delta_R ,
\end{aligned}
\end{equation}
The constant \(C\) is independent of \(j\) and \(R\).  Since
\(v\in L^2(a,b;L^2(\R^d))\), the coarea formula gives
\(\liminf_{R\to\infty}\delta_R=0\).

Using \eqref{eq:app-HS-dual-test-limit}, \eqref{eq:app-HS-AB-ineq},
\eqref{eq:app-HS-boundary-error}, \(rA=u\), \(u\leq u_+\), and
\eqref{eq:app-HS-dual-test-bound}, we obtain, for every \(R\geq2R_0\),
\begin{equation}\label{eq:app-HS-local-positive-bound}
\begin{aligned}
    \int_a^b\!\!\int_{\R^d} uG\,dxdt
    &=
    -\lim_{j\to\infty}
    \int_a^b\!\!\int_{B_R}
    r\left(A\partial_t\psi_{j,R}+B\Delta\psi_{j,R}\right)\,dxdt \\
    &\leq
    M(b-a)\norm{G}_{L^\infty(\R^d\times(a,b))}
    \int_a^b\!\!\int_{B_R} u_+\,dxdt
    +\delta_R .
\end{aligned}
\end{equation}
Taking \(R\to\infty\) along radii for which \(\delta_R\to0\) yields
\begin{equation}\label{eq:app-HS-global-positive-bound}
    \int_a^b\!\!\int_{\R^d} uG\,dxdt
    \leq
    M(b-a)\norm{G}_{L^\infty(\R^d\times(a,b))}
    \int_a^b\!\!\int_{\R^d} u_+\,dxdt .
\end{equation}
Taking the supremum in \eqref{eq:app-HS-global-positive-bound} over smooth
\(0\leq G\leq1\) compactly supported in \(\R^d\times(a,b)\) gives
\[
    \int_a^b\!\!\int_{\R^d} u_+\,dxdt
    \leq
    M(b-a)\int_a^b\!\!\int_{\R^d} u_+\,dxdt.
\]
Hence \(u\leq0\) a.e. in \(\R^d\times(a,b)\) whenever \(M(b-a)<1\).

Finally, partition \([0,\tau']\), \(\tau'<\tau\), into intervals
with \(M(b-a)<1\). The initial ordering starts the iteration, and
\eqref{eq:app-HS-relaxed-time-traces} passes the ordering to each right
endpoint. Thus \(u\leq0\) on \(\R^d\times[0,\tau']\); letting
\(\tau'\uparrow\tau\) proves \eqref{eq:app-HS-relaxed-density-comparison}.
\end{proof}
\begin{lem}[Elliptic inequality on the saturated phase]\label{lem:app-HS-elliptic-positive-phase}
Let \((\rho,p)\) be a relaxed weak subsolution (resp.\ supersolution) to
\eqref{eq:intro-HS-law} on \([0,\tau)\), with source
\(c\in L^\infty(Q_\tau)\).  If
\(\eta\in L^2(0,\tau;H^1(\R^d))\) is nonnegative, compactly supported in
\(\R^d\times(0,\tau)\), and satisfies \(\eta(1-\rho)=0\), then
\begin{equation}\label{eq:app-HS-elliptic-positive-phase}
    \int_0^\tau\!\!\int_{\R^d}\nabla\eta\cdot\nabla p\,dxdt
    \leq\;(\geq)
    \int_0^\tau\!\!\int_{\R^d}\eta\,c\,dxdt.
\end{equation}
In particular, if \((\rho,p)\) is relaxed weak solution, then, for
every
\(\eta\in L^2(0,\tau;H^1(\R^d))\) compactly supported in
\(\R^d\times(0,\tau)\) and satisfying \(\eta(1-\rho)=0\),
\begin{equation}\label{eq:app-HS-elliptic-positive-phase-identity}
    \int_0^\tau\!\!\int_{\R^d}\nabla\eta\cdot\nabla p\,dxdt
    =
    \int_0^\tau\!\!\int_{\R^d}\eta\,c\,dxdt .
\end{equation}
\end{lem}

\begin{proof}
We first prove the two one-sided statements, so assume \(\eta\geq0\).
Extend \(\eta\) by zero outside \((0,\tau)\), and set
\[
    \eta^\varepsilon(x,t)
    :=\frac1\varepsilon\int_t^{t+\varepsilon}\eta(x,s)\,ds,
    \qquad
    \eta_\varepsilon(x,t)
    :=\frac1\varepsilon\int_{t-\varepsilon}^t\eta(x,s)\,ds .
\]
For small \(\varepsilon\), these averages belong to
\(L^2_tH^1_x\cap H^1_tL^2_x\) and are compactly supported; mollification
justifies their use as test functions. If \((\rho,p)\) is a relaxed weak subsolution, testing
\eqref{eq:app-HS-relaxed-subsolution-inequality} with \(\eta^\varepsilon\) gives
\[
    \int\nabla\eta^\varepsilon\cdot\nabla p-\eta^\varepsilon c\rho\,dxdt
    \leq
    \int\rho\,\partial_t\eta^\varepsilon\,dxdt .
\]
Pointwise,
\[
    \rho(x,t)\partial_t\eta^\varepsilon(x,t)
    =
    \rho(x,t)\frac{\eta(x,t+\varepsilon)-\eta(x,t)}{\varepsilon}
    \leq
    \partial_t\eta^\varepsilon(x,t).
\]
Indeed, equality holds where \(\eta>0\), since then \(\rho=1\), while
on \(\{\eta=0\}\) the last derivative is nonnegative. Its space-time integral
vanishes, hence
\[
    \int\nabla\eta^\varepsilon\cdot\nabla p-\eta^\varepsilon c\rho\,dxdt
    \leq0 .
\]
Letting \(\varepsilon\downarrow0\) proves the subsolution case, using
\(\eta=\rho\eta\). The backward average gives the reverse inequality for
supersolutions, since its derivative is nonpositive on \(\{\eta=0\}\).
Applying the two inequalities to \(\eta_+\) and \(\eta_-\) proves
\eqref{eq:app-HS-elliptic-positive-phase-identity}.
\end{proof}

\begin{proof}[Proof of the pressure comparison in \eqref{eq:app-HS-relaxed-pressure-comparison}]
The density comparison in \eqref{eq:app-HS-relaxed-density-comparison} gives
\(\rho^-\leq\rho^+\).  Let \(\omega\in C_c^\infty(0,\tau)\) be nonnegative and
let \(\chi_R\in C_c^\infty(\R^d)\) satisfy \(0\leq\chi_R\leq1\),
\(\chi_R=1\) on \(B_R\), \(\chi_R=0\) outside \(B_{2R}\), and
\(\abs{\nabla\chi_R}\leq C/R\).  Set
\[
    w:=(p^- -p^+)_+,
    \qquad
    \eta_R(x,t):=\omega(t)\chi_R(x)^2w(x,t).
\]
This is an admissible nonnegative test for both pairs: on
\(\{\eta_R>0\}\), \(p^->0\) gives \(\rho^-=1\), and density comparison gives
\(\rho^+=1\).
Subtracting the supersolution case of
\eqref{eq:app-HS-elliptic-positive-phase} for \(p^+\) from the subsolution
case for \(p^-\) gives
\[
    \int_0^\tau \omega(t)\int_{\R^d}
    \chi_R^2\abs{\nabla w}^2
    +2\chi_Rw\nabla\chi_R\cdot\nabla w\,dxdt
    \leq
    \int_0^\tau\!\!\int_{\R^d}\eta_R(c^- -c^+)\,dxdt
    \leq0 .
\]
Therefore
\[
    \int_0^\tau \omega(t)\int_{\R^d}
    \chi_R^2\abs{\nabla w}^2\,dxdt
    \leq
    4\int_0^\tau \omega(t)\int_{\R^d}w^2\abs{\nabla\chi_R}^2\,dxdt .
\]
Letting \(R\to\infty\), using \(w\in L^2_tH^1_x\) and dominated
convergence, gives
\[
    \int_0^\tau \omega(t)\int_{\R^d}\abs{\nabla w}^2\,dxdt=0.
\]
Thus, for a.e.\ \(t\), \(w(\cdot,t)\) is constant and belongs to
\(L^2(\R^d)\), so it vanishes, proving
\eqref{eq:app-HS-relaxed-pressure-comparison}.
\end{proof}

\section{Global solutions and blowup analysis of the obstacle problem}\label{app:global-obstacle}

To identify the blowup limit in more general settings permitted by Lemma~\ref{lem:obst-blowup}, we need various tools from the classification of global solutions to the obstacle problem. We introduce a few of these tools in the next lemma.

\begin{lem}\label{lem:ellipsoid_blowup}
Let $d\geq2$, and let $v$ be a convex global solution of the
obstacle problem with source $f\equiv1$ in $\R^d$ and with at most
quadratic growth. Set $E:=\{v=0\}$. Then:
\begin{enumerate}[label=(\roman*)]
    \item\label{item:ellipsoid_2_pt_criterion} If $E$ is bounded and contains two points, then $E$ is an ellipsoid.
    \item\label{item:quadratic_interior_criterion} If $E$ contains two points, and there exists a positive-definite 2-homogeneous polynomial $q$ with $\Delta q = 1$ such that $v\leq q$, then $E$ has nonempty interior.
    \item\label{item:ellipsoid_unique} Assume $q$ is as in item~\ref{item:quadratic_interior_criterion}, and that $E\subset \overline{B_1}, E\cap \partial B_1\neq \emptyset$. Then $E$ is an ellipsoid uniquely determined by $q$, and $v$ is uniquely determined by $q$ as well.
\end{enumerate}
\end{lem}
\begin{proof}

If $E$ has zero measure, then Liouville gives that $v$ is a nonnegative quadratic polynomial, in which case the assumptions that $E$ is bounded and has at least two points contradict. Thus, $E$ has positive measure, and since it is convex, it has nonempty interior. It is known by the classification of global solutions to the obstacle problem that an $E$ with these properties must be an ellipsoid \cite{FriedmanSakai1986, EFW25}.

Suppose item~\ref{item:quadratic_interior_criterion} fails. Since $v$ is convex, $E$ is convex with empty interior, so it has zero measure. It follows that $q - v$ is nonnegative harmonic, and thus constant by Liouville. Since $v(0) = q(0) = 0$, $q\equiv v$. But this contradicts that $v$ vanishes at another point while $q$ vanishes only at 0. Thus, $E$ must have nonempty interior.

We now move on to item~\ref{item:ellipsoid_unique}. The assumptions imply that $E$ is bounded and contains 0 and a point on the unit sphere. Then item~\ref{item:ellipsoid_2_pt_criterion} gives that $E$ is an ellipsoid. Let $Q_E$ denote the Newtonian potential of $E$,
\[ Q_E(x) = \begin{cases}
    C(d)\int_E \log |x-y|dy & d=2\\ C(d)\int_E |x-y|^{2-d}dy & d>2
\end{cases}   \]
    so that $\Delta Q_E = -\chi_E$. After rotating so that the coordinate axes align with the principal axes of $E$, $Q_E$ is given by a quadratic polynomial $b_0 - \sum b_i(x_i-y_i)^2$ inside $E$, where $y$ is the center of $E$, and the $b_i$ are positive. If $d\geq 3$, then $Q_E\to 0$ as $|x|\to\infty$, so the harmonic function $q-v+Q_E$ is bounded below and hence constant. If $d=2$, then $H:=q-v+Q_E$ is harmonic with at most quadratic growth, hence is a harmonic polynomial of degree at most two. Since $q-v=H-Q_E\geq0$ and $Q_E=O(\log(2+|x|))$, the leading homogeneous part of $H$ cannot change sign; thus $H$ is constant. Thus, for all $d\geq 2$, $v=q+Q_E-k$ for some constant $k$.

    Now, $v$ vanishes inside $E$, so $q = k-Q_E$ inside $E$. Expanding $q=k-Q_E$ in the coordinates above gives $q(x)=k-b_0+\sum_i b_i(x_i-y_i)^2$. Since the $2$-homogeneous polynomial $q$ has no linear term and $b_i>0$, we obtain $y=0$, and hence $q=\sum_i b_i x_i^2$ in $E$. This is after rotating the principal axes of $E$ to the coordinate axes. It follows that if in original coordinates we had $q = x\cdot Ax$, then the principal axes of $E$ are the eigenvectors of $A$, and the $b_i$ the corresponding eigenvalues.
    
    Now, define a map $F$ by $F(a_1,\dots, a_d) = (b_1,\dots, b_d)$, such that $-\sum b_i x_i^2$ is the quadratic associated with the Newtonian potential of the ellipsoid $E = \{ x : \sum \frac{x_i^2}{a_i^2} \leq 1 \}$, up to additive constant. In the proof of Lemma A in \cite{KarpEllipsoid}, it is shown that $F$ is 0-homogeneous and maps $\{ a_i \in (0,\infty)^d : \sum a_i = 1\}$ bijectively onto $\{ b_i\in (0,\infty)^d : \sum b_i = \frac{1}{2}\}$. Thus, the axis ratios of $E$ are determined by the eigenvalues of $D^2 q$. The assumption that $E\subset \overline{B_1}$ and $E\cap \partial B_1\neq \emptyset$ fixes the scaling of $E$, so $E$ is uniquely determined by $q$.

Finally, if $v'$ is another global solution to the obstacle problem with at most quadratic growth and vanishing exactly on $E$, then $v - v'$ is harmonic with quadratic growth at infinity, and hence a quadratic polynomial by Liouville. But $v - v'$ vanishes identically on $E$, so it vanishes everywhere.

\end{proof}

\renewcommand{\bibliofont}{\small\setlength{\baselineskip}{10pt}}
\bibliographystyle{alpha}
\bibliography{references}

\end{document}